\documentclass[12pt,oneside,titlepage,twoside]{scrbook} 
\usepackage[latin1]{inputenc}
\usepackage{geometry}
\usepackage{amsmath}
\usepackage{amsfonts}
\usepackage{amssymb}
\usepackage{dsfont}
\usepackage{amsthm}
\usepackage{multirow}
\usepackage{dcolumn}
\usepackage{comment}
\usepackage{hyperref}
\usepackage[all,knot]{xy}
\usepackage{rotating}
\xyoption{arc}
\usepackage{graphicx}
\usepackage[titletoc]{appendix}

\usepackage{a4wide}
\xyoption{arc}
\usepackage{mathrsfs}
\usepackage{bbm}
\input xypic
\usepackage{color}

 \theoremstyle{plain}
 \newtheorem{thm}{Theorem}[section]
  \newtheorem*{nonumberthm}{Theorem}
 \newtheorem{lem}[thm]{Lemma}
 \newtheorem{prop}[thm]{Proposition}

 \newtheorem{cor}[thm]{Corollary}

 \newtheorem{example}[thm]{Example}
 \theoremstyle{definition}
 \newtheorem{defn}[thm]{Definition}
 \newtheorem{defthm}[thm]{Definition/Theorem}
 \newtheorem{rem}[thm]{Remark}

  \newtheorem{conj}{Conjecture}
  \newtheorem*{res}{Result}
    \newtheorem*{ansatz}{Ansatz}

\newcommand{\bth}{\begin{thm}}
\renewcommand{\eth}{\end{thm}}
\newcommand{\bpr}{\begin{prop}}
\newcommand{\epr}{\end{prop}}
\newcommand{\ble}{\begin{lem}}
\newcommand{\ele}{\end{lem}}
\newcommand{\bco}{\begin{cor}}
\newcommand{\eco}{\end{cor}}
\newcommand{\bde}{\begin{defn}}
\newcommand{\ede}{\end{defn}}
\newcommand{\bex}{\begin{example}}
\newcommand{\eex}{\end{example}}
\newcommand{\bre}{\begin{rem}}
\newcommand{\ere}{\end{rem}}
\newcommand{\bcj}{\begin{conj}}
\newcommand{\ecj}{\end{conj}}

\newcommand{\beq}{\begin{equation}}
\newcommand{\eeq}{\end{equation}}
\newcommand{\ve}{{\varepsilon}}
\newcommand{\ot}{{\otimes}}
\newcommand{\op}{{\oplus}}
\newcommand{\lb}{\label}
\newcommand{\bpf}{\begin{proof}}
\newcommand{\epf}{\end{proof}}

\newcommand{\Hom}{\mathrm{Hom}}

\newcommand{\uu}{{\mathfrak u}}

\newcommand{\s}{{\mathfrak s}}

\newcommand{\E}{{\cal E}}

\newcommand{\T}{{\cal T}}
\newcommand{\C}{{\cal C}}
\newcommand{\Z}{{\cal Z}}
\newcommand{\D}{{\cal D}}

\newcommand{\bZ}{{\mathbb Z}}
\newcommand{\bC}{{\mathbbm C}}

\newcommand{\V}{{\cal V}}

\newcommand{\kk}{\mathbbm{k}}

\newcommand{\Rep}{\mathrm{Rep}}

\newcommand{\Fun}{\mathrm{Fun}}

\newcommand{\Aut}{\mathrm{Aut}}

\newcommand{\TL}{\mathcal{T\hspace{-1.5pt}L}}

\newcommand{\MF}{\mathrm{MF}}
\newcommand{\ZMF}{\mathrm{ZMF}}
\newcommand{\HMF}{\mathrm{HMF}}

\newcommand{\hmf}{\mathrm{hmf}}

\newcommand{\MFbi}{\mathrm{MF}_\mathrm{bi}}
\newcommand{\ZMFbi}{\mathrm{ZMF}_\mathrm{bi}}
\newcommand{\HMFbi}{\mathrm{HMF}_\mathrm{bi}}

\newcommand{\mfbi}{\mathrm{mf}_\mathrm{bi}}
\newcommand{\zmfbi}{\mathrm{zmf}_\mathrm{bi}}
\newcommand{\hmfbi}{\mathrm{hmf}_\mathrm{bi}}

\newcommand{\MFgr}{\mathrm{MF}^\mathrm{gr}}
\newcommand{\ZMFgr}{\mathrm{ZMF}^\mathrm{gr}}
\newcommand{\HMFgr}{\mathrm{HMF}^\mathrm{gr}}

\newcommand{\ZMFbigr}{\mathrm{ZMF}^\mathrm{gr}_\mathrm{bi}}
\newcommand{\HMFbigr}{\mathrm{HMF}^\mathrm{gr}_\mathrm{bi}}

\newcommand{\Pd}{\mathcal{PT}_d}
\newcommand{\Pdgr}{\mathcal{PT}_d^\mathrm{gr}}

\newcommand{\Ad}{\mathrm{Ad}}
\newcommand{\id}{\mathrm{id}}
\newcommand{\Id}{I\hspace{-.5pt}d}

\usepackage{enumitem}

\newlist{CFT}{enumerate}{1}
\setlist[CFT]{label=(C\arabic*):}
\usepackage{enumitem}
 
\newlist{VOA}{enumerate}{1}
\setlist[VOA]{label=VOA\arabic*:} 

\begin{document}

\chapter*{Matrix factorizations and the Landau-Ginzburg/conformal field theory correspondence}
\thispagestyle{empty}

\quad

\quad

\quad

\begin{center}

\large{Dissertation with the aim of achieving a doctoral degree

\quad

at the Faculty of Mathematics, Computer Science and Natural Sciences 

\quad

Department of Mathematics

\quad

of Universit{\"a}t Hamburg

\quad

\quad

\quad

\quad

\quad

submitted by Ana Ros Camacho

\quad

Hamburg, 2014}

\end{center}

\vspace{6cm}

\newpage
\quad

\quad

\quad

\quad

\quad
\quad

\quad

\quad

\quad

\quad
\quad

\quad

\quad

\quad

\quad
\quad

\quad

\quad

\quad

\quad
\quad

\quad

\quad
\quad

\quad

\quad

\quad

\quad
\quad

\quad

\quad

\quad

\quad
\quad

\quad

\quad

\quad

\quad
\quad

\quad

\quad

\quad

\quad

\quad

\quad
\begin{flushleft}
Day of oral defense: November 12th, 2014

\quad

The following evaluators recommend the admission of the dissertation:

Name: Prof. Dr. Ingo Runkel

Name: J--Prof. Dr. Daniel Roggenkamp
\end{flushleft}

\newpage
\chapter*{Eidesstattliche Versicherung}

\quad

\quad

Hiermit erkl\"are ich an Eides statt, dass ich die vorliegende Dissertationsschrift selbst verfasst und keine anderen als die angegebenen Quellen und Hilfsmittel benutzt habe.

\quad

\quad

\quad

\quad

\quad

\quad

\quad

\quad

\quad

\quad

\quad

\quad

\quad

\begin{flushright}
Hamburg, den 22. September 2014

\quad

\quad

\quad

\quad

Ana Ros Camacho
\end{flushright}

\newpage

\quad

\begin{flushright}
\textit{A vast similitude interlocks all,}

\textit{All spheres, grown, ungrown, small, large, suns, moons, planets}

\textit{All distances of place however wide,}

\textit{All distances of time, all inanimate forms,}

\textit{All souls, all living bodies, though they be ever so different, or in different worlds,}

\textit{All gaseous, watery, vegetable, mineral processes, the fishes, the brutes,}    

\textit{All nations, colors, barbarisms, civilizations, languages,}

\textit{All identities that have existed, or may exist, on this globe, or any globe,}

\textit{All lives and deaths, all of the past, present, future,}

\textit{This vast similitude spans them, and always has spann'd,}

\textit{And shall forever span them and compactly hold and enclose them.}

\quad
\quad

Walt Whitman, \textit{``On the beach at night alone"}
\end{flushright}

\quad

\quad

\quad

\quad

\quad

\quad

\begin{flushleft}
\textit{Out of the night that covers me,}

\textit{Black as the pit from pole to pole,}

\textit{I thank whatever gods may be}

\textit{For my unconquerable soul.}

\quad

\textit{In the fell clutch of circumstance}

\textit{I have not winced nor cried aloud.}

\textit{Under the bludgeonings of chance}

\textit{My head is bloody, but unbowed.}

\quad

\textit{Beyond this place of wrath and tears}

\textit{Looms but the horror of the shade,}

\textit{And yet the menace of the years}

\textit{Finds and shall find me unafraid.}

\quad

\textit{It matters not how strait the gate,}

\textit{How charged with punishments the scroll,}

\textit{I am the master of my fate:}

\textit{I am the captain of my soul.}

\quad
\quad

W.E. Henley, \textit{``Invictus"}
\end{flushleft}

\newpage
\tableofcontents

\newpage

\addcontentsline{toc}{chapter}{Introduction}
\chapter*{Introduction}

The focus of the present thesis is the so-called Landau-Ginzburg/conformal field theory (LG/CFT) correspondence. This correspondence dates from the late 80's and early 90's in the physics literature (\cite{KMS}, \cite{howewest1}, \cite{howewest2}, \cite{howewest3}, \cite{martinec}, \cite{vafawarner}) and in particular, it predicts a relation between defects in Landau-Ginzburg models and defects in conformal field theories. This relation is supported by examples, but not understood in general, nor up to date is there a clear mathematical conjecture of the LG/CFT correspondence. The pursuit of a precise mathematical statement of this conjecture continues to generate a rich mathematical output, and this PhD thesis is another contribution towards this end.

On the one hand of this correspondence, by a Landau-Ginzburg model we mean a 2-dimensional $\left( 2,2 \right)$--supersymmetric model characterized by a polynomial $W \in S$ (with $S$ a polynomial ring) called potential. Let two Landau-Ginzburg models be characterized by two polynomials $W$ and $W'$ resp. A defect between two Landau-Ginzburg models is a codimension one interface and it is described by a matrix (bi)factorization \cite{brunrogg1}, which is defined as a pair $\left(M,d^M \right)$, where $M$ is a $\mathbb{Z}_2$--graded free (bi)module over a polynomial ring and $d^M$ is a degree 1 $S$--(bi)module morphism satisfying that $d^M \circ d^M=(W-W').\mathrm{id}_M$. In 1980 Eisenbud \cite{eisenbud} first described matrix factorizations within the framework of maximal Cohen-Macaulay modules. They have been shown to occur in many areas of pure mathematics, such as representation theory, singularity theory, homological mirror symmetry, knot invariants and topological field theories (see e.g. \cite{buchweitz,khovroz,orlov,orlov2}). Matrix factorizations are the most important concept and the main tool of this thesis.

On the other side of the correspondence we have conformal field theories (CFTs). More precisely, we are interested in rational full CFTs, meaning a collection of single-valued functions called the correlators which satisfy the so-called Ward identities, and are compatible with the operator product expansion. In addition, a rational vertex algebra (some generalization of a commutative algebra) describes the chiral symmetries of the CFT. Starting from the ground--breaking paper by Moore and Seiberg \cite{mooreseiberg} and later by the works of Fuchs--Runkel--Schweigert et al \cite{tft1,tft2,tft3,tft4,tft5}, at present we have a good understanding of their behaviour. More precisely, a full CFT can be fixed with only a tuple $\left( \mathcal{V},A \right)$ where $\mathcal{V}$ is a rational vertex algebra and $A$ is a special symmetric Frobenius algebra in the representation category $\mathrm{Rep} \left( \mathcal{V} \right)$. Since $\mathcal{V}$ is rational, $\mathrm{Rep} \left( \mathcal{V} \right)$ is a modular tensor category. Defects between different CFTs are described by bimodules over two of these special symmetric Frobenius algebras.

Inspired by the physics literature, we would like to compare defects from both sides proving equivalences of tensor categories: on the one side we have categories of matrix factorizations; on the other, categories of modules over special symmetric Frobenius algebras. This was the main objective of the research carried out during the author's PhD. Two papers came out of this research: \cite{drcr} and \cite{crcr}.

This thesis is structured as follows: on Chapter \ref{ch:catback} we offer a brief review on some concepts of category theory required for the understanding of the later chapters, like e.g. bicategories, modular categories or Temperley-Lieb categories. 

In Chapter \ref{ch:mfs}, we introduce matrix factorizations and review their basic properties, as well as offer an exposition on how they arise within the physics context. 

In Chapter \ref{ch:LGCFT}, first we introduce the categorical approach to full CFTs and then explain the LG/CFT correspondence, reviewing results appearing in the physics literature supporting this correspondence.

Chapter \ref{LGCFTTL} contains the joint work developed together with Alexei Davydov and Ingo Runkel described in the paper \cite{drcr}. In this project we compare a certain kind of matrix factorizations, the so-called permutation-type matrix factorizations, and certain representations of the vertex operator algebra associated to the coset $\frac{\hat{\mathfrak{su}}(2)_{d-2} \oplus \hat{\mathfrak{u}}(1)_{4}}{\hat{\mathfrak{u}}(1)_{2d}}$. More precisely, the setup is:
\begin{itemize}
\item Consider the Deligne tensor product of categories of representations $$\mathrm{Rep}\left(\hat{\mathfrak{su}}(1)\right)_{d-2} \boxtimes \overline{\mathrm{Rep}} \left(\hat{\mathfrak{u}}(1)\right)_{2d}\boxtimes \mathrm{Rep} \left(\hat{\mathfrak{u}}(1)\right)_{4}$$ where $\mathrm{Rep}\left(\hat{\mathfrak{su}}(1)\right)_{d-2}$ is the category of integrable highest weight representations of $\hat{\mathfrak{su}}(2)_{d-2}$, $\mathrm{Rep}\left(\hat{\mathfrak{u}}(1)\right)_{2d}$ ($\mathrm{Rep}\left(\hat{\mathfrak{u}}(1)\right)_{4}$) is the category of representations of the rational vertex operator algebra obtained when extending the $\mathfrak{u}(1)$-current algebra by two fields of weight $d$ (resp. of weight 2). The overline means that we replace the braiding and twist by their inverses. Label the simple objects as $\left[ l,m,s \right]$ where $l \in \lbrace 0,\ldots,d-2 \rbrace$, $m \in \mathbb{Z}_{2d}$, $s \in \mathbb{Z}_4$. We will consider the full subcategory whose objects are isomorphic to direct sums of simples satisfying that $l+m \in 2 \mathbb{Z}$ and $s=0$. Denote it as $\C(N{=}2,d)$.

\item By a permutation-type matrix factorization we mean a matrix bifactorization of $x^d-y^d$ of the form $P_J=\begin{pmatrix} 0 & \prod\limits_{i \in J} \left( x-\eta^i y \right) \\ \prod\limits_{i \in \lbrace 0,\ldots,d-1 \rbrace \setminus J} \left( x-\eta^i y \right) & 0 \end{pmatrix} \ ,$ where $\eta = e^{2 \pi i/d}$ and $J$ is a subset of $\mathbbm{Z}_d$. We denote the category whose objects are matrix factorizations of $W=x^d-y^d$ ($\in S=\mathbbm{C} \left[ x,y \right]$) and whose morphisms are $S$-linear maps of degree zero closed with respect to the differential on the morphism space modulo homotopy as $\HMF_{\mathbbm{C} \left[ x,y \right],x^d-y^d}$. If we consider instead $\mathbb{C}$-graded matrix factorizations as objects, then we denote it as $\HMF^{\mathrm{gr}}_{\mathbbm{C} \left[ x,y \right],x^d-y^d}$. Consider the full subcategory of $\mathrm{HMF}^{\mathrm{gr}}_{\mathbbm{C}\left[x,y\right],x^d-y^d}$ which is tensor generated by permutation-type matrix factorizations whose sets are of the kind $S=\lbrace m,...,m+l \rbrace$ and whose morphisms are morphisms of matrix bifactorizations with $\mathbbm{C}$-degree zero. Denote it as $\mathcal{PT}^{\mathrm{gr}}_d$.
\end{itemize}

Our work builds on the following result:

\begin{nonumberthm}{\cite{brunrogg1}}
$\C(N{=}2,d)$ and $\mathcal{PT}^{\mathrm{gr}}_d$ are equivalent as $\mathbbm{C}$-linear categories, via the assignment
$$\left[l,l+2m,0 \right] \longleftrightarrow P_{\lbrace m,\ldots,m+l \rbrace} \ .$$ Moreover, this equivalence is multiplicative.
\end{nonumberthm}

This raises the following natural question: are these two categories equivalent as tensor categories as well?

From a quantum field theory point of view, one expects a tensor equivalence, but mathematically this is surprising: the tensor product on the matrix factorization side is that of the underlying modules over polynomial rings and has trivial associator. On the other hand, the associator on the vertex algebra side involves the quantum 6j-symbols for $su(2)$.

Our approach to this problem is the following: consider the Deligne tensor product of the Temperley Lieb category $\mathcal{TL}_\kappa$ and the category of $\mathbb{Z}_d$-graded vector spaces $\mathbb{Z}_d\text{-}\mathrm{Vec}$. For $d$ odd, one can construct tensor functors from this category to $\C(N{=}2,d)$ and $\mathcal{PT}^{\mathrm{gr}}_d$ resp.:
\begin{equation}
\begin{split}
\mathcal{TL}_{\kappa} \boxtimes \mathbb{Z}_d\text{-}\mathrm{Vec} & \to \mathcal{PT}^{\mathrm{gr}}_d \\
\mathcal{TL}_\kappa \boxtimes \mathbb{Z}_d\text{-}\mathrm{Vec} & \to \C(N{=}2,d)
\end{split}
\label{functors}
\end{equation}

One then proves that both these functors annihilate the unique proper tensor ideal in $\mathcal{TL}_{\kappa}$ and therefore descend to equivalences from the quotient category onto the image of the functors:

\begin{nonumberthm}{(\textbf{\ref{thm:main}})}
For $d$ odd, $\mathcal{PT}^{\mathrm{gr}}_d$ and $\C(N{=}2,d)$ are equivalent as tensor categories.
\end{nonumberthm}

Which is the first main result of this thesis.

Chapter \ref{ch:orbeq} contains the joint work with Nils Carqueville and Ingo Runkel described in the paper \cite{crcr}. In this project we proved a conjecture stated in \cite{carqrunkel2}. It is a well known fact that simple singularities have an ADE classification. For two variables (where more can be added via Kn{\"o}rrer periodicity), the associated polynomials are 
\begin{equation}
\begin{split}
W^{A_{d-1}} &=x^d-y^2 \quad , \quad W^{D_{d+1}}=x^d-xy^2 \\
W^{E_6} &=x^3+y^4 \quad , \quad W^{E_7}=x^3+xy^3 \quad , \quad W^{E_8}=x^3+y^5 \\
\end{split}
\nonumber
\end{equation}
Landau-Ginzburg models with these potentials are believed to correspond to $N=2$ minimal conformal field theories. These rational conformal field theories, for a given value of the central charge, are expected to be (generalised) orbifolds of each other. Inspired by this fact we wanted to prove a similar result for matrix factorizations. 

The framework to approach this problem is the bicategory of Landau-Ginzburg models $\mathcal{LG}$, where the objects are pairs $\left( S,W \right)$ of a polynomial ring over $\mathbbm{k}$ (for a fixed field $\mathbbm{k}$) and a potential. The category of 1-morphisms from $\left( S',W' \right)$ to $\left( S,W \right)$ is that of (finite rank) matrix bifactorizations $\mathrm{hmf}_{bi;(S,W),(S',W')}$. The composition of 1-morphisms is given by the tensor product of matrix factorizations. This bicategory has adjoints which can be explicitly written via Atiyah classes providing a practical hands-on toolkit for computations \cite{carqmurfet}. In \cite{carqrunkel2} an equivalence relation on objects of $\mathcal{LG}$ is introduced, which states that two objects are equivalent if there exists a 1-morphism between them which has invertible quantum dimension (that means, it has to satisfy certain properties concerning its adjoints). This implies non-obvious equivalences of categories, for example:

\begin{prop}{\cite{carqrunkel2}}
Let $\left( S,W \right)$, $\left( S',W' \right) \in \mathrm{Ob} \left( \mathcal{LG} \right)$ and $X \in \mathrm{hmf}_{S \otimes_\mathbbm{k} S', W-W'}$ with invertible quantum dimension. Denote by $X^\dagger$ the right adjoint of $X$. Then,
$$\mathrm{hmf}_{S', W'} \cong \mathrm{mod} \left( X^\dagger \otimes X \right)$$
where $X^\dagger \otimes X$ is a symmetric separable Frobenius algebra object in the 1-morphism category $\mathcal{LG} \left( \left( S,W \right),\left( S,W \right) \right)$ and $\mathrm{mod} \left( X^\dagger \otimes X \right)$ denotes the category of modules over $X^\dagger \otimes X$ in $\mathrm{hmf}_{S, W}$.
\label{proporbeq}
\end{prop}

By giving explicitly matrix factorizations of non-zero quantum dimension, in 
\cite{crcr} we classify equivalence classes within the ADE potentials listed above.

\begin{nonumberthm}{(\textbf{\ref{thm:ADEorbifolds}})}
The equivalences between simple singularities are generated by (the ring is $\mathbbm{C}[x,y]$ in each case)
$$
W^{A_d} \sim W^{D_{d+1}}
~~,\quad
W^{A_{11}} \sim W^{E_{6}}
~~,\quad
W^{A_{17}} \sim W^{E_7}
~~,\quad
W^{A_{29}} \sim W^{E_8} \ .
$$
\end{nonumberthm}

We also describe explicitly the symmetric separable Frobenius algebras from Proposition \ref{proporbeq}, that is, we give algebras $A_{D_{d+1}}$, $A_{E_6}$, $A_{E_7}$, $A_{E_8}$ in the respective 1-endomorphism categories, such that $\mathrm{hmf} \left(\mathbbm{C}[x,y],W^T \right) \cong \mathrm{mod} \left(A_T \right)$ for $T \in\{ D_{d+1}, E_{6,7,8} \}$.

In addition to these results, we also point out some guidelines with which we hope to find new orbifold equivalences.

Finally, in the Appendix one can find a missing proof in Chapter \ref{ch:LGCFT}, as well as a collection of results on categories of equivariant objects and pointed categories useful for the understanding of the approach taken in Chapter \ref{LGCFTTL}.

\chapter{Categorical background}
\label{ch:catback}

In this chapter we will introduce some background on categories which will be necessary for our exposition.

\section{Category theory background: bicategories and modular categories}
\label{sec:catth}

We will assume only very basic knowledge of categories. Our main references are \cite{borceux,benabou,maclane,mueger,deligne,BaKiBook,kelly,carqmurfet,schweigert}.

First of all, we would like to construct a certain higher categorical structure: that of bicategories.

\begin{defn}
A bicategory $\mathcal{B}$ is specified by the following data:
\begin{itemize}
\item a class $\mathrm{Ob} \left( \mathcal{B} \right)$ of objects;
\item for each pair $A,B \in \mathrm{Ob} \left( \mathcal{B} \right)$, a category $\mathcal{B} \left( A,B \right)$ whose objects are called \textit{1-morphisms} and whose morphisms are called \textit{2-morphisms}; we write $\alpha \odot \beta$ for the composite of the 2-morphisms $\alpha,\beta$;
\item for each triple of objects $A,B,C \in \mathrm{Ob} \left( \mathcal{B} \right)$, a composition law given by a functor $$c_{ABC} \colon \mathcal{B} \left( A,B \right) \times \mathcal{B} \left( B,C \right) \to \mathcal{B} \left( A,C \right);$$ given 1-morphisms $f \colon A \to B$, $g \colon B \to C$ of the bicategory $\mathcal{B}$, we write $g \circ f$ for their composite $c_{ABC} \left( f,g \right)$; given other 1-morphisms $f' \colon A \to B$, $g' \colon B \to C$ and 2-morphisms $\gamma \colon f \Rightarrow f'$, $\delta \colon g \Rightarrow g'$, we write $\delta \ast \gamma$ for their composite $c_{ABC} \left( \gamma,\delta \right)$;
\item for each $A \in \mathrm{Ob} \left( \mathcal{B} \right)$, there is a functor $\upsilon_A \colon \mathbf{1} \to \mathcal{B} \left( A,A \right)$, where $\mathbf{1}$ is the terminal objects of the category of small categories. The \textit{identity 1-morphism} $1_A \colon A \to A$ is the image of the unique object of $\mathbf{1}$ under the functor $\upsilon_A$; we write $\iota_A$ for the identity 2-morphisms on $1_A$ and $\iota_f$ for the identity 2-endomorphism of a 1-morphism $f \colon A \to B$;
\item for each quadruple $A,B,C,D \in \mathrm{Ob} \left( \mathcal{B} \right)$, a natural isomorphism $$\alpha_{ABCD} \colon c_{ABD} \circ \left( 1 \times c_{BCD} \right) \Rightarrow c_{ACD} \circ \left( c_{ABC} \times 1 \right)$$ which is called the \textit{associativity isomorphism};
\item for each pair $A,B \in \mathrm{Ob} \left( \mathcal{B} \right)$, two natural isomorphisms
\begin{equation}
\begin{split}
\lambda_{AB} &\colon c_{AAB} \circ \left( \upsilon_A \times 1 \right)  \Rightarrow 1\\
\rho_{AB} &\colon c_{ABB} \circ \left( 1 \times \upsilon_B \right) \Rightarrow 1
\end{split}
\nonumber
\end{equation} which are called the \textit{left} and \textit{right unit isomorphisms} resp.;
\end{itemize}
satisfying
\begin{enumerate}
\item Associativity coherence: given objects $A,B,C,D,E \in \mathrm{Ob} \left( \mathcal{B} \right)$ and 1-morphisms $$\xymatrix{ A \ar[r]^f & B \ar[r]^g & C \ar[r]^h & D \ar[r]^k & E}$$ such that the following diagram commutes:
\begin{equation}
\xymatrix{
k \circ \left( h \circ \left( g \circ f \right) \right) \ar[d]_-{\alpha_{g \circ f,h,k}} \ar[rr]^-{\iota_k \ast \alpha_{f,g,h}} && k \circ \left( \left( h \circ g \right) \circ f \right) \ar[rr]^-{\alpha_{f,h \circ g,k}} && \left( k \circ \left( h \circ g \right) \right) \circ f \ar[d]^-{\alpha_{g,h,k} \ast \iota_f} \\ \left( k \circ h \right) \circ \left( g \circ f \right) \ar[rrrr]^-{\alpha_{f,g,k \circ h}} &&&& \left( \left( k \circ h \right) \circ g \right) \circ f
}
\nonumber
\end{equation}
where we have written $\alpha_{f,g,h}$ instead of $\left( \alpha_{ABCD} \right)_{(f,g,h)}$ for the sake of simplicity.

\item Identity coherence: given objects $A,B,C \in \mathrm{Ob} \left( \mathcal{B} \right)$ and 1-morphisms $$\xymatrix{A \ar[r]^f &B \ar[r]^g &C}$$ such that the following diagram commutes:
\begin{equation}
\xymatrix{
\left( g \circ 1_B \right) \circ f \ar[dr]_-{\rho_g \ast \iota_f} && g \circ \left( 1_B \circ f \right) \ar[dl]^-{\iota_g \ast \lambda_f} \ar[ll]_-{\alpha_{f,1_B,g}} \\ & g \circ f&}
\nonumber
\end{equation}
where we have written $\lambda_f,\rho_g$ instead of $\left( \lambda_{AB} \right)_f$, $\left( \rho_{BC} \right)_g$ again for simplicity.

\end{enumerate}
\label{bicat}
\end{defn}

We will also introduce the notion of adjoints in a bicategory:
\begin{defn}
\begin{itemize}
\item Let $\mathcal{B}$ be a bicategory. An \textit{adjunction} in $\mathcal{B}$ is a tuple $\left( f,g,\mathrm{coev},\mathrm{ev} \right)$ where $f,g$ are two 1-morphisms (the \textit{adjoint pair}) $f \colon A \to B$ and $g \colon B \to A$ (with $A,B\in \mathrm{Ob} \left( \mathcal{B} \right)$) and 
\begin{equation}
\begin{split}
\mathrm{coev} & \colon 1_B \Rightarrow f \circ g \\
\mathrm{ev} & \colon g \circ f \Rightarrow 1_A
\end{split}
\nonumber
\end{equation}
two 2-morphisms which altogether satisfy the following equalities \footnote{Notice here that we use explicitly the fact that $\alpha$ is an isomorphism, not just a natural transformation.}:
\begin{equation}
\begin{split}
\left( \iota_f \ast \mathrm{ev} \right) \odot \alpha_{f,g,f} \odot \left( \mathrm{coev} \ast \iota_f \right) &= \rho_f \odot \lambda_f \\
\left( \mathrm{ev} \ast \iota_g \right) \odot \alpha^{-1}_{g,f,g} \odot \left( \iota_g \ast \mathrm{coev} \right) &=\lambda_g \odot \rho_g
\end{split}
\label{bicatzorro}
\end{equation}
In this case we say that $g$ is \textit{left adjoint} to $f$ and that $f$ is \textit{right adjoint} to $g$. The 2-morphisms $\mathrm{ev}$ and $\mathrm{coev}$ are referred to as the \textit{evaluation} and \textit{coevaluation} maps resp. of the adjunction.
\item $\mathcal{B}$ has \textit{left adjoints} (resp. \textit{right adjoints}) if every 1-morphism in $\mathcal{B}$ admits a left adjoint (resp. right adjoint). 
\end{itemize}
\end{defn}

If they exist these adjoints are unique up to isomorphism, and the left and right adjoints of $f$ are denoted by $^\dagger f$ and $f^\dagger$, resp. If a 1-morphism $f:A \rightarrow B$ has both a left and right adjoint then we write the evaluation and coevaluation maps for ${^\dagger f}$ left adjoint to $f$ as:
\begin{equation}
\begin{split}
\mathrm{ev}_f & \colon {^\dagger f} \circ f \Rightarrow 1_A \\
\mathrm{coev}_f & \colon 1_B \Rightarrow f \circ {^\dagger f}
\end{split}
\nonumber
\end{equation}
and for $f$ left adjoint to $f^\dagger$,
\begin{equation}
\begin{split}
\widetilde{\mathrm{ev}}_f & \colon f \circ f^\dagger \Rightarrow 1_B \\
\widetilde{\mathrm{coev}}_f & \colon 1_A \Rightarrow f^\dagger \circ f
\end{split}
\nonumber
\end{equation}

Note here that left and right adjoints do not coincide in general. We will see some examples within the bicategory of Landau-Ginzburg models in Chapter \ref{ch:mfs}.

\begin{rem}
Equations \ref{bicatzorro} are frequently called the ``(right) snake diagrams" or also the ``(right) Zorro moves". This is because adjoints have a graphical language which allows a very comfortable description of these equations. We will though not introduce it here.
\end{rem}

One may wonder at this point if one can relate these left and right adjoints. One case of interest is when they coincide: ${^\dagger f}=f^\dagger$\footnote{Actually, one may soften this condition to an isomorphism instead of an equality. We choose the strictified version for the sake of simplicity.}. Under this assumption, we say such a bicategory $\mathcal{B}$ is \textsl{pivotal} if the following diagrams commute:
\begin{equation}
\begin{split}
\xymatrix{ 
f^\dagger \ar[d]^{\lambda_{BA}^{-1}} \ar[r]^-{\rho^{-1}_{BA}} & f^\dagger \circ 1_B \ar[r]^-{\iota_{f^\dagger} \ast \mathrm{coev}_g} & f^\dagger \circ \left( g \circ g^\dagger \right) \ar[rr]^-{\iota_{f^\dagger} \ast \left( \phi \ast \iota_{g^\dagger} \right)} && f^\dagger \circ \left( f \circ g^\dagger \right) \ar[d]^-{\alpha^{-1}_{g^\dagger,g,f^\dagger}} \\ 1_A \circ f^\dagger \ar[d]^{\widetilde{\mathrm{coev}}_g \ast \iota_{f^\dagger}} &&&& \left( f^\dagger \circ f \right) \circ g^\dagger \ar[d]^-{\mathrm{ev}_f \ast \iota_{g^\dagger}} \\
\left( g^\dagger \circ g \right) \circ f^\dagger \ar[dr]_-{\left( \iota_{g^\dagger} \ast \phi \right) \ast \iota_{f^\dagger}} &&  && g^\dagger \\ 
&\left( g^\dagger \circ f \right) \circ f^\dagger \ar[r]^-{\alpha_{f^\dagger,f,g^\dagger}} & g^\dagger \circ \left( f \circ f^\dagger \right) \ar[r]^-{\iota_{g^\dagger} \ast \widetilde{\mathrm{ev}}_f} & g^\dagger \circ 1_B \ar[ur]^{\rho_{BA}} & 
}
\end{split}
\nonumber
\end{equation}
for $A,B \in \mathrm{Ob} \left( \mathcal{B} \right)$, $f,g \in \mathcal{B} \left( A,B \right)$, $\phi \colon f \Rightarrow g$, and
\begin{equation}
\begin{split}
\xymatrix{ 
\left( g \circ f \right)^\dagger \ar[r]^-{\lambda_{CA}^{-1}}  \ar[d]^-{\lambda_{CA}^{-1}}& \left( g \circ f \right)^\dagger \circ 1_C \ar[rr]^-{\iota_{(g \circ f)^\dagger} \ast \mathrm{coev}_g} && \left( g \circ f \right)^\dagger \circ \left( g \circ g^\dagger \right) \ar[d]_-{\alpha^{-1}_{g^\dagger,g,(g \circ f)^\dagger}} \\
1_A \circ \left( g \circ f \right)^\dagger \ar[d]^-{\widetilde{\mathrm{coev}}_f \ast \iota_{(g \circ f)^\dagger}} &&& \left( \left( g \circ f \right)^\dagger \circ g \right) \circ g^\dagger \ar[d]_-{\iota_{(g \circ f)^\dagger \circ g} \ast \lambda_{BC}} \\
\left( f^\dagger \circ f \right) \circ \left( g \circ f \right)^\dagger \ar[d]^-{\alpha_{f^\dagger,f,(g \circ f)^\dagger}} &&& \left( \left( g \circ f \right)^\dagger \circ g \right) \circ 1_B \circ g^\dagger \ar[d]_-{\iota_{(g \circ f)^\dagger \circ g} \ast \mathrm{coev}_f \ast \iota_{g^\dagger}} \\
f^\dagger \circ \left( f \circ \left( g \circ f \right)^\dagger \right) \ar[d]^-{\rho_{BA} \ast \iota_{f \circ (g \circ f)^\dagger}} &&& \left( \left( g \circ f \right)^\dagger \circ g \right) \circ \left( f \circ f^\dagger \right) \circ g^\dagger \ar[d]_-{\iota_{(g\circ f)^\dagger \circ g} \ast \alpha_{g^\dagger,f^\dagger,f}}
\\
f^\dagger \circ 1_B \circ \left( f \circ \left( g \circ f \right)^\dagger \right) \ar[d]^-{\iota_{f^\dagger} \ast \widetilde{\mathrm{coev}}_g \ast \iota_{f \circ (g \circ f)^\dagger}} &&& \left( \left( g \circ f \right)^\dagger \circ g \right) \circ f \circ \left( f^\dagger \circ g^\dagger \right) \ar[d]_{\alpha_{(g \circ f)^\dagger,g,f} \ast \iota_{f^\dagger \circ g^\dagger}} \\
f^\dagger \circ \left( g^\dagger \circ g \right) \circ \left( f \circ \left( g \circ f \right)^\dagger \right) \ar[d]^-{\alpha^{-1}_{g,g^\dagger,f^\dagger} \ast \iota_{f \circ(g \circ f)^\dagger}} &&& \left( g \circ f \right)^\dagger \circ \left( g \circ f \right) \circ \left( f^\dagger \circ g^\dagger \right) \ar[d]_-{\mathrm{ev}_{g \circ f} \ast \iota_{f^\dagger \circ g^\dagger}} \\
\left( f^\dagger \circ g^\dagger \right) \circ g \circ \left( f \circ \left( g \circ f \right)^\dagger \right) \ar[d]^-{\iota_{f^\dagger \circ g^\dagger} \ast \alpha^{-1}_{(g \circ f)^\dagger,f,g}} &&& 1_A \circ \left( f^\dagger \circ g^\dagger \right) \ar[d]_-{\lambda_{CA}}\\ 
\left( f^\dagger \circ g^\dagger \right) \circ \left( g \circ f \right) \circ \left( g \circ f \right)^\dagger \ar[drrr]^-{\iota_{f^\dagger \circ g^\dagger} \ast \widetilde{\mathrm{ev}}_{g \circ f}} &&& f^\dagger \circ g^\dagger \\
&&& \left( f^\dagger \circ g^\dagger \right) \circ 1_C \ar[u]^-{\rho_{CA}}
}
\end{split}
\nonumber
\end{equation}
for $A,B,C \in \mathrm{Ob} \left( \mathcal{B} \right)$ and composable 1-morphisms $\xymatrix{A \ar[r]^-f & B \ar[r]^-g &C}$ and their duals.

One can show \cite{carqrunkel3} that in a pivotal bicategory there are natural monoidal isomorphisms $\lbrace \delta_f \rbrace$ ($f \in \mathcal{B} \left( A,B \right)$) between the identity functor and $(-)^{\dagger\dagger}$ on $\mathcal{B}(A,B)$, see \cite{carqrunkel2}.

Then, one can write $\widetilde{\mathrm{ev}}$ and $\widetilde{\mathrm{coev}}$ in terms of $\mathrm{ev}$ and $\mathrm{coev}$ and these natural isomorphisms:
\begin{equation}
\begin{split}
\widetilde{\mathrm{ev}}_f &=\mathrm{ev}_{f^\dagger} \odot (\delta_f \circ \iota_{f^\dagger}) \\
\widetilde{\mathrm{coev}}_f &=(\iota_{f^\dagger} \circ (\delta_f)^{-1}) \odot \mathrm{coev}_{f^\dagger}
\end{split}
\nonumber
\end{equation}

Let us introduce a special case of composition of the above described maps. The \textit{left quantum dimension} of~$f$ is the image of $\iota_B$ under the map $\mathcal{D}_l(f) : \mathrm{End}(1_B) \rightarrow \mathrm{End}(1_A)$ given by 
\begin{equation}
\mathcal{D}_l(f) = \mathrm{ev}_f \odot \, [1_{f^\dagger} \circ (\lambda_f \odot ((-)\circ 1_f) \odot \lambda_f^{-1})) ] \odot \, \widetilde{\mathrm{coev}}_f  \,
\nonumber
\end{equation}
and similarly for the \textit{right quantum dimension} $\mathrm{qdim}_r(f) = \mathcal{D}_r(f)(\iota_A)$. For $g\in \mathcal{B}(B,C)$ these operators satisfy 
\begin{equation}
\mathcal{D}_l(f) \odot \mathcal{D}_l(g) = \mathcal{D}_l (g \circ f) 
\, , \quad 
\mathcal{D}_r(g) \odot \mathcal{D}_r(f) = \mathcal{D}_r (g \circ f ) \,  ,
\nonumber
\end{equation}
and for a proof we refer to e.g. \cite{carqmurfet}.

If the class of objects of a bicategory consists of only one object, the bicategory is equivalent to a \textit{monoidal}\footnote{Sometimes monoidal categories are also called tensor categories, although both notions will not be the same in this thesis. Let $\mathbbm{k}$ be an algebraically closed field (which can be assumed to be the field $\bC$ of complex numbers). We call a category $\C$ \textit{tensor} if it is an additive $\mathbbm{k}$-linear monoidal category such that the tensor product is $\mathbbm{k}$-linear in both arguments.} \textit{category}: with only one object, the 1-morphisms and 2-morphisms are described by the objects and morphisms resp. of a category $\mathcal{C}$. Denote the composition functor $c$ in this case as $\otimes \colon \mathcal{C} \times \mathcal{C} \to \mathcal{C}$. We denote then the composition of the morphisms of $\mathcal{C}$ with $\circ$. The identity 1--morphism $1_\mathcal{C}$ becomes a unit object $I \in \mathrm{Ob} \left( \mathcal{C} \right)$, where the identity 2--morphism $\iota_\mathcal{C}$ is the identity morphism $\mathrm{id}_\mathcal{C}$, and similarly for the associativity and unit isomorphisms: concerning the associativity isomorphisms, they become isomorphisms $\alpha_{A,B,C}: \left( A \otimes B \right) \otimes C \rightarrow A \otimes \left( B \otimes C \right)$ called \textit{associators}, natural in $A$, $B$, $C \in \mathrm{Ob} \left( \mathcal{C} \right)$; on the other hand the unit natural isomorphisms become isomorphisms $\lambda_A : I \otimes A \rightarrow A$, $\rho_A : A \otimes I \rightarrow A$, natural in $A \in \mathrm{Ob} \left( \mathcal{C} \right)$, called \textit{left} and \textit{right unit isomorphisms}. Denote then the monoidal category and its structure morphisms as a tuple $\left( \mathcal{C}, \otimes, I, \alpha, \lambda, \rho \right)$ --or simply as $\mathcal{C}$ when clear by the context--where some subindices will be added to prevent confusion if necessary. This tuple of data must of course satisfy the corresponding associativity and identity coherence conditions.

\begin{rem}
\begin{itemize}
\item Assume $\mathcal{C}$ additive. The endomorphisms of the identity arrow (also called the \textit{tensor unit}) $\mathrm{End} \left( I \right)$ form a commutative ring $\mathbbm{k}$, called the \textit{ground ring}, and for every $A, B \in \mathrm{Ob} \left( \mathcal{C} \right)$, $\mathcal{C} \left( A,B \right)$ is a $\mathbbm{k}$-module.
\item If the natural isomorphisms $\alpha$, $\rho$ and $\lambda$ are actually identities, we say the category is \textit{strict}. By the so-called coherence theorems (see e.g. \cite{maclane}) there is no loss of generality in imposing this strictness property.
\end{itemize}
\end{rem}

Define the following objects:
\begin{defn}
\begin{itemize}
\item An \textit{algebra object}, or simply an \textit{algebra}, in a strict monoidal category $\mathcal{C}$ is a triple $\left( A, m, \eta \right)$ where $A \in \mathrm{Ob} \left( \mathcal{C} \right)$, $m \in \mathrm{Hom} \left( A \otimes A, A \right)$ (\textit{multiplication morphism}) and $\eta \in \mathrm{Hom} \left( \mathbbm{1},A \right)$ (\textit{unit morphism}), such that:
\begin{equation}
\begin{split}
m \circ \left( m \otimes \mathrm{id}_A \right) &=m \circ \left( \mathrm{id}_A \otimes m \right) \\
m \circ \left( \eta \otimes \mathrm{id}_A \right)) &= \mathrm{id}_A =m \circ \left( \mathrm{id}_A \otimes \eta \right)
\end{split}
\nonumber
\end{equation}
\item A \textit{coalgebra} in a strict monoidal category $\mathcal{C}$ is triple $\left( A, \Delta, \epsilon \right)$ where $A \in \mathrm{Ob} \left( \mathcal{C} \right)$, $\Delta \in \mathrm{Hom} \left( A, A \otimes A \right)$ (\textit{coassociative coproduct}) and $\epsilon \in \mathrm{Hom} \left( A,I \right)$ (\textit{counit}), such that: 
\begin{equation}
\begin{split}
\left( \Delta \otimes \mathrm{id}_A \right) \circ \Delta &=\left( \mathrm{id}_A \otimes \Delta \right) \circ \Delta \\
\left( \epsilon \otimes \mathrm{id}_A \right) \circ \Delta &=\mathrm{id}_A =\left( \mathrm{id}_A \otimes \epsilon \right) \circ \Delta
\end{split}
\nonumber
\end{equation}
\end{itemize}
\end{defn}

To call $A$ an algebra is appropriate because in the particular case that $\mathcal{C}$ is the category of $\mathbbm{C}$-vector spaces (or some other field), the prescription reduces to the conventional notion of an algebra.

It is possible to define functors between monoidal categories, which are called accordingly \textit{monoidal functors}\footnote{The notion of tensor functor works analogously to the one for categories: a monoidal functor between tensor categories is \textit{tensor} if it is $\mathbbm{k}$-linear.}, and natural transformations between monoidal functors, called \textit{monoidal natural transformations} and we will describe them now in detail.

\begin{defn}
\begin{itemize}
\item Let $\left( \mathcal{C}, \otimes_{\mathcal{C}}, \alpha_{\mathcal{C}}, I_{\mathcal{C}}, \lambda_{\mathcal{C}}, \rho_{\mathcal{C}} \right)$ and $\left( \mathcal{D}, \otimes_{\mathcal{D}}, \alpha_{\mathcal{D}}, I_{\mathcal{D}}, \lambda_{\mathcal{D}}, \rho_{\mathcal{D}} \right)$ be two monoidal categories. A \textit{monoidal functor} from $\mathcal{C}$ to $\mathcal{D}$ is a triple $\left( F,\varphi_0, \varphi_2 \right)$ where:
\begin{itemize}
\item $F$ is a functor $F \colon \mathcal{C} \to \mathcal{D}$;
\item $\varphi_0$ is an isomorphism $\varphi_0 \colon I_\mathcal{D} \to F \left( I_\mathcal{C} \right)$ in the category $\mathcal{D}$;
\item $\varphi_2 \colon \otimes_\mathcal{D} \circ \left( F \times F \right) \to F \circ \otimes_{\mathcal{C}}$ is a natural isomorphism of functors $\mathcal{C} \times \mathcal{C} \rightarrow \mathcal{D}$, including in particular an isomorphism for any pair of objects $A$, $B \in \mathrm{Ob} \left( \mathcal{C} \right)$, $\varphi_2 \left( A,B \right) \colon F \left( A \right) \otimes_{\mathcal{D}} F \left( B \right) \to F \left( A \otimes_{\mathcal{C}} B \right)$;
\end{itemize}
such that:
\begin{itemize}
\item compatibility with associativity holds:
\begin{equation}
\xymatrix{
\left( F \left( A \right) \otimes F \left( B \right) \right) \otimes F \left( C \right) \ar[rrr]^{\alpha_{F(A),F(B),F(C)}} \ar[d]^{\varphi_2 (A,B) \otimes \mathrm{id}_{F(C)}} &&& F \left( A \right) \otimes \left( F \left( B \right) \otimes F \left( C \right) \right) \ar[d]^{\mathrm{id}_{F(A)\otimes \varphi_2(B,C)}} \\
F \left( A \otimes B \right) \otimes F \left( C \right) \ar[d]^{\varphi_2 (A \otimes B, C)} &&& F \left( A \right) \otimes F \left( B \otimes C \right)  \ar[d]^{\varphi_2 (A,B \otimes C)} \\ F \left( \left( A \otimes B \right) \otimes C \right) \ar[rrr]^{F(\alpha_{A \otimes B \otimes C})} &&& F\left( A \otimes \left( B \otimes C \right) \right)
}
\nonumber
\end{equation}
\item compatibility with the left unit holds:
\begin{equation}
\xymatrix{
I_\mathcal{D} \otimes F \left( A \right) \ar[rr]^{\lambda_{F(A)}} \ar[d]^{\varphi_0 \otimes \mathrm{id}_{F(A)}} && F \left( A \right) \\ 
F \left( I_\mathcal{C} \right) \otimes F \left( A \right) \ar[rr]^{\varphi_2(I_\mathcal{C},A)} && F \left( I_\mathcal{C} \otimes A \right) \ar[u]^{F( \lambda_A)}
}
\nonumber
\end{equation}
\item compatibility with the right unit holds:
\begin{equation}
\xymatrix{
F \left( A \right) \otimes I_\mathcal{D} \ar[rr]^{\rho_{F(A)}} \ar[d]^{\mathrm{id}_{F(A)} \otimes \varphi_0} && F \left( A \right) \\ 
F \left( A \right) \otimes F \left( I_\mathcal{C} \right) \ar[rr]^{\varphi_2(A,I_\mathcal{C})} && F \left( A \otimes I_\mathcal{C} \right) \ar[u]^{F( \rho_A)}
}
\nonumber
\end{equation}
\end{itemize}
\item Let $\left( F, \varphi_0, \varphi_2 \right)$, $\left( F', \varphi'_0, \varphi'_2 \right)$ be two monoidal functors. A \textit{monoidal natural transformation} between monoidal functors is a natural transformation $\eta: F \rightarrow F'$ such that:
\begin{itemize}
\item compatibility with the tensor unit holds:
\begin{equation}
\xymatrix{
& F \left( I_{\mathcal{C}} \right) \ar[dd]^{\eta_{I_{\mathcal{C}}}} \\
I_{\mathcal{D}} \ar[ur]^{\varphi_0} \ar[dr]^{\varphi'_0} & \\
& F' \left( I_{\mathcal{C}} \right)
}
\nonumber
\end{equation}
\item compatibility with the tensor product holds: for all pairs $\left( A,B \right)$, $A, B \in \mathrm{Ob} \left( \mathcal{C} \right)$
\begin{equation}
\xymatrix{
F \left( A \right) \otimes F \left( B \right) \ar[rr]^{\varphi_2(A,B)} \ar[d]^{\eta_A \otimes \eta_B} && F \left( A \otimes B \right) \ar[d]^{\eta_{A \otimes B}} \\ 
F' \left( A \right) \otimes F' \left( B \right) \ar[rr]^{\varphi'_2(A,B)} && F' \left( A \otimes B \right)
}
\nonumber
\end{equation}
\end{itemize}
\end{itemize}
\end{defn}

Thanks to this last definition, then one can define:

\begin{defn}
\textit{Monoidal natural isomorphisms} are invertible monoidal natural transformations. An \textit{equivalence of monoidal categories} is a pair of monoidal functors $F \colon \mathcal{C} \to \mathcal{D}$, $G \colon \mathcal{D} \to \mathcal{C}$ and monoidal natural isomorphisms $\eta \colon \mathrm{id}_{\mathcal{D}} \to F \circ G$, $\theta \colon G \circ F \to \mathrm{id}_\mathcal{C}$.
\end{defn}

Our objective for the rest of this section is to introduce modular categories. 

At this point, let us introduce the following kind of categories.

\begin{defn}
We say that a category $\mathcal{C}$ is an \textit{abelian category} if there is a zero object (that we will denote as $0$) and the morphisms possess various properties: every morphism set is an abelian group (i.e. for any two $A,B \in \mathrm{Ob} \left( \mathcal{C} \right)$, $\mathcal{C}\left( A,B \right)$ is an abelian group) and composition of morphisms is bilinear (i.e. for any morphisms $f,f' \in \mathcal{C} \left( A,B \right)$, $g,g' \in \mathcal{C} \left( B,C \right)$, $\left( f+f' \right)\circ \left( g+g' \right)=f \circ g+f' \circ g+f \circ g'+f' \circ g'$, where we denote as $+$ the operation of the abelian group); every finite set of objects has a biproduct (i.e. we can form finite direct sums and direct products); every morphism has a kernel and a cokernel; every monomorphism is the kernel of its cokernel, and every epimorphism is the cokernel of its kernel; and finally, every morphism $f$ can be written as the composition $f=h \circ g$ of a monomorphism $h$ and an epimorphism $g$.
\end{defn}

\begin{rem}
Notice here that in abelian categories we can have finite direct sums (of both objects and morphisms).
\end{rem}

Assume from now on $\mathcal{C}$ to be an abelian strict tensor category with ground ring $\mathbbm{C}$. We need to enrich it with quite a bit of structure.

\begin{defn}
A \textit{simple} object $U \in \mathrm{Ob} \left( \mathcal{C} \right)$ is an object satisfying that $\mathrm{End} \left( U \right)= \mathbbm{C} \mathrm{id}_U$\footnote{Actually this is the usual definition of an \textit{absolutely simple} object, but in any abelian category over an algebraically closed ground field the notions of simple and absolutely simple objects are equivalent.}. In particular, $I$ is automatically simple. A \textit{semisimple} category is then characterized by the property that every object is the direct sum of finitely many simple objects.
\end{defn}

Semisimplicity of a tensor category $\mathcal{C}$ implies in particular \textit{dominance} of $\mathcal{C}$. This means that there exists a family of $\lbrace U_i \rbrace_{i \in I}$ of simple objects with the following property: for any $A,B \in \mathrm{Ob} \left( \mathcal{C} \right)$ every morphism $f \in \mathrm{Hom} \left( A,B \right)$ can be decomposed into a finite sum $f=\sum\limits_r g_r \circ h_r$ with $h_r \in \mathrm{Hom} \left( A,U_i \right)$ and $g_r \in \mathrm{Hom} \left( U_i,B \right)$ for suitable members $U_i=U_{i \left( r \right)}$ (possibly with repetitions) of this family.

Assume $\mathcal{C}$ to be in addition semisimple. Next, we would like to supplement $\mathcal{C}$ with three additional ingredients: 
\begin{enumerate}
\item Dualities: the bicategorical adjoints for a bicategory with only one object (i.e. a monoidal category) are usually called \textit{duals}. A right adjoint is called a \textit{right dual}, and it is denoted with a left superscript $^\vee$\footnote{Resp. for left adjoints (left duals), denoted with a right superscript $^\vee$. As we will see in a second, in the category we want to construct there is automatically a left duality, so we only require existence of right duals.}.

\item Braiding: in a tensor category, a braiding consists of a family of isomorphisms $b_{A,B} \in \mathrm{Hom} \left( A \otimes B, B \otimes A \right)$, one for each pair $A$, $B \in \mathcal{C}$.
\item Twist: for each object $A \in \mathrm{Ob} \left( \mathcal{C} \right)$, a family of isomorphisms $\theta_A$.
\end{enumerate}

Of course, braiding, twist and dualities are subject to a number of consistency conditions. Namely, one has to impose:
\begin{itemize}
\item Equations \ref{bicatzorro} for the right duals, 
\item functoriality of the braiding,
\begin{equation}
b_{A,B} \circ \left( g \otimes f \right)=b_{A',B'} \circ \left( f \otimes g \right)
\nonumber
\end{equation}
and tensoriality,
\begin{equation}
\begin{split}
b_{A \otimes B,D} &=\left( b_{A,D} \otimes \mathrm{id}_B \right) \circ \left( \mathrm{id}_A \otimes b_{B,D} \right) \\
b_{A,B \otimes D} &=\left( \mathrm{id}_B \otimes b_{A \otimes D} \right) \circ \left( b_{A \otimes B} \otimes \mathrm{id}_D \right)
\nonumber
\end{split}
\end{equation}
for any $A,B,D,A',B' \in Ob \left( \mathcal{C}\right)$ and $f \colon A \to A'$, $g\colon B \to B'$,
\item functoriality of the twist,
\begin{equation}
\theta_B \circ f =f \circ \theta_A
\nonumber
\end{equation}
for any two objects $A,B \in Ob \left( \mathcal{C}\right)$ and $f:A \rightarrow B$, and

\item compatibility of the twist with duality,
\begin{equation}
\left( \theta_A \otimes \mathrm{id}_{A^\vee} \right) \circ \mathrm{coev}_A=\left( \mathrm{id}_A \otimes \theta_{A^\vee} \right) \circ \mathrm{coev}_A
\nonumber
\end{equation}
(for any $A \in Ob \left( \mathcal{C}\right)$) and with braiding,
\begin{equation}
\theta_{A \otimes B}=b_{B,A} \circ \left( \theta_B \otimes \theta_A \right) \circ b_{A,B}
\nonumber
\end{equation}

\end{itemize}

One can provide a name for such a category with these morphisms:
\begin{defn}
A strict tensor category with a duality, a braiding and a twisting satisfying the above specified compatibility conditions is called a \textit{ribbon category}.
\end{defn}

Note here that in a ribbon category there is automatically also a left duality. One can check this left duality coincides with the right duality not only on objects, but also on morphisms.  

It is also possible to define weaker notions than a ribbon category. If we simply have a rigid category (that means, a category where every object has a left and a right dual satisfying Equations \ref{bicatzorro} for the right duals and analogously for the left duals) which is in addition equipped with a natural monoidal isomorphism between the identity morphism and the double dual $\left( - \right)^{\vee\vee}$ then we call it \textit{pivotal} (note the analogy with the bicategorical definition). Define left and right traces of endomorphisms $f$:
\begin{equation}
\begin{split}
\mathrm{tr}_l \left( f \right) &= \mathrm{ev} \circ \left( \mathrm{id} \otimes f \right) \circ \widetilde{\mathrm{coev}} \\
\mathrm{tr}_r \left( f \right) &=\widetilde{ev} \circ \left( f \otimes \mathrm{id} \right) \circ \mathrm{coev}
\end{split}
\nonumber
\end{equation}
If these two notions of trace coincide, then we say the category is \textit{spherical}. If one then further introduces a braiding to a spherical category (satisfying the correct compatibility conditions), then we recover the notion of ribbon category again.

At this point we are finally in a position to state the definition of a modular category. 

\begin{defn}
A \textit{modular category} is a semisimple abelian ribbon category with ground field $\mathbbm{C}$ that has only a finite number of isomorphism classes of simple objects $\lbrace V_i \rbrace_{i \in I}$ and the $\vert I \vert \times \vert I \vert$-matrix with entries $S_{ij}=\mathrm{tr} \left( b_{V_j,V_i} \circ b_{V_i,V_j} \right) \in \mathrm{End} \left( I \right) \cong \mathbbm{C}$ (called the \textit{$S$-matrix}) is invertible over $\mathbbm{C}$.
\end{defn}

Modular categories play an important role in this thesis as they show up in the context of conformal field theory -as we will see in the next chapters.

To close this section, recall an extra definition we will need in later chapters:
\begin{defn}
Let $\mathbbm{k}$ be an algebraically closed field. A \textit{fusion category} is a rigid, semisimple tensor category whose morphism spaces are finite dimensional $\mathbbm{k}$-vector spaces, with finitely many isomorphism classes of simple objects and $\mathrm{End} \left( I \right)=\mathbbm{k}.\mathrm{id}_I$. 
\end{defn}

\section{Temperley-Lieb categories}
\label{sec:TLcats}

In this section we would like to briefly introduce a category which will be of special relevance in Chapter \ref{LGCFTTL}: the Temperley-Lieb category. Our main sources for this section are \cite{abramsky}, \cite{scott}, \cite[Chapter XII]{turaev} and the upcoming \cite{davydov}.

Let $I=\left[ 0,1 \right] \subset \mathbb{R}$ be the standard interval.

\begin{defn}
\begin{itemize}
\item A \textit{plane tangle} (with $n$ inputs and $m$ outputs) consists of a finite number of strings, arcs and circles in $\mathbb{R} \times I$, the end of the strings and/or arcs being the $n$ fixed points in the line $\mathbb{R} \times 0$ and $m$ fixed points in the line $\mathbb{R} \times 1$.

\item Two plane tangles $t$, $t'$ are \textit{isotopic} if there is a smooth deformation of the first embedding into the second.
\item An isotopy class of a plane tangle is called a \textit{plane tangle diagram}.
\end{itemize}
\end{defn}

We will organize these plane tangles in a category, that we will denote as $\mathcal{R}$. The objects are labeled by natural numbers $n \in \mathbb{Z}_{\geq 0}$, and given two objects $n,m$ the morphisms are given by $\mathcal{R} \left( n,m \right)$, the class of tangle diagrams with $n$ inputs and $m$ outputs. The composition is given by vertical concatenation of the plane tangle diagrams (which requires the choice of a homeomorphism $\mu \colon I \cup I \to I$) and it is associative (up to canonical isotopy). The identity morphism is the plane tangle diagram consisting of $n$ vertical strings.

In addition, one can check that this category is (strict) monoidal: given two objects $n$ and $m$, the tensor product acts on them as $n \otimes m=n+m$. For the morphisms, it is given by horizontal concatenation of plane tangle diagrams\footnote{This requires several choices of homeomorphisms, but we refer to the literature for further details.}. The monoidal unit is $0$.

Note here that the morphism sets of $\mathcal{R}$ are graded by the number of circles: $$\mathcal{R} \left( n,m \right)=\bigcup\limits_{l \geq 0} \mathcal{R}^l \left( n,m \right)$$
Denote by $\kappa \in \mathcal{R}^1 \left( 0,0 \right)$ the class of a plane tangle consisting of one circle. Using this $\kappa$, we define $\mathcal{TR}$ to be the quotient of $\mathcal{R}$ by the relation $\kappa \otimes 1=1 \otimes \kappa$.

Actually, we are interested in categories related to this one. 
\begin{defn}
\begin{itemize}
\item $\mathcal{TL}$ is the category with objects labeled by natural number with morphism spaces $\mathcal{TL} \left( n,m \right)=\mathbbm{k} \left[ \mathcal{TR} \left( n,m \right) \right]$, where $\mathbbm{k}$ is a field.
\item Denote by $\mathbbm{k} \left( \kappa \right)$ the field of rational functions in the variable $\kappa$. The \textit{generic Temperley-Lieb category} $\mathcal{TL}_{gen}$ is the localization of $\mathcal{TL}$, where we replace the morphism spaces of $\mathcal{TL}$ by the tensor product over $\mathbbm{k} \left[ \kappa \right]$ of the morphism spaces of $\mathcal{TL}$ with $\mathbbm{k} \left( \kappa \right)$. 
\end{itemize}
\end{defn}

The endomorphisms of the generic Temperley-Lieb category are precisely some well-known algebra: the \textit{generic Temperley-Lieb algebra} $TL_{gen} \left( n \right)=\mathcal{TL}_{gen} \left( n,n \right)$ is the unital, associative $\mathbbm{k} \left( \kappa \right)$-linear algebra with generators $e_1,\ldots,e_{n-1}$ satisfying:
\begin{equation}
\begin{split}
e_i \circ e_{i \pm 1} \circ e_i &=e_i \\
e_i \circ e_i &=\kappa e_i \\
e_i \circ e_j &=e_j \circ e_i \quad \quad \quad (\mid i-j \mid > 1)
\end{split}
\label{temperleyliebgenerators}
\end{equation}

It can be shown that generic Temperley-Lieb algebras are semi-simple. This means that $\mathcal{TL}_{gen}$ can be extended to a semi-simple category if we throw in images of projectors. A formal procedure of doing it is via idempotent complete categories.

\begin{defn}
An \textit{idempotent complete category} is a category where every idempotent splits. A morphism $e$ in a category $\mathcal{C}$ is \textit{idempotent} if it satisfies that $e^2=e$, and that \textit{splitting in $\mathcal{C}$} means that $\forall e: A \rightarrow A$ (where $A \in \mathrm{Ob} \left( \mathcal{C} \right)$) an idempotent morphism, there exists a $B \in \mathrm{Ob} \left( \mathcal{C} \right)$ and two morphisms in $\mathcal{C}$ $f: A \rightarrow B$, $g: B \rightarrow A$ such that $g \circ f=e$ and $f \circ g=\mathrm{id}$. 
\end{defn}

We denote the idempotent completion of the category $\mathcal{TL}_{gen}$ as $\mathcal{TL}_{\mathrm{ic}}$, and we will call it the \textit{Temperley-Lieb category}. Then, one can prove that:
\begin{prop}{\cite{davydov}}
$\mathcal{TL}_{\mathrm{ic}}$ is semi-simple.
\end{prop}

One can describe the set of simple objects as follows. In the generic Temperley-Lieb algebras there is a morphism of particular interest:
\begin{defn}
The \textit{Wenzl-Jones idempotent} or \textit{projector} (denoted $p_n$) is the morphism in $TL \left( n \right)$ defined via the relations:
\begin{equation}
\begin{split}
p_n &\neq 0 \\
p_n \circ p_n &=p_n \\
p_n \circ e_i= e_i \circ p_n &=0 \quad \quad \quad \forall i \in \lbrace 1,\ldots,n-1 \rbrace
\end{split}
\label{WenzlJones}
\end{equation}
\end{defn}
Note here that the Wenzl-Jones idempotent, characterized by Equation \ref{WenzlJones}, is unique (for a proof see e.g. \cite{scott}). 
If we define:
\begin{defn}
Let $\mathbbm{k} \left[ \kappa \right] \subset \mathbbm{k} \left[ q^{\pm 1} \right]$ be an extension where $\kappa=q+q^{-1}$. The \textit{quantum integer} $\left[ l \right]$ is the element in the ring defined as:
\begin{equation}
\left[ l \right]:= \frac{q^l-q^{-l}}{q-q^{-1}}
\nonumber
\end{equation}
for any $l \in \mathbb{N}$, where we call $q$ the \textit{quantum parameter}. If necessary, we will specify the quantum parameter at the quantum integers with a subscript, i.e. $\left[ l \right]_q$ instead of $\left[ l \right]$.
\end{defn}

$\left[ l \right] p_l$ will be called the \textit{unnormalized Wenzl-Jones idempotent}. The Wenzl-Jones projectors can be computed via the following recurrence formula:
\begin{prop}{(Wenzl recursive formula) \cite{wenzl}}
For $n \geq 1$, the Wenzl-Jones idempotent satisfies that
\begin{equation}
p_{n+1}=p_n -\frac{\left[ n \right]}{\left[ n+1 \right]} p_n \circ e_n \circ p_n
\nonumber
\end{equation}
(with $p_1=\mathrm{id}$).
\label{wenzlrecursiveformula}
\end{prop}

\begin{rem}
Notice here that $TL_{gen} \left( n \right)$ naturally embeds into $TL_{gen} \left( n+1 \right)$. This is easy to describe in the planar tangle description: one only has to add a vertical string to the right side of the $TL_{gen} \left( n \right)$ diagram to become one in $TL_{gen} \left( n+1 \right)$. For this reason, we slightly abused notation and wrote the Wenzl recursive formula as specified -with the terms on the right hand living in $TL_{gen} \left( n+1 \right)$ instead of in $TL_{gen} \left( n \right)$ as it may suggest.
\end{rem}

\begin{example}
Using Proposition \ref{wenzlrecursiveformula}, one can compute some of the Wenzl-Jones projectors for the smallest values of n, e.g.
\begin{equation}
\begin{split}
p_1 &= \mathrm{id} \\
p_2 &= \mathrm{id}-\frac{1}{\left[ 2 \right]} e_1 \\
p_3 &= \mathrm{id}- \frac{\left[ 2 \right]}{\left[ 3 \right]}\left( e_1+e_2 \right)+\frac{1}{\left[ 3 \right]} \left( e_2 \circ e_1+e_1 \circ e_2 \right)
\end{split}
\nonumber
\end{equation}
Let us check that they satisfy Eq. \ref{WenzlJones}. It is clear that they are not zero. Concerning composition of the Temperley-Lieb generators with $p_2$ and $p_3$,
\begin{equation}
\begin{split}
p_2 \circ e_1 &=e_1 - \frac{1}{\left[ 2 \right]} e_1 \circ e_1=0 \\
p_3 \circ e_1 &=e_1- \frac{\left[ 2 \right]}{\left[ 3 \right]} e_1 \circ e_1-\frac{\left[ 2 \right]}{\left[ 3 \right]} e_2 \circ e_1+\frac{1}{\left[ 3 \right]} e_2 \circ e_1 \circ e_1+\frac{1}{\left[3 \right]} e_1 \circ e_2 \circ e_1 =0 \\
p_3 \circ e_2 &=e_2- \frac{\left[ 2 \right]}{\left[ 3 \right]} e_1 \circ e_2-\frac{\left[ 2 \right]}{\left[ 3 \right]} e_2 \circ e_2+\frac{1}{\left[ 3 \right]} e_2 \circ e_1 \circ e_2+\frac{1}{\left[3 \right]} e_1 \circ e_2 \circ e_2 =0
\end{split}
\nonumber
\end{equation}
and analogously for $e_1 \circ p_2$, $e_1 \circ p_3$ and $e_2 \circ p_3$, using Eq. \ref{temperleyliebgenerators} and the identity $\left[ 3 \right]-\left[ 2 \right]^2+1=0$. Concerning idempotency, $p_1$ is clearly idempotent; for $p_2$ and $p_3$ this follows quite straightforward from the previous computation:
\begin{equation}
\begin{split}
p_2 \circ p_2 &=p_2 -\frac{1}{\left[ 2 \right]} p_2 \circ e_1=p_2
\end{split}
\nonumber
\end{equation}
Analogously for $p_3$:
\begin{equation}
\begin{split}
p_3 \circ p_3 &=p_3+\frac{\left[ 2 \right]}{\left[ 3 \right]^2} p_3 \circ \left( e_1+e_2 \right) +\frac{1}{\left[ 3 \right]} p_3 \circ \left( e_2 \circ e_1 +e_1 \circ e_2 \right)=p_3
\end{split}
\nonumber
\end{equation}
\end{example}

Consider the objects $T_i \in \mathcal{TL}_{\mathrm{ic}}$ ($i \in \mathbb{Z}_{\geq 0}$) defined as the images of the Wenzl-Jones projectors 
\begin{equation}
T_i=\mathrm{im} \left( p_i \right).
\end{equation}

Actually,
\begin{lem}{\cite{davydov}}
The set of simple objects of $\mathcal{TL}_{\mathrm{ic}}$ has the form $\lbrace T_i \vert i \geq 0 \rbrace$.
\end{lem}

These $T_i$s help to describe some property of $\mathcal{TL}_{\mathrm{ic}}$ --but we need to first introduce some definitions.

\begin{defn}
\begin{itemize}
\item We say that a tensor category $\C$ is {\em freely generated by} $X\in \mathrm{Ob}\left(\C \right)$ together with a collection of morphisms $\{ f_j:X^{\ot n_j}\to X^{\ot m_j} \}$ making a collection of diagrams $D_s$ commutative if for any tensor category $\D$ the functor of taking values 
$$\Fun_\ot(\C,\D)\to \D',\qquad F\mapsto F(X)$$
is an equivalence. Here, $\Fun_\ot(\C,\D)$ is the category of tensor functors (with tensor natural transformations as morphisms). The target $\D'$ is the category with objects $(Y,\{g_j\})$, where $Y\in\D$ and the $g_j:Y^{\ot n_j}\to Y^{\ot m_j}$ make the collection of diagrams $D_s$, with $X$ replaced by $Y$
and $f_j$ by $g_j$, 
commutative in $\D$. Morphisms $(Y,\{g_j\})\to (Y',\{g'_j\})$ in $\D'$ are morphisms $Y\to Y'$ in $\D$ fitting into commutative squares with all $g_j,g'_j$.

\item We call an object $T$ of a tensor category $\C$ {\em self-dual} if it comes equipped with morphisms
$$
	n:I\to T \otimes T~~ ,\qquad u:T \otimes T \to I \ ,
$$ 
such that the diagrams 
$$\xymatrix{T \ar[rrr]^-{\id} \ar[d]_-{\lambda^{-1}_T} &&& T\\ I\ot T\ar[r]^-{n\ot \id} & (T\ot T)\ot T \ar[r]^-{\alpha_{T,T,T}^{-1}} & T\ot(T\ot T)\ar[r]^-{\id\ot u} & T\ot I \ar[u]_-{\rho_T}}$$
\beq\label{eq:zigzag}
\xymatrix{T \ar[rrr]^-{\id} \ar[d]_-{\rho^{-1}_T} &&& T\\ T\ot I\ar[r]^-{\id\ot n} & T\ot (T\ot T) \ar[r]^-{\alpha_{T,T,T}} & (T\ot T)\ot T\ar[r]^-{u\ot\id} & I\ot T \ar[u]_-{\lambda_T}} 
\eeq
commute. 
The scalar $\kappa\in \kk$ defined by the composition $\kappa\, \id=u\circ n \colon I\to I$ is called the (self-dual) {\em dimension} of $T$.
\end{itemize}
\end{defn}

Constructing such $n$ and $u$ maps for the $T_i$'s, it is possible to check that the dimension of the $T_n$ (which can also be computed as the trace of $p_n$) is $\mathrm{dim}(T_n) = [n+1]_q$. 

At this point, recall that $\mathcal{TL}_{gen}$ has a tensor product originally induced from $\mathcal{R}$. Actually, $\mathcal{TL}_{\mathrm{ic}}$ also has a tensor product, induced from $\mathcal{TL}_{gen}$. Then, we can state that:
\begin{prop}{\cite{davydov}}
$\mathcal{TL}_{\mathrm{ic}}$ is the $\mathbbm{k} \left( \kappa \right)$-linear tensor category freely generated by the self-dual object $T=T_1$.
\end{prop}

Now choose $q \in \mathbbm{k}$. For $q$ a root of unity of order $> 2$, the last well-defined Wenzl--Jones projector is $p_{d-1}$, where $d$ is the order of $q$ if it is odd and half the order of $q$ if it is even. In this case the category $\mathcal{TL}_{\mathrm{ic}}$ has a maximal fusion quotient $\mathcal{T}_\kappa$ which can be defined as the quotient $$\mathcal{T}_\kappa ~:=~ \mathcal{TL}_{\mathrm{ic}}/\langle p_{d-1}\rangle$$ by the ideal of morphisms tensor generated by the Wenzl--Jones projector $p_{d-1}\in TL_{d-1}(\kappa)$, see \cite{eo}. 

\begin{thm}[\cite{gw}]
$\langle p_{d-1} \rangle$ is the unique tensor ideal in $\mathcal{TL}_{\mathrm{ic}}$.
\end{thm}

One can restate this result as follows: any non-faithful tensor functor $\mathcal{TL}_{\mathrm{ic}}\to\mathcal{D}$ (where $\mathcal{D}$ is a tensor category) factors through $\mathcal{T}_\kappa\to\D$. 
Thus we have the following.
\bth
A tensor functor from $\T_\kappa$ to a tensor category $\D$ is determined by a self-dual object of dimension $\kappa$ in $\D$ with vanishing Wenzl--Jones projector $p_{d-1}$. 
\eth

\begin{rem}
The condition on the Wenzl-Jones idempotent to vanish is a condition on the self-dual object $T$: $p_{d-1}$ is simply some combination of self-duality morphisms and tensor products of them, so it reduces to be only some constraint on $T$. 
\end{rem}

The next corollary provides an easy-to-use replacement for the vanishing
condition on the Wenzl--Jones projector.

Simple objects of $\T_\kappa$ are $T_i,\ i=0,...,d-2$
with $T_0=I, T_1=T$, and the tensor product with $T$ is $T\ot T_i\simeq
T_{i-1}\op T_{i+1}$ for $0<i<d-2$ and $T\ot T_{d-2}\simeq T_{d-3}$. Then,

\bco\lb{ftl}
Let $\D$ be a rigid fusion category with simple objects $S_i,\
i=0,...,d-2$ and the tensor product  $S_1\ot S_i\simeq S_{i-1}\op
S_{i+1}$ for $0<i<d-2$ and $S_1\ot S_{d-2}\simeq S_{d-3}$. A tensor
functor $\TL_{\mathrm{ic}} \to \D$ such that $T_i \mapsto S_i$ factors through
$\T_\kappa$.
\eco

\bpf
The non-faithfulness of the tensor functor is manifest since
$\TL_{\mathrm{ic}}(T\ot T_{d-2},T\ot T_{d-2})$ is 2-dimensional, while
$\D(S_1\ot S_{d-2},S_1\ot S_{d-2})$ is only 1-dimensional.
\epf

\chapter{Matrix factorizations}
\label{ch:mfs}

\section{Basic definitions}

In this chapter we will describe our object of main interest, matrix factorizations, and some of their related algebraic structures, as well as an overview of how they arise in the physics framework. Some basic bibliography for matrix factorizations we use are \cite{eisenbud,buchweitz,yoshinobuch,khovroz,yoshino}.

Let $\mathbbm{k}$ be a field, $\mathbf{x}=\lbrace x_1,\ldots,x_n \rbrace$ a finite set of variables, $S=\mathbbm{k} \left[ x_1,\ldots,x_n \right]$ a polynomial ring. Consider a polynomial $W \in S$.
\begin{defn}
\begin{itemize}
\item The \textit{Jacobi ideal} is an ideal in $S$ generated by the partial derivatives $\frac{\partial W}{\partial x_1},\ldots,\frac{\partial W}{\partial x_n}$. Denote it as: $$\mathrm{Jac} \left( W \right):= \langle \frac{\partial W}{\partial x_1},\ldots,\frac{\partial W}{\partial x_n} \rangle$$
\item The \textit{Jacobi ring} is the quotient ring $S / \mathrm{Jac} \left( W \right)$.
\item $W$ is called a \textit{potential} if the Jacobi ring is finite-dimensional.
\end{itemize}
\end{defn}

Depending on the literature, a potential can also be called \textit{superpotential}. Some other property we may require for the potential is that:

\begin{defn}
(For the case $\mathbbm{k}=\mathbbm{C}$) We say a polynomial is \textit{homogeneous} if there exist $\omega_1,\ldots,\omega_n \in \mathbb{Q}_{\geq 0}$ such that $$W \left( \lambda^{\omega_1}x_1,\ldots,\lambda^{\omega_n}x_n \right)=\lambda^2 W \left( x_1,\ldots,x_n \right)$$ for any $\lambda \in \mathbbm{C}^\times$.

\end{defn}

The most important definition of the present thesis is the following:
\begin{defn}
A \textit{matrix factorization} of a potential $W$ consists of a pair $\left( M, d^M \right)$ where
\begin{itemize}
\item $M$ is a $\mathbb{Z}_2$-graded free $S$-module;
\item $d^M: M \rightarrow M$ degree 1 $S$-linear endomorphism (the \textit{twisted differential}) such that $$d^M \circ d^M=W.\mathrm{id}_M.$$ (where the right hand side of the equation stands for the endomorphism $m \mapsto W.m$, $\forall m \in M$).
\end{itemize}
We may display the $\mathbb{Z}_2$-grading explicitly as $M=M_0 \oplus M_1$, $d^M=\left( \begin{matrix} 0 & d_1^M \\ d_0^M & 0 \end{matrix} \right)$ or graphically as:
\beq
M~:\quad
\xymatrix{M_1 \ar@/^10pt/[rr]^{d_1^M} && M_0 \ar@/^10pt/[ll]^{d_0^M}}\quad 
\nonumber
\eeq
\label{def:mf}
\end{defn}

\begin{example}
Let $S=\mathbbm{C} \left[ x \right]$ and $W=x^d$. One of the simplest matrix factorizations is given by $\left( \mathbbm{C} \left[ x \right]^{\oplus 2}, M_m \right)$ with $M_m=\left( \begin{matrix} 0 & x^m \\ x^{d-m} & 0 \end{matrix} \right)$ with $0 \leq m \leq d$.
\end{example}

If there is no risk of confusion, we will denote $\left( M,d^M \right)$ simply as $M$. We say that a matrix factorization is of \textit{finite rank} if its underlying free $S$-module is of finite rank. 

Given two matrix factorizations, we define a morphism between them as follows.

\begin{defn}
Given two matrix factorizations $\left( M,d^M \right)$, $\left( M',d^{M'} \right)$ of a potential $W$, a \textit{morphism of matrix factorizations} $f:\left( M,d^M \right) \rightarrow \left( M',d^{M'} \right)$ is a $S$-linear morphism.
\end{defn}

With this data we can construct the following categories of matrix factorizations. 
\begin{itemize}
\item $\MF_{S,W}$ is the category whose objects are matrix factorizations of $W$ and whose morphisms are morphisms of matrix factorizations. Notice that each morphism space of $\MF_{S,W}$ is a $\mathbb{Z}_2$-graded complex with differential: $$\delta f=d^{M'} \circ f- \left( -1 \right)^{|f|} f \circ d^M$$ where $\vert f \vert$ denotes the degree of $f$, for $f \in \MF_{S,W} \left( M, M' \right)$. One can check that it is a differential via double composition of $\delta$:
\begin{equation}
\begin{split}
\delta \circ \delta f &= d^{M'} \circ \delta f - \left( -1 \right)^{| \delta f|} \delta f \circ d^M  \\ &=d^{M'} \circ d^{M'} \circ f -\left( -1 \right)^{|f|} d^{M'} \circ f \circ d^M + \left( -1 \right)^{|f|} d^{M'} \circ f \circ d^M -f \circ d^M \circ d^M \\ &=W.\mathrm{id}_{M'} \circ f-f \circ W.\mathrm{id}_M .
\end{split}
\nonumber
\end{equation}
These last two terms cancel together because,
\begin{equation}
\begin{split}
W.\mathrm{id}_{M'} \circ f \left( m \right) &=W.f \left( m \right) \\
f \circ W.\mathrm{id}_M \left( m \right) &=f \left( W.m \right)=W.f\left( m \right)
\end{split}
\nonumber
\end{equation} 
for any $m \in M$.
\item $\mathrm{ZMF}_{S,W}$ is the category whose objects are the same as $\mathrm{MF}_{S,W}$ and whose morphisms, for any two objects $M$ and $N$, are those morphisms from $M$ to $M'$ which are of degree zero and which lie in the kernel of $\delta$: $$\ZMF_{S,W}(M,M') = \{ f : M \to M' | \text{ $f$ is $S$-linear of degree 0 and $\delta(f)=0$ } \} \ $$
\item $\HMF_{S,W}$ is the category whose objects are the same as those of $\ZMF_{S,W}$ and whose morphisms, given two objects $M$ and $M'$, are those of $\ZMF_{S,W}(M,M')$ mod those morphisms which are the image of the differential of morphisms of degree 1 from $M$ to $M'$:
\begin{equation}
	\HMF_{S,W}(M,M') = \ZMF_{S,W}(M,M') / \{  \delta(g) | \text{ $g : M \to M'$ is $S$-linear of degree 1 } \} \ 
	\nonumber
\end{equation}
\end{itemize}

\begin{example}
Following up with the previous example, if we have two matrix factorizations $\left( \mathbbm{C} \left[ x \right]^{\oplus 2}, M_m \right)$ and $\left( \mathbbm{C} \left[ x \right]^{\oplus 2}, M_l \right)$ of the potential $W=x^d$, the morphism space in $\mathrm{MF}_{S,W}$ between them is given by the set of $2 \times 2$ matrices with entries in $\mathbbm{C} \left[ x \right]$. $\ZMF \left( M_m,M_l \right)$ needs to distinguish two cases:
\begin{itemize}
\item If $l \geq m$, $\ZMF \left( M_m,M_l \right)=\lbrace \left( \begin{matrix}
a & 0 \\ 0 & a x^{l-m} \end{matrix} \right) \vert a \in \mathbbm{C} \left[ x \right] \rbrace$, and
\item If $l < m$, $\ZMF \left( M_m,M_l \right)=\lbrace \left( \begin{matrix}
a  x^{m-l} & 0 \\ 0 & a \end{matrix} \right) \vert a \in \mathbbm{C} \left[ x \right] \rbrace$
\end{itemize}
Finally, the morphism space in $\HMF_{S,W}$ between $M_m$ and $M_l$ is isomorphic to $\mathbbm{C} \left[ x \right] / \langle x^i \rangle$ where $i$ is the smallest element among the set $\lbrace l,m,d-l,d-m \rbrace$. 
\end{example}

Notice here that, from the definitions we just gave, the rank of the module of a matrix factorization can be either of finite or infinite. We use lower case letters ($\mathrm{mf}_{S,W}$, $\mathrm{zmf}_{S,W}$, $\mathrm{hmf}_{S,W}$) whenever we refer to the full subcategory of objects that are isomorphic, in their respective categories, to finite rank matrix factorizations. For $\mathrm{mf}_{S,W}$ and $\mathrm{zmf}_{S,W}$ this just means that one restricts to finite rank matrix factorizations. However, in $\mathrm{hmf}_{S,W}$ there are isomorphisms (in $\mathrm{hmf}$, so up to homotopy) between finite and infinite rank matrix factorizations.

Following up with the graphical display of the $\mathbb{Z}_2$-grading of a matrix factorization, we will often write morphisms $f \in \ZMF_{S,W}(M,M')$ (or representatives of classes in $\HMF_{S,W}(M,M')$) in a diagram as follows:
\beq
  \xymatrix{M_1 \ar@/^10pt/[rr]^{d_1^M} \ar[dd]_{f_1} && M_0 \ar@/^10pt/[ll]^{d_0^M}  \ar[dd]^{f_0} \\ \\
N_1 \ar@/^10pt/[rr]^{d_1^{N}} && N_0 \ar@/^10pt/[ll]^{d_0^{N}}
}
\nonumber
\eeq
That $f$ is in $\ZMF_{S,W}(M,M')$ is equivalent to $f_0$ and $f_1$ being $S$-linear maps such that the subdiagram with upward curved arrows commutes and that with downward curved arrows commutes: 
\beq 
f_0 \circ d_1^M = d_1^N \circ f_1 \quad , \quad 
f_1 \circ d_0^M = d_0^N \circ f_0 \ .
\nonumber
\eeq

Note also that in the definition we have considered left modules -a definition for right modules is analogous. In particular, we would like to use bimodules. Let $\left( S,W \right)$ and $\left( S',W' \right)$ be two pairs of rings and potentials, and consider an $S$-$S'$-bimodule -or, an $S \otimes_{\mathbbm{k}} S'$-left module\footnote{Actually, it would be more correct to write $S \otimes_\mathbbm{k} S'^{op}$, but as our rings are commutative, we omit this notation.}. Recall that an $S$-$S'$-bimodule is free if the corresponding $\left( S \otimes_\mathbbm{k} S' \right)$-left module is free. With this bimodule we construct a matrix factorization of the potential $W \otimes 1-1 \otimes W'$ \footnote{If not in risk of confusion, we may sometimes write $W-W'$ instead of $W \otimes 1-1 \otimes W'$.}, and in order to point out the difference between a matrix factorization whose module is a bimodule instead of a left module we will call it a \textit{matrix bifactorization}.

\begin{example}{\cite{brunrogg1}}
An example of a matrix bifactorization which will be later of particular interest is the following: let $\left( S,W \right)=\left( \mathbb{C} \left[ x \right],x^d \right)$, $\left( S',W' \right)=\left(\mathbbm{C} \left[ y \right],y^d \right)$. A \textit{permutation-type matrix bifactorization} of $W-W'=x^d-y^d$ is a matrix bifactorization of the shape $\left( \mathbbm{C} \left[ x,y \right]^{\oplus 2}, P_J \right)$ with:
\begin{equation}
P_J=\left( \begin{matrix}
0 & \prod\limits_{i \in J} \left( x-\eta^i y \right) \\ \prod\limits_{i \in \overline{J}=\lbrace 0,\ldots,d-1 \rbrace \setminus J} \left( x-\eta^i y \right) & 0
\end{matrix} \right)
\nonumber
\end{equation}
where $J \subset \lbrace 0,\ldots,d-1 \rbrace$ and $\eta=e^{\frac{2 \pi i}{d}}$ a primitive $d$-th root of unity.
\label{permtypembf}
\end{example}

We will denote categories of matrix bifactorizations as in the case of matrix factorizations but adding a $_{\mathrm{bi}}$ subscript: $\mathrm{MF}_{\mathrm{bi}; (S,W),(S'W')}$ (or simply $\MFbi$), and resp. $\mathrm{ZMF}_{\mathrm{bi}; (S,W),(S',W')}$ (or $\ZMFbi$) and $\mathrm{HMF}_{\mathrm{bi}; (S,W),(S'W')}$ (or $\HMFbi$). Analogously for $\mathrm{mf}_{\mathrm{bi}; (S,W),(S'W')}$ (or simply $\mfbi$), $\mathrm{zmf}_{\mathrm{bi}; (S,W),(S'W')}$ (or $\zmfbi$) and $\mathrm{hmf}_{\mathrm{bi}; (S,W),(S'W')}$ (or simply $\hmfbi$)

\subsection{The bicategory of Landau-Ginzburg models}

At this point, we would like to go one step further concerning categories of matrix factorizations. The purpose of this subsection is to construct a bicategory whose morphism categories are categories of matrix factorizations, which we will denote as $\mathcal{LG}$. We will take:
\begin{itemize}
\item Objects: pairs $\left( S,W \right)$ where $S$ is a polynomial ring over a fixed field $\mathbbm{k}$ with an arbitrary finite number of variables and $W \in S$ a potential.
\item 1- and 2-morphisms: for any two objects $\left( S,W \right)$, $\left( S',W' \right)$, the 1- and 2-morphisms between them are given by $\hmf_{\mathrm{bi};(S,W),(S',W')}^\omega$ \footnote{The superscript $^\omega$ means we take the idempotent completion of $\hmf_{\mathrm{bi}}$: $\mathrm{hmf}_{S,W}$ is not necessarily idempotent complete, but its full subcategory whose objects are those matrix factorizations which are direct summands of finite-rank matrix factorizations is. The reason for taking idempotent completions is that when composing 1-morphisms in $\hmf$ the resulting matrix factorization is not a priori finite-rank, only a summand in the homotopy category of something finite-rank. We thus obtain an object which doesn't belong to this category --so taking the idempotent completion we recover an object in $\hmf$. This seems to be the less technical approach to solve this problem, with the advantage that the formalism also applies to graded rings and graded matrix factorizations - the other approach would be working throughout with power series rings and completed tensor products. Some more details on idempotent complete categories can be found at \cite{neeman}, and for an example of $\mathrm{hmf}_{S,W}$ not being idempotent complete we refer to e.g. Example A.5 in \cite{kmvdb}.}.
\end{itemize}

Next we would like to describe the tensor product of matrix factorizations. Let $S_1$, $S_2$ and $S_3$ be three polynomial rings, $W_1 \in S_1$, $W_2 \in S_2$ and $W_3 \in S_3$ three potentials, and two matrix bifactorizations $\left( B, d^B \right)\in \mathrm{HMF}_{\mathrm{bi}; (S_1,W_1),(S_2,W_2)}$, $\left( B', d^{B'} \right) \in \mathrm{HMF}_{\mathrm{bi}; (S_2,W_2),(S_3,W_3)}$. From this, we define the \textit{tensor product matrix (bi)factorization} $$\left( B' \otimes B,d^{B' \otimes B} \right) \in \mathrm{HMF}_{\mathrm{bi}; (S_1,W_1),(S_3, W_3)}$$  in terms of its underlying $\left( S_3 \otimes_{\mathbbm{k}} S_1 \right)$-module 
\begin{equation}
\left( \left( B'_0 \otimes_{S_2} B_0 \right) \oplus \left( B'_1 \otimes_{S_2} B_1 \right) \right)\oplus \left( \left( B'_1 \otimes_{S_2} B_0 \right) \oplus \left( B'_0 \otimes_{S_2} B'_1 \right) \right)
\label{tensorproddirsum}
\end{equation}

with differential $d_{B' \otimes B}=d_{B'} \otimes \mathrm{id}_B+\mathrm{id}_{B'} \otimes d_B$.

\begin{rem}
Whenever $S_2 \neq \mathbbm{k}$ $B' \ot B$ is an infinite-rank matrix factorization over $S_3 \otimes_{\mathbbm{k}} S_1$. However, as it was proved in \cite{dm}, $B' \otimes B$ is (naturally isomorphic to) a direct summand of some finite-rank matrix factorizations in $\mathrm{hmf}_{\mathrm{bi}; (S_1,W_1),(S_3, W_3)}$. 
\end{rem}

Note here that we are dealing with tensor products of graded morphisms, and that we choose to make an explicit use of the Koszul sign rule. That means: for $b' \in B'$ and $b \in B$ we have $\left( \mathrm{id}_{B'} \ot d^B \right)\left( b' \ot b \right)=\left( -1 \right)^{|b'|} b' \ot d^B \left( b \right)$.
For this reason, $d^{B' \otimes B}$ has such a simple expression -otherwise it would be a bit more complicated. Let us prove that it is indeed a differential:
\begin{lem}
$d^{B' \otimes B} \circ d^{B' \otimes B}=W_3.\mathrm{id}_{B'} \otimes \mathrm{id}_B-\mathrm{id}_{B'} \otimes \mathrm{id}_B.W_1$
\end{lem}
\begin{proof}
\begin{equation}
\begin{split}
d^{B' \otimes B} \circ d^{B' \otimes B} &=\left( d_{B'} \otimes \mathrm{id}_B+\mathrm{id}_{B'} \otimes d_B \right) \circ \left( d_{B'} \otimes \mathrm{id}_B+\mathrm{id}_{B'} \otimes d_B \right) \\ &=\left( d_{B'} \otimes \mathrm{id}_B \right) \circ \left( d_{B'} \otimes \mathrm{id}_B \right)+\left( d_{B'} \otimes \mathrm{id}_B \right) \circ \left( \mathrm{id}_{B'} \otimes d_B \right) \\ &+ \left( \mathrm{id}_{B'} \otimes d_B \right) \circ \left( d_{B'} \otimes \mathrm{id}_B \right)+\left( \mathrm{id}_{B'} \otimes d_B\right) \circ \left( \mathrm{id}_{B'} \otimes d_B \right) \\ 
&=W_3.\mathrm{id}_{B'} \ot \mathrm{id}_B - \mathrm{id}_{B'}.W_2 \otimes \mathrm{id}_B
+d^{B'} \otimes d^B \\ &-d^{B'} \otimes d^B+
\mathrm{id}_{B'} \otimes W_2.\mathrm{id}_B - \mathrm{id}_{B'} \otimes \mathrm{id}_B.W_1 \\ &=W_3.\mathrm{id}_{B'} \otimes \mathrm{id}_B-\mathrm{id}_{B'} \otimes \mathrm{id}_B.W_1
\end{split}
\nonumber
\end{equation} where the $W_2$ terms cancel with each other because we are tensoring over $S_2$. Because of this, we can regard the action with $W_2$ from the left side of the tensor product of identity maps as an action from the right side of the tensor product -hence, both terms cancel.
\end{proof}

The tensor product of morphisms of matrix factorizations is defined as follows: let the matrix bifactorizations $B_1,B_2 \in \mathrm{HMF}_{\mathrm{bi}; (S_1,W_1),(S_2,W_2)}$, $B'_1,B'_2\in \mathrm{HMF}_{\mathrm{bi}; (S_2,W_2),(S_3,W_3)}$, and two morphisms $f \colon B_1 \to B_2$ and $g \colon B'_1 \to B'_2$. We define the \textit{tensor product of morphisms of matrix factorizations}
$$g \otimes f \colon B'_1 \otimes B_1 \to B'_2 \otimes B_2$$ in $\mathrm{HMF}_{\mathrm{bi}; (S_1,W_1),(S_3,W_3)}$\footnote{It is easy to see that the tensor product of morphisms belongs to this category. Simply recall that (following the previous notation) for $g \otimes f \in \mathrm{HMF}_{\mathrm{bi}; (S_1,W_1),(S_3,W_3)}$, the differential of the morphism space is given by: $$\delta \left( h \right)=d^{B'_2 \otimes B_2} \circ \left( g \otimes f \right) -\left( -1 \right)^{|g \otimes f|} \left( g \otimes f \right) \circ d^{B'_1 \otimes B_1}.$$}

The above defined tensor product of matrix factorizations and morphisms of matrix factorizations will describe the composition of 1- and 2-morphisms in our bicategory.

We still have to specify some more data for our bicategory:
\begin{itemize}
\item[$\circ$] For $B \in \mathrm{HMF}_{\mathrm{bi}; (S_1,W_1),(S_2,W_2)},B' \in \mathrm{HMF}_{\mathrm{bi}; (S_2,W_2),(S_3,W_3)},B'' \in \mathrm{HMF}_{\mathrm{bi}; (S_3,W_3),(S_4,W_4)}$, the associator is the 2-isomorphism $$\alpha_{B'',B',B}: B'' \otimes \left( B' \otimes B \right) \rightarrow \left( B'' \otimes B' \right) \otimes B$$ which is given by the isomorphism of bimodules
$$ b'' \otimes \left(b' \otimes b  \right) \mapsto \left( b'' \otimes b' \right) \otimes b$$
for $b \in B$, $b' \in B'$ and $b'' \in B''$.
\item Concerning the unit 1-morphism, let us write $S^e=S \otimes_\mathbbm{k} S$ and $\tilde{W}=W \otimes 1-1 \otimes W \in S^e$. If $n$ is the number of variables in the ring of polynomials $S$, we fix $n$ formal symbols $\theta_i$ as a basis of $\left( S^e \right)^{\oplus n}$. Then the $S^e$-module underlying $I_W \in \mathrm{hmf}_{S^e,\tilde{W}}$ is the exterior algebra $$\Delta_W=\bigwedge \left( \bigoplus_{\substack{i=1}}^n S^e \theta_i \right)$$ on which the differential is given by $$d_{\Delta_W}=\sum\limits_{i=1}^n \left( \left( x_i-x'_i\right) \theta^*_i+\partial_{[i]}W \cdot \theta_i \wedge \left( - \right)\right)$$ where $\partial_{[i]}W=W \left( \left( x'_1,\ldots,x'_{i-1},x_i,\ldots,x_n \right)-W\left(x'_1,\ldots,x'_i,x_{i+1},\ldots,x_n \right) \right)/\left( x_i-x'_i \right)$ and $\theta_i^*$ is a derivation on $\Delta_W$ satisfying that $\theta^*_i \left( \theta_j \right)=\delta_{i,j}$. The endomorphisms of $I_W$ in $\mathrm{hmf}_{S^e,\tilde{W}}$ are given by $\mathrm{Jac} \left( W \right)$, see e.g. \cite{kapustinrozansky}.
\item The left and right unit isomorphisms on $B \in \mathrm{hmf}_{bi; (S_1,W_1),(S_2,W_2)}$, $\lambda_B \colon I_{W_1} \otimes B \to B$ and $\rho_B \colon B \otimes I_{W_2} \to B$ which are the composition of first projecting $I$ to its $\theta$-degree zero component and then using the multiplication in the rings $S_1$ and $S_2$ resp. More explicitly, set e.g. $S=\mathbbm{C} \left[ x \right]$. In components -- recall the direct sum decomposition of the components of the matrix factorizations from Definition \ref{def:mf} and Equation \ref{tensorproddirsum}--, $\lambda_M$ and $\rho_M$ look like:
$$\lambda_M=\left( \begin{matrix} L_{M_0} & 0 & 0 & 0 \\ 0 & 0 & 0 &L_{M_1}\end{matrix} \right), \quad \quad \rho_M=\left( \begin{matrix} R_{M_0} & 0 & 0 & 0 \\ 0 & 0 & R_{M_1} & 0 \end{matrix}\right)$$
where the maps $L_{M_i}$ and $R_{M_i}$ ($i \in \lbrace 0,1 \rbrace$) are defined, for a given $\mathbbm{C} \left[ x \right]$-$\mathbbm{C} \left[ x \right]$-bimodule $B$, as
\begin{equation}
\begin{split}
L_B \colon \mathbbm{C} \left[ x,y \right] \ot_{\mathbbm{C} \left[ x \right]} B & \to B \\
f \left( x,y \right) \ot b & \mapsto f \left( x,x \right).b \\
R_B \colon B \ot_{\mathbbm{C} \left[ x \right]} \mathbbm{C} \left[ x,y \right] & \to B \\
b \otimes f \left( x,y \right) & \mapsto b.f \left( x,x \right)
\end{split}
\label{leftrightactions}
\end{equation}
In $L_B$ $\mathbb{C} \left[ x \right]$ acts on $\mathbb{C} \left[ x,y \right]$ via multiplication by $y$ and in $R_B$ via multiplication by $x$.
\end{itemize}

\begin{rem}
$\alpha$ is an isomorphism of free modules, but $\lambda$ and $\rho$ are only invertible up to homotopy.
\end{rem}

Altogether, we have specified all the necessary data of the bicategory $\mathcal{LG}$ of Landau-Ginzburg models. The coherence axioms are indeed satisfied and this was checked in \cite{macnamee}, \cite{carqrunkel1}. Notice here this means that $\mathcal{LG} \left( \left( S,W \right),\left( S,W \right) \right)$, i.e. $\hmfbi$, is monoidal. However, the unit object in the category of $\bZ_2$-graded $S$-$S$-bimodules, the bimodule $S$ is not free as an $S \otimes_k S$-left module. As a consequence, the categories $\MFbi$ and $\ZMFbi$ are non-unital monoidal.

At this point, we would like to describe further structures inside $\mathcal{LG}$, namely adjunctions and duals. It was recently proved in \cite{carqmurfet} that:
\begin{thm}
Every 1-morphism in $\mathcal{LG}$ has both a left and a right adjoint. Specifically, if a 1-morphism is represented by a finite-rank matrix factorization $X$ of $V-W$, where $W \in S_1=\mathbbm{k} \left[x_1,\ldots,x_n \right]$ and $V \in S_2=\mathbbm{k} \left[ z_1,\ldots,z_m \right]$ are potentials, then
$$X^\dagger =S_1 \left[ n \right] \otimes_{S_1} X^\vee$$
$${^\dagger X}=X^\vee \otimes_{S_2} S_2 \left[m \right],$$
where $X^\vee:=\mathrm{Hom}_{\mathbbm{k} \left[ x_1,\ldots,x_n,z_1,\ldots,z_m \right]} \left( X,\mathbbm{k} \left[ x_1,\ldots,x_n,z_1,\ldots,z_m \right] \right)$ and $\left[ n \right]$ is the $n$-fold shift functor (which exchanges $d_0^X$ and $d_1^X$ $n$ times), are resp. the right and left adjoints of $X$ in $\mathcal{LG}$.

The evaluation and coevaluation morphisms act on the elements of the basis of $X$ $\lbrace e_i \rbrace$ and $\gamma \in \Delta_W$:
\begin{equation}
\begin{split}
\widetilde{\mathrm{ev}}_X \left( e_j \otimes e_i^* \right) &= \sum\limits_{l>0} \sum\limits_{a_1 < \ldots < a_l} \left( -1 \right)^{l+(n+1)|e_j|}\theta_{a_1} \ldots \theta_{a_l} \mathrm{Res} \left[ \frac{\lbrace \partial_{[a_l]}^{z,z'} d_X \ldots \partial_{[a_1]}^{z,z'} d_X \Lambda^{(x)} \rbrace_{ij} dx }{\partial_{x_1} W,\ldots,\partial_{x_n} W} \right] \\
\mathrm{ev}_X \left( e_i^* \otimes e_j \right) &= \sum\limits_{l>0} \sum\limits_{a_1 < \ldots < a_l} \left( -1 \right)^{\binom{l}{2}+ l|e_j|}\theta_{a_1} \ldots \theta_{a_l} \mathrm{Res} \left[ \frac{\lbrace \Lambda^{(z)} \partial_{[a_1]}^{x,x'} d_X \ldots \partial_{[a_l]}^{x,x'} d_X  \rbrace_{ij} dz }{\partial_{z_1} V,\ldots,\partial_{z_m} V} \right] \\
\widetilde{\mathrm{coev}}_X \left( \bar{\gamma} \right) &= \sum_{i,j} \left( -1 \right)^{(\bar{r}+1)|e_j|+s_n} \lbrace \partial_{[\bar{b}_{\bar{r}}]}^{x,x'} \left( d_X \right) \ldots \partial_{[\bar{b}_{\bar{1}}]}^{x,x'} \left( d_X \right) \rbrace_{ji} e_i^* \otimes e_j \\
\mathrm{coev}_X \left( \gamma \right) &= \sum_{i,j} \left( -1 \right)^{\binom{l}{2}+m r+s_m} \lbrace \partial_{[b_1]}^{z,z'} \left( d_X \right) \ldots \partial_{[b_r]}^{z,z'} \left( d_X \right) \rbrace_{ij} e_i \otimes e_j^*
\end{split}
\nonumber
\end{equation}
where $\Lambda^{(x)}:=\left( -1 \right)^n \partial_{x_1} d_X \ldots\partial_{x_n} d_X$, $\Lambda^{(z)}:=\left( -1 \right)^n \partial_{z_1} d_X \ldots\partial_{z_m} d_X$ and $b_i$, $\bar{b}_{\bar{j}}$ and $s_n$, $s_m \in \mathbb{Z}_2$ are uniquely determined by requiring that $b_1 < \ldots < b_r$, $\bar{b}_1 < \ldots < \bar{b}_{\bar{r}}$ and $\bar{\gamma} \theta_{\bar{b}_1} \ldots \theta_{\bar{b}_{\bar{r}}}=\left( -1 \right)^{s_n} \theta_1 \ldots \theta_n$ and $\gamma \theta_{b_1} \ldots \theta_{b_r}=\left( -1 \right)^{s_m} \theta_1 \ldots \theta_m$\footnote{A full, detailed description of these maps can be found in \cite{carqmurfet}}. 
\label{cmtheorem}
\footnote{For convenience in our computations in Chapter \ref{LGCFTTL}, we will use the formulas of the evaluation and coevaluation maps for one variable (i.e. $m=n=1$) from \cite{carqrunkel3}.}
\end{thm}

These formulas from Theorem \ref{cmtheorem} arose by approaching the problem of finding the adjoints using the so-called \textit{associative Atiyah classes}. A complete description and proofs of Theorem \ref{cmtheorem} and Proposition \ref{pivotalityLG} are given in \cite{carqmurfet}. In addition, this approach also served to find an expression for the inverses of the unit morphisms:

\begin{thm}
Following the notation of Theorem \ref{cmtheorem}, the unit morphisms have the following inverses:
\begin{equation}
\begin{split}
\lambda_X^{-1} \left( e_i \right) &=\sum\limits_{l>0} \sum\limits_{a_1 < \ldots < a_l} \sum_j \theta_{a_1} \ldots \theta_{a_l} \lbrace \partial_{[a_l]}^{z,z'} d_X \ldots \partial_{[a_1]}^{z,z'} d_X \rbrace_{ji} \otimes e_j \\
\rho_X^{-1} \left( e_i \right) &=\sum\limits_{l>0} \sum\limits_{a_1 < \ldots < a_l} \sum_j \left( -1 \right)^{\binom{l}{2}+ l|e_j|} e_j \otimes \lbrace \partial_{[a_1]}^{x,x'} d_X \ldots \partial_{[a_l]}^{x,x'} d_X \rbrace_{ji} \theta_{a_1} \ldots \theta_{a_l}
\end{split}
\nonumber
\end{equation}
\end{thm}

With adjoints and unit inverses, one can check the following proposition.
\begin{prop}
Let $X$, $Y$ be matrix factorizations of $V \left( z_1,\ldots,z_m \right)-W \left( x_1,\ldots,x_n \right)$ and $U \left( y_1,\ldots,y_p \right)-V \left( z_1,\ldots,z_m \right)$ resp., and $\varphi: X \rightarrow Y$ a morphism (${^\dagger \varphi}: {^\dagger Y} \rightarrow {^\dagger X}$). Then,
\begin{enumerate}
\item $\left( \varphi \otimes 1_{{^\dagger X}} \right) \circ \mathrm{coev}_X=\left( 1_Y \otimes  {^\dagger\varphi} \right) \circ \mathrm{coev}_Y$,
\item $\mathrm{ev}_Y \circ \left( 1_{^\dagger Y} \otimes \varphi \right)=\mathrm{ev}_X \circ \left( ^\dagger \varphi \otimes 1_X \right)$,
\item ${^\dagger \varphi}=\lambda_{{^\dagger X}} \circ \left( \mathrm{ev}_Y \otimes 1_{{^\dagger X}} \right) \circ \left( 1_{{^\dagger Y}} \otimes \varphi 1_{{^\dagger X}} \right) \circ \left( 1_{{^\dagger Y}} \otimes \mathrm{coev}_X \right) \circ \rho^{-1}_{{^\dagger Y}}$
\end{enumerate}
and similarly for $\widetilde{\mathrm{ev}}$ and $\widetilde{\mathrm{coev}}$. Also,
\begin{equation}
\begin{split}
\left( \mathrm{ev}_{(Y \otimes X)^\dagger} \otimes \mathrm{id}_{X^\dagger} \otimes \mathrm{id}_{Y^\dagger} \right) \circ \left( \mathrm{id}_{(Y \otimes X)^\dagger} \otimes \mathrm{coev}_X \otimes \mathrm{coev}_Y \right)\\
\simeq  \left( \mathrm{id}_{X^\dagger} \otimes \mathrm{id}_{Y^\dagger} \otimes \widetilde{\mathrm{ev}}_{(Y \otimes X)^\dagger} \right) \circ \left( \widetilde{\mathrm{coev}}_{X^\dagger} \otimes \widetilde{\mathrm{coev}}_{Y^\dagger} \otimes \mathrm{id}_{(Y \otimes X)^\dagger} \right)
\end{split}
\nonumber
\end{equation}
(where $\simeq$ hides some signs which are explained in detail in \cite[Section 7]{carqmurfet}) and similarly for the left dual.
\label{pivotalityLG}
\end{prop}

Thanks to these results, we can also provide an explicit description of the quantum dimension associated to a matrix factorization:

\begin{prop}
Let $V \left( x_1,\ldots,x_m \right)$ and $W \left( y_1,\ldots,y_n \right)$ be two potentials and $X$ a matrix factorization of $W-V$. Then, the left quantum dimension is: $$\mathrm{qdim}_l \left( X \right)=\left( -1 \right)^{\binom{m+1}{2}} \mathrm{Res} \left[ \frac{\mathrm{str} \left( \partial_{x_1} d^X \ldots \partial_{x_m} d^X \partial_{y_1} d^X \ldots \partial_{y_n} d^X \right) d\mathbf{y}}{\partial_{y_1} W,\ldots,\partial_{y_n} W} \right]$$
and the right quantum dimension is:
$$\mathrm{qdim}_r \left( X \right)=\left( -1 \right)^{\binom{n+1}{2}} \mathrm{Res} \left[ \frac{\mathrm{str} \left( \partial_{x_1} d^X \ldots \partial_{x_m} d^X \partial_{y_1} d^X \ldots \partial_{y_n} d^X \right) d\mathbf{x}}{\partial_{x_1} V,\ldots,\partial_{x_m} V} \right]$$
\label{qdims}
\end{prop}

\section{On Landau-Ginzburg models}

Landau-Ginzburg models are a topic which by itself is of high interest in the physics literature. Initially a model to describe superconductivity \cite{ginzburglandau,abrikosov}, in 1988 Vafa and Warner proposed a generalization to $\left( 2,2 \right)$--supersymmetric theories in 2 spacetime dimensions in \cite{vafawarner}. This was the beginning of a fertile research (see e.g. \cite{greene,witten,gaiotto}) which continues to this date. They in addition play an important role in the active field of research of mirror symmetry and algebraic geometry \cite{kapustinli,orlov}.

In this section we would like to introduce them. We will give an overview of how to construct the action of a Landau-Ginzburg model the way it is performed in the physics literature, and hence in this section we will not be mathematically rigorous. We will also explain the relation between matrix factorizations and Landau-Ginzburg models. We follow and mix together the expositions of \cite{bhls0305}, \cite{brunrogg1}, \cite{brunrogg2} and \cite{mirrorsymmetrybook}.

Let $\Sigma$ and $X$ be two manifolds, which are called the \textit{worldsheet} and the \textit{target manifold} resp. and which are Riemannian in general. 

We are going to choose the worldsheet to be $\Sigma=\mathbb{R}^2$, and the target manifold to be $X=\mathbbm{C}^n$. Denote the coordinates of the worldsheet as $\left( x^0,x^1 \right)$, they commute: $\left[ x_0,x_1 \right]=0$. We will say these coordinates are \textit{bosonic} and functions of bosonic coordinates $\phi \left( x_0,x_1 \right)$ will be called \textit{bosonic fields}.

We will extend the worldsheet by adding to the bosonic coordinates four coordinates $\theta^+$, $\theta^-$, $\bar{\theta}^+$ and $\bar{\theta}^-$. They are related to each other via complex conjugation, $\left[ \theta^\pm \right]^*=\bar{\theta}^\pm$ \footnote{With the $\pm$ supercripts we mean right-moving or left-moving under a Lorentz transformation, i.e. it acts on bosonic and fermionic coordinates as:
\begin{equation}
\begin{split}
\left( \begin{matrix} x^0 \\ x^1 \end{matrix} \right) & \mapsto \left( \begin{matrix} \cosh\gamma & \sinh\gamma \\ \sinh\gamma & \cosh\gamma \end{matrix} \right) \left( \begin{matrix} x^0 \\ x^1 \end{matrix} \right) \\
\theta^{\pm} & \mapsto \mathrm{e}^{\pm \gamma/2} \theta^\pm \\
\bar{\theta}^{\pm} & \mapsto \mathrm{e}^{\pm \gamma/2} \bar{\theta}^{\pm}
\end{split}
\nonumber
\end{equation} for $\gamma \in \left( 1,\infty \right)$.}.
These new coordinates square to zero $\left( \theta^+ \right)^2=0=\left( \theta^- \right)^2$ (and analogously for their complex conjugate) and they anticommute\footnote{Recall that anticommutators satisfy $\lbrace a,b \rbrace=\lbrace b,a \rbrace$, for any $a$,$b$ bosonic (or fermionic) coordinates.} with each other:
\begin{equation}
\lbrace \theta^{\alpha},\theta^{\beta} \rbrace=0, \quad \quad \lbrace \theta^{\alpha}, \bar{\theta}^{\beta} \rbrace=0, \quad \quad \lbrace \bar{\theta}^{\alpha},\bar{\theta}^{\beta} \rbrace=0
\nonumber
\end{equation}
with $\alpha,\beta \in \lbrace +,- \rbrace$ (any possible combination of values of $\alpha$ and $\beta$ is allowed).
These anticommuting coordinates will be called \textit{fermionic coordinates} and functions of them will be called \textit{fermionic fields}. If the action of a model depending on bosonic and fermionic fields is invariant under supersymmetric transformations we call it a \textit{supersymmetric sigma model}.

\begin{rem}
With respect to the partial derivatives with respect to the bosonic and fermionic variables, it will be useful for us to recall that:
\begin{itemize}
\item For each pair of fermionic coordinates, the partial derivatives with respect to them anticommute the same way as the respective coordinates do, that means for example $\lbrace \frac{\partial}{\partial \theta^{\pm}}, \frac{\partial}{\partial \bar{\theta}^{\pm}}\rbrace=0$.
\item The partial derivatives with respect of a bosonic coordinate and a fermionic coordinate commute, e.g. $\left[ \frac{\partial}{\partial x^0}, \frac{\partial}{\partial \theta^+} \right]=0$.
\item The anticommutator of a fermionic coordinate with its partial derivative is equal to one, e.g. $\lbrace \theta^+,\frac{\partial}{\partial \theta^+} \rbrace=1$.
\item In general, for any two fermionic coordinates $F_1$ and $F_2$ and a bosonic coordinate $B$, it holds
\begin{equation}
\lbrace F_1,F_2 B \rbrace=\lbrace F_1,F_2 \rbrace B-F_2 \left[ F_1,B \right]
\nonumber
\end{equation}
\end{itemize}
\label{Carlo}
\end{rem}

Given all these, the $\left( 2,2 \right)$\textit{-superspace} is the space with coordinates $x_0,x_1, \theta^{\pm}, \bar{\theta}^{\pm}$ as described above, and we will denote it with $\Sigma^{sup}$. The \textit{superfields} $\Phi \left( x_0,x_1, \theta^{\pm}, \bar{\theta}^{\pm} \right) \in \Sigma^{sup}$ are functions defined on the superspace. 

We will denote the complex conjugate of a given superfield $\Phi$ as $\bar{\Phi}$.

Consider the change of bosonic coordinates $x^{\pm}=x_0 \pm x_1$. Let us introduce some differential operators in this superspace, the so-called \textit{supercharges}:
\begin{equation}
\begin{split}
\mathcal{Q}_{\pm} &=\frac{\partial}{\partial \theta^{\pm}} +i \bar{\theta}^{\pm} \partial_{\pm} \\
\bar{\mathcal{Q}}_{\pm} &=-\frac{\partial}{\partial \bar{\theta}^{\pm}} -i \theta^{\pm} \partial_{\pm}
\nonumber
\end{split}
\end{equation}
where $\partial_{\pm}=\frac{1}{2}\left( \frac{\partial}{\partial x_0} \pm \frac{\partial}{\partial x_1}\right)$. The supercharges satisfy: 
\begin{equation}
\lbrace \mathcal{Q}_{\pm}, \bar{\mathcal{Q}}_{\pm} \rbrace=-2i \partial_{\pm}
\label{superchargesparenthesis}
\end{equation}
with all the other anticommutators vanishing. Actually, this is easy to see --let's take for example $\lbrace \mathcal{Q}_+ , \bar{\mathcal{Q}}_+ \rbrace$:
\begin{equation}
\begin{split}
\lbrace \mathcal{Q}_+ , \bar{\mathcal{Q}}_+ \rbrace &=- \lbrace \frac{\partial}{\partial \theta^+},\frac{\partial}{\partial \bar{\theta}^+} \rbrace-i \lbrace \frac{\partial}{\partial \theta^+},\theta^+ \partial_+ \rbrace -i \lbrace \bar{\theta}^+ \partial_+,\frac{\partial}{\partial \bar{\theta}^+} \rbrace+\lbrace \bar{\theta}^+ \partial_+,\theta^+ \partial_+ \rbrace \\ &=-2i \partial_+
\end{split}
\nonumber
\end{equation} where we have used the commutators and anticommutators detailed in Remark (\ref{Carlo}).

Furthermore, let us introduce another set of differential operators which we will call the \textit{covariant derivatives}:
\begin{equation}
\begin{split}
\mathcal{D}_{\pm} &=\mathcal{Q}_{\pm}-2i\bar{\theta}^{\pm} \partial_{\pm}  \\
\bar{\mathcal{D}}_{\pm}&=\bar{\mathcal{Q}}_{\pm}+2i\theta^{\pm} \partial_{\pm}
\end{split}
\label{covdersupchar}
\end{equation} or more explicitly,
\begin{equation}
\begin{split}
\mathcal{D}_{\pm} &=\frac{\partial}{\partial \theta^{\pm}}-i\bar{\theta}^{\pm} \partial_{\pm} \\
\bar{\mathcal{D}}_{\pm} &= -\frac{\partial}{\partial \bar{\theta}^{\pm}} +i\theta^{\pm} \partial_{\pm}
\nonumber
\end{split}
\end{equation}
Similarly to Eq. \ref{superchargesparenthesis}, $\mathcal{D}_{\pm}$, $\bar{\mathcal{D}}_{\pm}$ anticommute with the supercharges and satisfy
\begin{equation}
\lbrace \mathcal{D}_{\pm},\bar{\mathcal{D}}_{\pm} \rbrace=2i \partial_{\pm}
\nonumber
\end{equation}
with all other anticommutators vanishing. Let us compute some cases as an instance. Take for example the case $\lbrace \mathcal{D}_+, \mathcal{Q}_+ \rbrace$:
\begin{equation}
\begin{split}
\lbrace \mathcal{D}_+, \mathcal{Q}_+ \rbrace &= \lbrace \frac{\partial}{\partial \theta^+}, \frac{\partial}{\partial \theta^+} \rbrace +i \lbrace \frac{\partial}{\partial \theta^+},\bar{\theta}^+ \partial_+ \rbrace -i \lbrace \bar{\theta}^+ \partial_+,\frac{\partial}{\partial \theta^+} \rbrace+\lbrace \bar{\theta}^+ \partial_+,\bar{\theta}^+ \partial_+ \rbrace \\ &=0
\nonumber
\end{split}
\end{equation}
Concerning the anticommutativity of the covariant derivatives, take $\lbrace \mathcal{D}_+,\bar{\mathcal{D}}_+ \rbrace$ for example:
\begin{equation}
\begin{split}
\lbrace \mathcal{D}_+,\bar{\mathcal{D}}_+ \rbrace &= - \lbrace \frac{\partial}{\partial \theta^+},\frac{\partial}{\partial \bar{\theta}^+} \rbrace+ i \lbrace \frac{\partial}{\partial \theta^+},\theta^+ \partial_+ \rbrace +i \lbrace \bar{\theta}^+ \partial_+,\frac{\partial}{\partial \bar{\theta}^+}\rbrace+\lbrace \bar{\theta}^+ \partial_+,\theta^+ \partial_+ \rbrace \\ &=2i \partial_+
\nonumber
\end{split}
\end{equation}

Among all possible superfields, there is a special kind which will be important for us. A \textit{chiral superfield} is a superfield which satisfies:
\begin{equation}
\bar{\mathcal{D}}_{\pm} \Phi=0
\label{defnchiral}
\end{equation}
In case a superfield satisfies:
\begin{equation}
\mathcal{D}_{\pm} \bar{\Phi}=0
\nonumber
\end{equation}
then it is called an \textit{anti-chiral superfield}.

A result concerning products of chiral superfields which we give without proof is that if $\Phi_1$ and $\Phi_2$ are chiral superfields, then their product $\Phi_1 \Phi_2$ is also a chiral superfield. This will be specially useful for us to construct an action.

Consider the following change of variables: $y^{\pm}=x^{\pm}-i\theta^{\pm} \bar{\theta}^{\pm}$, which rearranges the coordinates dependence so we can write chiral superfields as $\Phi=\Phi \left( y^{\pm}, \theta^{\pm} \right)$.

Furthermore, we can describe a general chiral superfield in terms of fermionic and bosonic fields:
\begin{equation}
\begin{split}
\Phi \left( y^{\pm},\theta^{\pm} \right) &=\phi \left( y^{\pm} \right)+\theta^+ \psi'_+ \left( y^{\pm} \right) + \theta^- \psi'_-\left( y^{\pm} \right) +\theta^+ \theta^- F' \left( y^{\pm} \right) \\
&= \phi-i \theta^+ \bar{\theta}^+ \partial_+ \phi-i \theta^-\bar{\theta}^- \partial_- \phi-\theta^+ \theta^- \bar{\theta}^- \bar{\theta}^+ \partial_+ \partial_- \phi \\ &+\theta^+ \psi_+ -i \theta^+ \theta^- \bar{\theta}^- \partial_- \psi_+ + \theta^- \psi_- -i\theta^+ \theta^-\bar{\theta}^+ \partial_+ \psi_- + \theta^+ \theta^- F
\label{chiralsuperfield}
\end{split}
\end{equation}

We can see this as follows. Consider the chiral superfield $\Phi \left( y^{\pm},\theta^{\pm} \right)$. We can make a Taylor expansion around $\theta^{\pm}=0$; due to the dependence of $y^\pm$ on $\theta^{\pm}$ it is notation-wise convenient to rewrite $\Phi \left( y^{\pm},\theta^{\pm} \right)=\Xi \left( x^\pm,\theta^\pm \right)$. Then, taking into account that the fermionic variables square to zero:
\begin{equation}
\begin{split}
\Xi \left( x^{\pm}, \theta^{\pm} \right)= \Xi \left( x^{\pm}, 0 \right)+\theta^+ \frac{\partial \Xi \left( x^{\pm},0 \right)}{\partial \theta^+}+\theta^- \frac{\partial \Xi \left( x^{\pm},0 \right)}{\partial \theta^-}+\theta^+ \theta^- \frac{\partial^2 \Xi \left( x^{\pm},0 \right)}{\partial \theta^+ \partial \theta^-}
\end{split}
\nonumber
\end{equation}
Redefine: 
\begin{equation}
\begin{split}
\phi &:=\Xi \left( x^{\pm}, 0 \right)=\Phi \left( y^{\pm}, 0 \right) \\
\psi'_\pm &:=\frac{\partial \Xi \left( x^{\pm},0 \right)}{\partial \theta^\pm} \\
F' &:=\frac{\partial^2 \Xi \left( x^{\pm},0 \right)}{\partial \theta^+ \partial \theta^-}
\end{split}
\nonumber
\end{equation}
We obtain the first line of Eq. \ref{chiralsuperfield}. Then, acting via partial derivation and using Remark (\ref{Carlo}) we can see more explicitly that these fields are of the shape:
\begin{equation}
\begin{split}
\psi'_\pm &=\frac{\partial \Phi \left( y^{\pm},0 \right)}{\partial y^\pm} \frac{\partial y^\pm}{\partial \theta^\pm}+\frac{\partial \Phi \left( y^{\pm},0 \right)}{\partial \theta^\pm}=-i \bar{\theta}^\pm \partial_\pm \Phi \left( y^{\pm}, 0 \right) +\frac{\partial \Phi \left( y^{\pm},0 \right)}{\partial \theta^\pm} \\ &=-i\bar{\theta}^\pm \partial_\pm \phi +\psi_\pm \\
F' &=\frac{\partial \psi_-}{\partial \theta^+}=-\bar{\theta}^-\bar{\theta}^+ \partial_+ \partial_- \phi-\bar{\theta}^- \partial_- \psi_+-i\bar{\theta}^+ \partial_+ \psi_- + F
\end{split}
\nonumber
\end{equation}
where we have redefined $\psi_{\pm}=\frac{\partial \Phi \left( y^{\pm},0 \right)}{\partial \theta^\pm}$ and $F=\frac{\partial^2 \Phi \left( y^{\pm},0 \right)}{\partial \theta^+ \partial \theta^-}$. Plugging together everything, we recover the second line of \ref{chiralsuperfield}.

$\phi$ is a bosonic field, and $\psi_{\pm}$ are fermionic fields. $F$ is called the \textit{auxiliary field}.

With all these ingredients, a \textit{2-dimensional $\left(2,2 \right)$-supersymmetric Landau-Ginzburg model} is a supersymmetric model whose worldsheet is the 2-dimensional $\left( 2,2 \right)$-superspace $\Sigma^{sup}$, whose target manifold is $\mathbbm{C}^n$\footnote{In the physics literature one sometimes takes a complex Riemannian manifold instead of $\mathbbm{C}^n$. For simplicity reasons we follow our choice.} and it is characterized by the fact that its action contains a potential term specified by a non-trivial holomorphic function called \textit{potential} $W:\mathbbm{C}^n \rightarrow \mathbbm{C}$.

In order to build an action for our Landau-Ginzburg model, we first aim to construct action functionals of superfields. We require them to be invariant under the transformation:
\begin{equation}
\delta= \epsilon_+ \mathcal{Q}_- -\epsilon_- \mathcal{Q}_+ -\bar{\epsilon}_+ \bar{\mathcal{Q}}_- +\bar{\epsilon}_- \bar{\mathcal{Q}}_+
\label{delta}
\end{equation}
(which we will call the \textit{variation}) where $\epsilon_+$ and $\epsilon_-$ are complex fermionic parameters. To begin with, notice that the variations of the fields $\phi$, $\psi_{\pm}$ take the form:
\begin{equation}
\begin{split}
&\delta \phi=\epsilon_+ \psi_- - \epsilon_- \psi_+ \quad \quad \quad \quad \delta \bar{\phi}=-\bar{\epsilon}_+ \bar{\psi}_- + \bar{\epsilon}_- \bar{\psi}_+ \\
&\delta \psi_+=2i \bar{\epsilon}_- \partial_+ \phi + \epsilon_+ F \quad \quad \quad \delta \bar{\psi}_+=-2i \epsilon_- \partial_+ \phi+\bar{\epsilon}_+ \bar{F} \\
&\delta \psi_-=-2i \bar{\epsilon}_+ \partial_- \phi + \epsilon_- F \quad \quad \delta \bar{\psi}_-=2i \epsilon_+ \partial_- \bar{\phi}+\bar{\epsilon}_- \bar{F}
\end{split}
\end{equation}

One gets these variations as follows. Applying the variation (\ref{delta}) to the chiral superfield, we get 4 terms for each supercharge. For the first two supercharges, we apply simply partial derivation and the Taylor expansion of our chiral superfield (\ref{chiralsuperfield}):
\begin{equation}
\begin{split}
\mathcal{Q}_- \Phi \left( y^\pm,\theta^\pm \right) &=\left( \frac{\partial}{\partial \theta^-}+i \bar{\theta}^- \frac{\partial}{\partial x^-} \right) \Phi \left( y^\pm,\theta^\pm \right) =\frac{\partial \Phi \left( y^\pm,\theta^\pm \right)}{\partial \theta^-} \\&=\frac{\partial}{\partial \theta^-} \left( \Phi \left( y^\pm,0 \right)+\theta^+ \psi_++\theta^- \psi_- +\theta^+ \theta^- F \right)=\psi_- +\theta^+ F
\end{split}
\nonumber
\end{equation}
and analogously for $\mathcal{Q}_+ \Phi \left( y^\pm, \theta^\pm \right)$, which is equal to $\psi_+-\theta^-F$. On the other hand, for the conjugated supercharges, we use the chiral superfield condition (\ref{defnchiral}) and Remark \ref{covdersupchar}:
\begin{equation}
\begin{split}
\bar{\mathcal{Q}}_- \Phi \left( y^\pm, \theta^\pm \right) &=\left( \bar{\mathcal{D}}_- -2i \theta^- \partial_- \right) \Phi \left( y^\pm, \theta^\pm \right) =-2i\theta^- \partial_- \Phi \left( y^\pm, \theta^\pm \right) \\ \bar{\mathcal{Q}}_+ \Phi \left( y^\pm, \theta^\pm \right) &=\left( \bar{\mathcal{D}}_+ -2i \theta^+ \partial_+ \right)\Phi \left( y^\pm, \theta^\pm \right)=-2i \theta^+ \partial_+ \Phi \left( y^\pm, \theta^\pm \right)
\end{split}
\nonumber
\end{equation}
Putting together terms,
\begin{equation}
\begin{split}
\delta \Phi \left( y^\pm, \theta^\pm \right) &=\left( \epsilon_+ \psi_- - \epsilon_- \psi_+ \right)+\theta^+ \left( \epsilon_+ F -\bar{\epsilon}_- 2i \theta^+ \partial_+ \Phi \left( y^\pm, \theta^\pm \right) \right) \\ &+\theta^- \left( \epsilon_- F +2i \bar{\epsilon}_+ \partial_- \Phi \left( y^\pm, \theta^\pm \right)\right)
\end{split}
\nonumber
\end{equation}
we can recognize the variations of $\phi$ and $\psi_\pm$.

We would like to construct some supersymmetric bulk action for this model-and for this, we need our action to be invariant under the variation $\delta$. As we would like to also integrate over fermionic coordinates, we need to specify that, for a general fermionic coordinate $\theta$,
\begin{equation}
\begin{split}
\int d \theta &=0 \\
\int \theta d \theta &=1
\end{split}
\nonumber
\end{equation}

Then, in general, for a superfield $\mathcal{F}_i$, let us first consider the functional: $$S_D=\int_\Sigma d^2 xd \theta^+ d \theta^- d \bar{\theta}^- d \bar{\theta}^+ K \left( \mathcal{F}_i \right)$$ where $K \left( \mathcal{F}_i \right)$ is an arbitrary differentiable function of the $\mathcal{F}_i$'s (usually called the \textit{K{\"a}hler potential}). Under the variation, this functional actually satisfies that $$\delta S_D=0$$
One can check this as follows. We have:
$$\delta S_D=\int_\Sigma d^2 xd \theta^+ d \theta^- d \bar{\theta}^- d \bar{\theta}^+ \left( \epsilon_+ \mathcal{Q}_- -\epsilon_- \mathcal{Q}_+ -\bar{\epsilon}_+ \bar{\mathcal{Q}}_- +\bar{\epsilon}_- \bar{\mathcal{Q}}_+ \right)K \left( \mathcal{F}_i \right)$$
The integration over $d^4 \theta$ is only nonzero if we have some coefficient $\theta^+ \theta^- \bar{\theta}^+ \bar{\theta}^-$ in front of us. For example, concerning the first supercharge, the first term vanishes since the integrand does not have $\theta^-$ (because of the derivative $\partial/\partial \theta^-$). The second term is a total derivative and vanishes after integration over $d^2x$. The others work the same.

$S_D$ is called the \textit{$D$-term} (the \textit{kinetic term}). We will choose for simplicity reasons that $K \left( \Phi,\bar{\Phi} \right)=\bar{\Phi}\Phi$: a K{\"a}hler potential of this shape is invariant under any assignment of $R$-charges to $\Phi$, and hence the $D$-term is invariant under vector and axial $R$-symmetries. Hence, for the $D$-term, we have
\begin{equation}
S_D=\int_{\Sigma} d^2x d^4 \theta \bar{\Phi}\Phi
\nonumber
\end{equation}
Recall that if $\Phi_1$ and $\Phi_2$ are chiral superfields, then their product $\Phi_1 \Phi_2$ is also a chiral superfield and that $\Phi$ has the $\theta$-expansion (\ref{chiralsuperfield}). The integration over $d^4 \theta$ amounts to extracting the coefficient of $\theta^4=\theta^+ \theta^- \bar{\theta}^- \bar{\theta}^+$ in the $\theta$-expansion of the product. Via direct computation, it is easy to see that:
\begin{equation}
\begin{split}
\bar{\Phi}\Phi \vert_{\theta^4} &=-\bar{\phi} \partial_+ \partial_- \phi+\partial_+ \bar{\phi} \partial_- \phi+\partial_- \bar{\phi} \partial_+ \phi-\partial_+ \partial_- \bar{\phi} \phi + \bar{\psi}_+ \partial_- \psi_+ -i \partial_- \bar{\psi}_+ \psi_+ \\ &+i \bar{\psi}_- \partial_+ \psi_- -i\partial_+ \bar{\psi}_- \psi_- + \vert F \vert^2
\end{split}
\nonumber
\end{equation}
Acting via partial integration, we finally get that:
\begin{equation}
S_{D}=\int_{\Sigma} d^2x \left( \vert \partial_0 \phi \vert^2- \vert \partial_1 \phi \vert^2 + \frac{i}{2} \bar{\psi}_+ \left( \partial_0 -\partial_1 \right) \psi_+ +\frac{i}{2} \bar{\psi}_- \left( \partial_0 +\partial_1 \right) \psi_- + \vert F \vert^2 \right)
\nonumber
\end{equation}

Apart from a functional of all the considered superfields, we can consider (for supersymmetry reasons) a functional only with the potential, of the form: $$S_F=\int_\Sigma d^2x d \theta^- d\theta^+ W \left( \Phi_i \right) \vert_{\bar{\theta}^{\pm}=0}+c.c.$$
Again, this functional satisfies that $\delta S_F=0$:
$$\delta S_F=\int_\Sigma d^2x d \theta^- d\theta^+ \left( \epsilon_+ \mathcal{Q}_- -\epsilon_- \mathcal{Q}_+ -\bar{\epsilon}_+ \bar{\mathcal{Q}}_- +\bar{\epsilon}_- \bar{\mathcal{Q}}_+ \right)W \left( \Phi_i \right) \vert_{\bar{\theta}^{\pm}=0}$$
For the first two supercharges, the integral has the form $$\pm \int_{\Sigma} d^2xd \theta^- d\theta^+ \epsilon_{\pm} \left( \frac{\partial}{\partial \theta^{\mp}} +i \bar{\theta}^{\mp} \partial_{\mp} \right) W \left( \Phi_i \right) \vert_{\bar{\theta}^{\pm}=0}$$
The first term vanishes for the standard reason, and the second because we set $\bar{\theta}^{\pm}=0$. For the other two supercharges, first recall (\ref{covdersupchar})and then, the integral has the form:
$$\mp \int_{\Sigma} d^2x d \theta^- d \theta^+ \bar{\epsilon}_{\pm} \left( \bar{\mathcal{D}}_{\pm} -2i \theta^{\pm} \partial_{\pm} \right) W \left( \Phi_i \right) \vert_{\bar{\theta}^{\pm}=0}$$ The first term vanishes as we are considering chiral superfields. Note that $W \left( \Phi_i \right)$ is a holomorphic function and it doesn't contain $\bar{\Phi}_i$. The second term vanishes because it is a total derivative in the bosonic variables.

We are going to call $S_F$ the \textit{$F$-term} (the \textit{potential term}). It will then be given by:
\begin{equation}
S_F=\int_{\Sigma}  d^2x d^2 \theta  W \left( \Phi \right) +\mathrm{c.c.}
\nonumber
\end{equation}
Again the integration over $d^2 \theta$ amounts to extracting the coefficient of $\theta^2= \theta^+ \theta^-$ in the $\theta$-expansion of $W \left( \Phi \right)$: $$W \left( \Phi \right)=W \vert_{\theta^{\pm}=0}+\theta^+ \frac{\partial W}{\partial \theta^+}+\theta^- \frac{\partial W}{\partial \theta^-} +\theta^+ \theta^- \frac{\partial^2 W}{\partial \theta^+ \partial \theta^-}$$
I.e. the coefficient we are looking for is:
\begin{equation}
W \vert_{\theta^2}=\frac{\partial^2 W}{\partial \theta^+ \partial \theta^-} =\frac{\partial^2 W}{\partial \Phi^2} \frac{\partial \Phi}{\partial \theta^+} \frac{\partial \Phi}{\partial \theta^-}+\frac{\partial W}{\partial \Phi} \frac{\partial^2 \Phi}{\partial \theta^+ \partial \theta^-}=-W'' \psi^+ \psi^-+W' F
\nonumber
\end{equation}
And analogously for the complex conjugate term of the potential term. Hence,
\begin{equation}
S_F=\int_{\Sigma}  d^2x \frac{1}{2}\left[-W'' \psi^+ \psi^-+W' F-\bar{W}'' \bar{\psi}^+ \bar{\psi}^-+\bar{W}' \bar{F} \right]
\nonumber
\end{equation}
Summing kinetic and potential terms together, we finally obtain the bulk action of a Landau-Ginzburg model:
\begin{equation}
\begin{split}
S_{\mathrm{bulk}} &=\int_{\Sigma} d^2x ( \vert \partial_0 \phi \vert^2- \vert \partial_1 \phi \vert^2 + \frac{i}{2} \bar{\psi}_+ \left( \partial_0 -\partial_1 \right) \psi_+ +\frac{i}{2} \bar{\psi}_- \left( \partial_0 +\partial_1 \right) \psi_- \\ &-\frac{1}{4} \vert W' \vert^2-\frac{1}{2} W'' \psi^+ \psi^- -\frac{1}{2}\bar{W}'' \bar{\psi}^+ \bar{\psi}^- +\vert F+\frac{1}{2} \bar{W}' \vert^2 )
\end{split}
\label{bulkaction}
\end{equation}
where the last term can be eliminated if we solve the equation of motion $F=-\frac{1}{2}\bar{W}'$.

\subsection{Matrix factorizations in Landau-Ginzburg models}

So far, we have formulated our theory on a worldsheet without boundary. The introduction of boundaries breaks the translation symmetry normal to the boundary, and hence, half of the supersymmetries. This will give rise to some interesting phenomena which will trigger the emergence of matrix factorizations.

Whenever we consider boundaries, there are two types of supersymmetry one can impose: \textit{B-type supersymmetry}, which preserves the supercharge $\mathcal{Q}_B=\mathcal{Q}_++\mathcal{Q}_-$ and its conjugate $\bar{\mathcal{Q}}_B$ everywhere in the theory; and \textit{A-type supersymmetry}, which preserves the combination $\mathcal{Q}_A=\mathcal{Q}_++\bar{\mathcal{Q}}_-$ and its conjugate $\bar{\mathcal{Q}}_A$.

We choose the worldsheet $\Sigma$ to be the strip with coordinates $\left( x^0,x^1 \right) \in \left( \mathbb{R}, \left[ 0,\pi \right] \right)$. We will focus our attention to $B$-type supersymmetry. In terms of the parameters $\epsilon_{\pm}$ one can describe $B$-type supersymmetry by setting $\epsilon=\epsilon_+=-\epsilon_-$.

In this new setting we can recombine the fermions as:
\begin{equation}
\begin{split}
\eta &=\psi_+ +\psi_- \\
\vartheta &=\psi_- - \psi_+
\label{fermionsrecomb}
\end{split}
\end{equation}
and the variation (\ref{delta}) becomes:
\begin{equation}
\delta=\epsilon_+ \left( \bar{\mathcal{Q}}_- + \bar{\mathcal{Q}}_+ \right) -\bar{\epsilon}_+ \left( \mathcal{Q}_- + \mathcal{Q}_+ \right)=\epsilon \bar{\mathcal{Q}}-\bar{\epsilon} \mathcal{Q}
\nonumber
\end{equation}
Hence, in the new variables (\ref{fermionsrecomb}),
\begin{equation}
\begin{split}
&\delta \phi=\epsilon \eta \quad \quad \quad \quad \quad\quad \delta \bar{\phi}=-\bar{\epsilon} \bar{\eta} \\
&\delta \eta=-2i \bar{\epsilon} \partial_0 \phi \quad \quad \quad \quad \delta \bar{\eta}=2i \epsilon \partial_0 \bar{\phi} \\
&\delta \vartheta=2i \bar{\epsilon} \partial_1 \phi + \epsilon \bar{W}' \quad \quad \delta \bar{\vartheta}=2i \epsilon \partial_1 \bar{\phi}+\bar{\epsilon} W'
\end{split}
\label{transformationsBsusy}
\end{equation}

The boundary superspace is spanned by the coordinate $\theta^0$ and its complex conjugate $\bar{\theta^0}$ which is defined as $\theta^0=\frac{1}{2} \left( \theta^+ + \theta^- \right)$. Hence, $$\frac{1}{2} \frac{\partial}{\partial \theta^0}=\frac{\partial}{\partial \theta^+}+\frac{\partial}{\partial \theta^-}$$
and the supercharges and covariant derivatives become 
\begin{equation}
\begin{split}
\bar{\mathcal{Q}} &=\frac{1}{2} \frac{\partial}{\partial \theta^0}+i\bar{\theta}^0 \frac{\partial}{\partial x_0} \quad \quad \quad \mathcal{Q} =-\frac{1}{2} \frac{\partial}{\partial \bar{\theta}^0}-i\theta^0 \frac{\partial}{\partial x_0} \\ \mathcal{D} &=\frac{1}{2} \frac{\partial}{\partial \theta^0}-i\bar{\theta}^0 \frac{\partial}{\partial x_0} \quad \quad \quad \bar{\mathcal{D}} =-\frac{1}{2} \frac{\partial}{\partial \bar{\theta}^0}+i\theta^0 \frac{\partial}{\partial x_0}
\end{split}
\nonumber
\end{equation}
After imposing B-type supersymmetry, one can collect the fields of the chiral superfield $\Phi$ into a bosonic and a fermionic multiplet $\Phi' \left( y^0, \theta^0 \right)$ and $\Theta' \left( y^0,\theta^0, \bar{\theta}^0 \right)$ resp., where $y^0=x^0-i\theta^0 \bar{\theta}^0$. On the one hand, the bosonic multiplet looks like $$\Phi' \left( y^0,\theta^0 \right)=\phi \left( y^0 \right)+\theta^0 \eta \left( y^0 \right)$$ which comes from Taylor expanding $\Phi \left( y^\pm, \theta^\pm \right)$ around $\theta^\pm=0$ and then changing coordinates ($y^\pm \to y^0$ and $\theta^\pm \to \theta^0$). On the other hand, the fermionic multiplet looks like $$\Theta' \left( y^0,\theta^0,\bar{\theta}^0 \right)=\theta \left( y^0 \right)-2\theta^0 F \left( y^0 \right) +2i \bar{\theta}^0 \left[ \partial_1 \phi \left( y^0 \right) +\theta^0 \partial_1 \eta \left( y^0 \right) \right]$$ and which comes from Taylor expanding $\Phi \left( y^\pm,\theta^\pm,\bar{\theta}^\pm \right)$ around $\bar{\theta}^\pm=0$ and changing coordinates again (as in the bosonic case).

\begin{rem}
\begin{itemize}
\item The bosonic multiplet is chiral, $\mathcal{D} \Phi'=0$.
\item The fermionic one is not chiral, and rather satisfies $\mathcal{D} \Theta'=-2i \partial_1 \Phi'$.
\end{itemize}
\end{rem}

Back to the Lagrangian, we would like to construct the necessary boundary terms in order to get a fully supersymmetric action. If we set the potential to zero, the B-type supersymmetry variation of the action gives rise to a surface term which can be compensated by
\begin{equation}
S_{\partial \Sigma}=\frac{i}{4} \int dx^0 \left[ \bar{\vartheta} \eta -\bar{\eta} \vartheta \right]\vert^\pi_0
\nonumber
\end{equation}

But if we turn on the potential, there is an extra term coming from the variation of the bulk action (\ref{bulkaction}) plus the boundary action:
\begin{equation}
\delta \left( S_{F,B-susy}+S_{\partial \Sigma} \right)=\frac{i}{2} \int dx^0 \left[ \epsilon \bar{\eta} \bar{W}'+\bar{\epsilon} \eta W' \right] \vert^\pi_0
\label{boundaryextraterm}
\end{equation}
which cannot be compensated by any boundary action as the transformations (\ref{transformationsBsusy}) do not provide the necessary terms.

In order to force our model to be supersymmetric, we introduce the following ansatz.

\begin{ansatz}
Introduce a boundary fermionic superfield $\Pi$ which is not chiral (it rather satisfies $D \Pi=E \left( \bar{\Phi}' \right)$) and has the expansion  $$\Pi \left( y^0,\theta^0,\bar{\theta}^0 \right)=\pi \left( y^0 \right)+\theta^0 l\left( y^0 \right)-\bar{\theta}^0 \left[ E \left( \phi \right)+\theta^0 \eta \left( y^0 \right)E'\left( \phi \right)\right]$$ The component fields of $\Pi$ transform as:
\begin{equation}
\begin{split}
\delta \pi &=\epsilon l-\bar{\epsilon} E \\
\delta l &=-2i\bar{\epsilon} \partial_0 \pi+\bar{\epsilon}\eta E' \\
\delta \bar{\pi} &=\bar{\epsilon}\bar{l}-\epsilon \bar{E} \\
\delta \bar{l} &=-2i\epsilon \partial_0 \bar{\pi}-\epsilon \bar{\eta} \bar{E'}
\end{split}
\nonumber
\end{equation}
\end{ansatz}

Define a fermionic field $J$ via
\begin{equation}
l=-i \bar{J}
\label{defJ}
\end{equation}
Using \ref{defJ} and making an analogous construction as we did for the bulk action, we can build two terms for the boundary action,
\begin{equation}
\begin{split}
S_{\partial \Sigma} &= -\frac{1}{2} \int dx^0 d^2 \theta \bar{\Pi} \Pi \vert_0^\pi-\frac{i}{2} \int dx^0 d\theta \Pi J \left( \Phi \right)_{\bar{\theta}=0} \vert_0^\pi+c.c. \\ &=\int dx^0 \left[ i \bar{\pi} \partial_0 \pi -\frac{1}{2} \bar{J}J-\frac{1}{2} \bar{E}E+\frac{i}{2} \pi \eta J'+\frac{i}{2} \bar{\pi} \eta \bar{J}' -\frac{1}{2} \bar{\pi} \eta E'+\frac{1}{2} \pi \bar{\eta} \bar{E}' \right] \vert_0^\pi
\end{split}
\nonumber
\end{equation}
Equation \ref{defJ} also makes the variation of the boundary fermion become:
\begin{equation}
\begin{split}
\delta \pi &=-i \epsilon \bar{J}-\bar{\epsilon}E  \\ \delta \bar{\pi} &=i \bar{\epsilon} J - \epsilon \bar{E}
\end{split}
\nonumber
\end{equation}

The boundary action is not supersymmetric but rather generates:
\begin{equation}
\delta S_{\partial \Sigma}=-\frac{i}{2} \int dx^0 \left[ \epsilon \bar{\eta} \left( \bar{E}\bar{J} \right)'+\bar{\epsilon} \eta \left( EJ \right)' \right]\vert_0^\pi
\nonumber
\end{equation}

But this expression is exactly what we need to compensate (\ref{boundaryextraterm}), if
\begin{equation}
\makebox{W=EJ}
\label{matrixfactorizationcond}
\end{equation}
Or also, $W=J E$ as $E$ and $J$ commute. Any pair $\left( E,J \right)$ that satisfies this condition makes the theory supersymmetric. One may also want to introduce not only one fermionic superfield but $m$ such fields $\Pi_m$, together with the associated $E_i$'s and $J_i$'s. The action of our theory will be then supersymmetry-preserving if $E_i$ and $J_i$ assemble into $\left( 2^{m-1} \times 2^{m-1} \right)$-matrices $\tilde{E}$ and $\tilde{J}$ that satisfy the factorization condition: $$\tilde{E} \tilde{J}=\tilde{J} \tilde{E}=W.\mathbbm{1}_{2^{m-1} \times 2^{m-1}}.$$
If we define $D=\left( \begin{matrix} 0 & E \\J & 0 \end{matrix} \right)$, we can rewrite Eq. \ref{matrixfactorizationcond} as:
$$D^2= W. \mathbbm{1}_{2^m \times 2^m}$$
Then, in this shape, we can recognize $E$ and $J$ as the twisted differentials of a matrix factorization of the potential $W$, and thus interpret the Lagrangian description of boundary conditions in Landau-Ginzburg models in terms of matrix factorizations.

The interpretation of this condition as a matrix factorization was first suggested by Kontsevich \cite{kontsevich}, and later developed in several works by \cite{douglas,bhls0305,kapustinli,lazaroiu,carqueville}, within the context of (open topological) string theory. Kontsevich's realization opened the doors for applications in this active field of physics of the mathematical background already independently developed by the mathematics community.

To finish this section, let us briefly remark that we have seen the rise of matrix factorizations when we introduce boundaries, and one may wonder what is the situation for defects in Landau-Ginzburg models. Recall that, given two Landau-Ginzburg models characterized by two potentials $W_1$ and $W_2$ resp., a \textit{defect} between two Landau-Ginzburg models is a codimension 1 interface between them. 

If our defect is for instance over the real line, so one of our Landau-Ginzburg models (say $W_1$) is on the upper half complex plane and the other ($W_2$) on the lower, we can perform the so-called \textit{folding trick}: fold the lower half complex plane onto the upper. The defect can then be understood as a boundary, and our Landau-Ginzburg model will be now parametrized by $W_1-W_2$ (the minus for $W_2$ is because of the different relative orientations of the boundary). Hence, again, we can understand defects in terms of matrix bifactorizations and apply all our mathematical machinery to them. For more details on this discussion we refer to \cite{brunrogg1}.

To conclude, we recapitulate the mathematical realization of Landau-Ginzburg models described along this Section in Table \ref{tab1}.
\begin{table}
\begin{center}
\begin{tabular}{c|c}
Physical entities & Mathematical interpretation \\ \hline
Landau-Ginzburg model & $\left( S,W \right)$ \\
B-type boundaries & Matrix factorization of $W$ \\ 
B-type defects & Matrix bifactorization of $W_1-W_2$ \\ \hline
\end{tabular}
\caption{Mathematical interpretation of Landau-Ginzburg models}
\label{tab1}
\end{center}
\end{table}

\chapter{The Landau-Ginzburg/conformal field theory correspondence}
\label{ch:LGCFT}

In this chapter, we would like to motivate the topic of the present thesis: the Landau-Ginzburg/conformal field theory correspondence (from now on, LG/CFT). We will review some features of conformal field theory (CFT) and then explain the relation between defects in Landau-Ginzburg models and defects in CFT. Our main sources for this chapter will be \cite{mooreseiberg,tft1,tft2,tft3,tft4,tft5,dkr,icmfrs,carqrunkel2}.

\section{The 2-categorical approach to conformal field theories}
\label{functorialapproach}

We will think of a conformal field theory as being defined by its correlation functions. By a \textit{full CFT} (as opposed to a chiral CFT) we mean a collection of single-valued functions (the correlators) which satisfy the so-called Ward identities, and are compatible with the operator product expansion. An important case is that of those full CFTs which are rational. Recall here that by rational we mean a CFT whose symmetries are described by a \textit{rational vertex algebra} (recall that vertex algebra is a generalization of a commutative algebra) \footnote{We will not need to work out with the inner structure of these algebras for our results, nor for motivating the LG/CFT correspondence. Thus, for a precise definition of vertex algebras and rationality we will refer to the foundational papers \cite{bpz,bocherds} and also \cite{barron1,barron2,kac,frenkel,vanekeren}.}. This particular kind of CFTs is so well-understood thanks to the fact that rational vertex algebras have a finite number of (isomorphism classes of) simple representations and any representation is a finite direct sum of the simple ones --facts that simplify their study.

Full CFTs can be described in a very concise way and have been intensively studied, among others, in \cite{mooreseiberg,tft1,tft2,tft3,tft4,tft5}. We are going to summarize the most important result for our purposes.

To begin with, we need to introduce some necessary definitions.
\begin{defn}
\begin{itemize}
\item A \textit{Frobenius algebra} in a (strict) monoidal category $\mathcal{C}$ is an object that is both an algebra and a coalgebra and for which the product and coproduct are related by $$\left( \mathrm{id}_A \otimes m \right) \circ \left( \Delta \otimes \mathrm{id}_A \right) = \Delta \circ m = \left( m \otimes \mathrm{id}_A \right) \circ \left( \mathrm{id}_A \otimes \Delta \right)$$
\item A \textit{$\Delta$-separable algebra} in a tensor category is an object that is both an algebra and a coalgebra and that in addition satisfies $m \circ \Delta=\mathrm{id}_A$. By misuse of language, we will refer to this property simply as \textit{separable}. If in addition it satisfies that $\epsilon \circ \eta=\beta_1 \mathrm{id}_{\mathbbm{1}}$ (for non-zero $\beta_1 \in \mathbbm{k}$), then we say it is a \textit{special algebra}.
\item A \textit{symmetric algebra} in a pivotal category is an algebra object $\left( A,m,\eta \right)$ together with a morphism $\epsilon \in \mathrm{Hom} \left( A,\mathbbm{1} \right)$ such that the two morphisms $\Phi_1,\Phi_2 \in \mathrm{Hom} \left( A,A^\vee \right)$ defined as
\begin{equation}
\begin{split}
\Phi_1 &:= \left[ \left( \epsilon \circ m \right) \otimes \mathrm{id}_{A^\vee} \right] \circ \left( \mathrm{id}_A \otimes \mathrm{coev}_A \right) \\
\Phi_2 &:=\left[ \mathrm{id}_{A^\vee} \otimes \left( \epsilon \circ m \right) \right] \circ \left( \widetilde{\mathrm{coev}}_A \otimes \mathrm{id}_A \right)
\nonumber
\end{split}
\end{equation}
are equal.
\end{itemize}
\end{defn}

In addition, we also need define an analogous concept to that of modules over an algebra but in the bicategorical setting (we follow here the notation of Section \ref{sec:catth}). 
\begin{defn}
Let $\mathcal{B}$ be a bicategory and $\left( A,\mu,\eta \right)$ an algebra object in $\mathcal{B} \left( a,a \right)$ for some $a \in \mathrm{Ob} \left( \mathcal{B} \right)$.
\begin{itemize}
\item A \textit{left $A$-module} is a 1-morphism $X \in \mathcal{B}\left( b,a \right)$ for some $b \in \mathrm{Ob} \left( \mathcal{B} \right)$, together with a left action of $A$ $\rho_{AX} \colon A \circ X \to X$ \footnote{One should not confuse the action of a module with the right unit isomorphism of Definition \ref{bicat}.} compatible with the algebra multiplication and counit:
\begin{equation}
\begin{split}
\rho_{AX} \odot \left( \mu \ast \mathrm{id}_X \right) &=\rho_{AX} \odot \left( \mathrm{id}_A \ast \rho_{AX} \right) \\
\rho_{AX} \odot \left( \eta \ast \mathrm{id}_X \right) &= \mathrm{id}_X .
\nonumber
\end{split}
\end{equation}
An analogous definition follows for \textit{right $A$-modules}, and in this case we will denote the right action as $\rho_{XA} \colon X \circ A \to X$.
\item A 2-morphism $\phi \colon X \to Y$ between two left $A$-modules is called a \textit{module map} if it is compatible with both the left actions of $A$ on $X$ and $Y$ resp., that means,
$$\phi \odot \rho_{AX}=\rho_{AY} \odot \left(  \mathrm{id}_A \ast \phi \right).$$

\item Let $A \in \mathcal{B}\left( a,a \right)$ and $B \in \mathcal{B} \left( b,b \right)$ be two algebras. An \textit{$B$-$A$-bimodule} is a 1-morphism $X \in \mathcal{B} \left( a,b \right)$ that is simultaneously a right $A$-module and a left $B$-module, together with the condition that both actions have to be compatible: $$\rho_{XA} \odot \left( \rho_{BX} \ast \mathrm{id}_A \right)=\rho_{BX} \odot \left( \mathrm{id}_B \ast \rho_{XA} \right).$$
\item Given two $B$-$A$-bimodules $X$, $Y$, a 2-morphism $\phi \colon X \to Y$ is called a \textit{bimodule map} if it is both a map of left and right modules. 

\end{itemize}
\end{defn}

Now let us define a tensor product between these modules. Let $A \in \mathcal{B} \left( a,a \right)$ be an algebra as before, and let $X \in \mathcal{B} \left( a,b \right)$, $Y \in \mathcal{B} \left( c,a \right)$ be right and left $A$-modules, resp.

\begin{defn}
The \textit{tensor product of $X$ and $Y$ over $A$}, that we will denote as $X \circ_A Y \in \mathcal{B} \left( c,b \right)$, is defined to be the coequalizer of $r=\rho_{XA} \ast 1_Y$ and $l=1_X \ast \rho_{AY}$ \footnote{This means, $X \circ_A Y$ is equipped with a map $\vartheta \colon X \circ Y \to X \circ_A Y$ with $\vartheta \odot l=\vartheta \odot r$ such that for all $\phi \colon X \circ Y \to Z$ with $\phi \odot l=\phi \odot r$ there is a unique map $\zeta \colon X \circ_A Y \to Z$ with $\zeta \odot \vartheta=\phi$.}.
\end{defn}

At this point we state an important result concerning full CFTs.
\begin{thm}{\cite{tft1}}
A full CFT can be fixed with only a tuple $\left( \mathcal{V}, A \right)$ where:
\begin{itemize}
\item $\mathcal{V}$ is a rational vertex algebra which encodes the chiral symmetry of a rational CFT, and
\item $A$ is a symmetric special Frobenius algebra $A$ in the representation category $\mathcal{C}=\mathrm{Rep} \left( \mathcal{V} \right)$.
\end{itemize}
\end{thm}
This representation category $\mathcal{C}=\mathrm{Rep} \left( \mathcal{V} \right)$ is $\mathbbm{C}$-linear and abelian by definition. For vertex algebras satisfying certain conditions\footnote{For more information on these conditions, we refer to \cite[Theorem 4.6]{huang}.}, $\mathrm{Rep} \left( \mathcal{V} \right)$ is a modular tensor category.

Then, given a modular category and all the possible full CFTs one can construct from it, we can define the following bicategory.

\begin{defn}
Let $\mathcal{M}$ be a modular tensor category. The bicategory $\mathrm{Frob}_\mathcal{M}$ is the bicategory with:
\begin{itemize}
\item Objects: special symmetric Frobenius algebras in $\mathrm{Ob} \left( \mathcal{M} \right)$;
\item Morphism category: given two special symmetric Frobenius algebras $A,B \in \mathrm{Ob} \left( \mathcal{M} \right)$, the 1- and 2-morphisms are given by the category of $A-B-$bimodules. 
\end{itemize}

\end{defn}

It is possible to attribute a physical interpretation to the morphisms of $\mathrm{Frob}_\mathcal{M}$, that we specify in Table \ref{tab2}.
\begin{table}
\begin{center}
\begin{tabular}{c|c}
$\mathrm{Frob}_\mathcal{M}$ & CFT interpretation \\ \hline
$\mathrm{Hom}\left( I,I \right)$ & Tensor category of chiral data \\
$\mathrm{Hom} \left( I,A \right)$ & Bondary conditions for the full CFT labelled by A \\
$\mathrm{Hom} \left( A,A \right)$ & Topological defects in the full CFT associated to $A$ \\
$\mathrm{Hom} \left( A,A' \right)$ & Topological defect lines which separate \\
& two different CFTs sharing the same chiral data \\ \hline
\end{tabular}
\caption{Mathematical interpretation of chiral rational CFTs}
\label{tab2}
\end{center}
\end{table}

An important proposition concerning the structure of this category is the following:

\begin{prop}{\cite{dkr}}
$\mathrm{Frob}_\mathcal{M}$ is pivotal and has adjoints.
\label{Frobbicatadj}
\end{prop}

Thus we have learnt that the data of a family of full rational CFTs based on the same chiral data are described by a bicategory, and use this description for our purposes. As we already mentioned, topological defects in a full CFT are described by categories of $A$-$A$-bimodules. As we have seen in the previous chapter, (B-type supersymmetry preserving) defects in Landau-Ginzburg models are indeed described by categories of matrix factorizations. Hence, a legitimate question at this point is if one could relate these two in a systematic way. This kind of relation was already noticed in the physics community and quite some works were dedicated to the study this topic. In the next subsection we review some of the most interesting results.

\section{On the Landau-Ginzburg/conformal field theory correspondence}
\label{sec:LGCFTcorrespondence}

Between the late 80's and early 90's a remarkable amount of physics literature focused in the similarities exhibited between Landau-Ginzburg models and $N=2$ superconformal field theories in two dimensions. The chiral symmetries of $N=2$ superconformal field theories are described by the so-called $N=2$ superconformal algebra:

\begin{defn}
The \textit{$N = 2$ Lie superalgebra} (also called \textit{$N=2$ superconformal algebra}) is a super-vector space with the basis of the even central element $c$, even $L_n, J_n$ and odd $G^{\pm}_r$ elements, labelled by integers $n\in\mathbb{Z}$ and half integers $r\in\frac{1}{2}+\mathbb{Z}$ for the \textit{Neveu-Schwarz} $N = 2$ Lie super-algebra $\mathfrak{ns}$ (and by integers $r\in\mathbb{Z}$ for the \textit{Ramond} $N = 2$ Lie super-algebra $\mathfrak{r}$), with brackets
$$\begin{array}{lcl}
[L_m,L_n] & = & (m - n)L_{m+n} + \frac{c}{12}(m^3 - m)\delta_{m+n,0}, \\ \\
{[J_m,J_n]} & = & \frac{c}{3}m\delta_{m+n,0},\\ \\ 
{[ L_n , G^\pm_r]} & = & (\frac{m}{2} - r)G^\pm_{m+r} , \\ \\
{[ J_m , G_r^\pm]} & = & \pm G^\pm_{m+r}, \\ \\
{[ L_m,J_n]} & = & -nJ_{m+n}, \\ \\ 
\{G^-_r, G^+_s\} & = & 2L_{r+s} - (r-s)J_{r+s} + \frac{c}{3}(r^2 - \frac{1}{4})\delta_{r+s,0}.
\end{array}$$
\end{defn}

There are indeed many results indicating that Landau-Ginzburg models and $N=2$ superconformal field theories are connected by a deeper reason. Let us give a short overview (not necessarily in chronological order) of some of them of particular relevance:

\begin{itemize}
\item In \cite{howewest1}, Howe and West show how to compute the fixed point of the Landau-Ginzburg Hamiltonian, the conformal weights of the chiral operators and their OPEs, as well as the central charge. In these computations, performed via the so-called $\epsilon$ expansion \footnote{This strategy was not suitable for application to general (non-homogeneous) $N=2$ potentials, an issue which was solved later in \cite{howewest3} with the Parisi approach.}, an essential role was played by the $N=2$ non-renormalization theorem. Happily, these results totally agree with the corresponding $N=2$ minimal model. Further, in \cite{howewest2}, they compute the two- and three-point functions for $N=2$ minimal models and the chiral correlators for fields whose $U \left( 1 \right)$ charges sum to 1.
\item Martinec \cite{martinec} provided a more algebraic geometric approach to this relation, linking $N=2$ Landau-Ginzburg models and conformal field theories on a K{\"a}hler manifold whose first Chern class vanishes via algebraic varieties defined by the potential in the Landau-Ginzburg model (after previous work by Gepner \cite{gepner}). In \cite{greene}, one can further find a way to represent a large class of Calabi-Yau manifolds in terms of renormalization fixed points of Landau-Ginzburg models.

\item Renormalization group flow provides a map from certain boundary conditions and defects lines in the LG model to certain conformal boundary conditions and defects lines in the CFT, see e.\,g.~\cite{bhls0305, brunrogg1}. In general, charges and correlators of fields in the LG model vary along the flow. However, by $\mathcal N=2$ supersymmetry the charges and correlators in a subsector of the LG model consisting of chiral primary fields are preserved and can be directly compared to their CFT equivalents.

\item In \cite{lerchevafawarner}, the comparison between CFT and LG models is pushed one step forward. There, Lerche, Vafa and Warner compare the ring of chiral fields of CFTs and the Jacobi ring of Landau-Ginzburg models. In addition, they show they also need to be finite-dimensional.
\item In \cite{vafawarner}, Vafa and Warner state the relation between Landau-Ginzburg models and conformal field theories as follows:

\begin{res}

The infrared fixed point of a Landau-Ginzburg model with potential $W$ is a conformal field theory with central charge $c_W$, where
$$c_W:=\sum\limits_{i=1}^n \left( 1-\vert x_i \vert \right)$$
(which is related to the central charge of the Virasoso algebra of the CFT via $c_{\mathrm{Vir}}=3 c_W$) with $\vert x_i \vert \in \mathbb{Q}$ the degrees associated to each variable $x_i$.

\end{res}

And in addition they begin a classification program of conformal theories with the help of singularity theory as described e.g. in \cite{arnold}.
\item Actually, this classification program continued in works by Gannon \cite{gannon} and Gray \cite{gray}, and before these two, in \cite{cecottivafa} -in this last case, within the frame of the classification of $N=2$ supersymmetric theories.
\end{itemize}

All this evidence suggests some kind of relation between modules over vertex operator superalgebras, which model topological defects between two conformal field theories and matrix factorizations, describing defects between Landau-Ginzburg models, as one may already have noticed when comparing Tables \ref{tab1} and \ref{tab2}.

To close up this chapter, let us strongly remark that although it is a clear conjecture that at the level of bulks the subspace of chiral primary bulk fields (for the CFT side) should correspond to the Jacobi ring of the associated Landau-Ginzburg model, up to date there is no clear mathematical conjecture of the Landau-Ginzburg/CFT correspondence at the level of boundaries and defects.

\newpage

\chapter[$N=2$ minimal conformal field theories and matrix bifactorizations of $x^d$]{$\boldsymbol{N{=}2}$ minimal conformal field theories and matrix bifactorizations of $\boldsymbol{x^d}$}

\label{LGCFTTL}

This chapter contains the joint work \cite{drcr} with Alexei Davydov and Ingo Runkel. We describe a tensor equivalence between two categories: one of representations associated to the $N=2$ superconformal algebra and another of permutation-type matrix factorizations.

Let $\kk$ be an algebraically closed field (which can be assumed to be the field $\bC$ of complex numbers).

\section[Categories of representations for $N=2$ minimal super vertex operator algebras]{Categories of representations for $\boldsymbol{N{=}2}$ minimal\newline super vertex operator algebras}

\subsection[Representations of $N=2$ minimal super vertex operator algebras]{Representations of $\boldsymbol{N=2}$ minimal super vertex operator algebras}

Let $V(N{=}2,d)$ be the super vertex operator algebra of the $N=2$ minimal model of central charge $c=\frac{3(d-2)}{d}$, where $d\in\bZ_{\ge 2}$, see \cite{ad1} and e.g.\ \cite{dv,Eholzer:1996zi,Adamovic:1998} for more on $N=2$ superconformal algebras.
Its bosonic part $V(N{=}2,d)_0$ can be identified with the coset $(\widehat{\s\uu(2)}_{d-2}\oplus \widehat{\uu(1)}_4)/\widehat{\uu(1)}_{2d}$ \cite{dv} (see \cite{Carpi:2012va} for a proof in the framework of conformal nets).

Accordingly, the category $\C(N{=}2,d)$ of representations of $V(N{=}2,d)_0$ can be realised as the category of local modules over a commutative algebra $A$ in the product
\begin{align}
\E &= \Rep(\widehat{\s\uu(2)}_{d-2})\  \boxtimes\ \overline{\Rep(\widehat{\uu(1)}_{2d})} \boxtimes\ \Rep(\widehat{\uu(1)}_4) \nonumber \\ 
& = \C(\s\uu(2),{d{-}2})  \boxtimes\ \C(\bZ_{2d},q^{-1}_{2d}) \boxtimes\ \C(\bZ_4,q_4)\ ,
\label{eq:E-def}
\end{align}
see \cite{Frohlich:2003hg}.
Here, for a ribbon category $\C$ the notation $\overline\C$ stands for the tensor category $\C$ with the opposite braiding and ribbon twist. 
The category $\C(\s\uu(2),d{-}2)= \Rep(\widehat{\s\uu(2)}_{d-2})$ is the category of integrable highest weight representations of the affine $\s\uu(2)$ at level $d-2$.
Its simple objects $[l]$ are labelled by $l=0,...,d-2$ and have lowest conformal weight $h_l=\frac{l(l+2)}{4d}$. Their dimensions are $\mathrm{dim}[l] = \frac{\eta^{l+1}-\eta^{-l-1}}{\eta-\eta^{-1}}$
with $\eta = e^{ 2 \pi i /d }$ and their ribbon twists are $\theta_l = e^{2\pi i h_l} \, \id_{[l]}$. The fusion rule of $\C(\s\uu(2),d{-}2)$ is 
$$[k] \otimes [l] ~\simeq~ \bigoplus_{m=|k-l|~\mathrm{step}\,2}^{\min(k+l,2d-4-k-l)} [m] \ .$$
The category $\Rep(\widehat{\uu(1)}_{2d})$ of representations of the vertex operator algebra for $\uu(1)$, rationally extended by two fields of weight $d$, is a pointed fusion category (a fusion category with a group fusion rule) with group $G$ of isomorphism classes of simple objects given by $\bZ_{2d}$. Braided monoidal structures on pointed fusion categories require $G$ to be abelian and
are classified by quadratic functions $q:G\to\bC^*$ \cite{js}. The ribbon twist of $\C(G,q)$ is $\theta_a = q(a)\,\id$. The $q_m$ appearing in \eqref{eq:E-def} are defined as
$q_m:\bZ_m\to\bC^*$ with $q_m(r) = e^\frac{\pi ir^2}{m}$ and $m$ even.

We can label simple objects of $\E$ by $[l,r,s]$, where $l \in \{0,...,d-2\}$, $r \in \bZ_{2d}$ and $s \in \bZ_4$. 
The ribbon twist for $\E$ is given by $\theta_{[l,r,s]} = e^{2 \pi i h_{l,r,s}} \,\id$ with
$$h_{l,r,s} \equiv \frac{l(l+2)}{4d} + \frac{s^2}{8} - \frac{r^2}{4d} \mod \bZ\ .$$ 
The underlying object of the algebra $A$ 
in the product \eqref{eq:E-def}
is $[0,0,0]\oplus[d{-}2,d,2]$. Note that $[d{-}2,d,2]$ is an invertible object of order 2 and ribbon twist 1, so that $[0,0,0]\oplus[d{-}2,d,2]$ has a uniquely defined commutative separable algebra structure. The tensor product with $[d{-}2,d,2]$ has the form
$$[d{-}2,d,2] \otimes [l,r,s] ~\simeq~ [d{-}2{-}l,r{+}d,s{+}2].$$
In particular no simple objects are fixed by tensoring with $[d{-}2,d,2]$ and hence all simple $A$-modules are free:
\beq\lb{moa}
A\otimes[l,r,s] ~\simeq~ A\otimes [d{-}2{-}l,r{+}d,s{+}2] ~\simeq~ [l,r,s]\ \op\ [d{-}2{-}l,r{+}d,s{+}2]\ .
\eeq

Recall that a simple $A$-module is {\em local} if all its simple constituents have the same 
ribbon twist (see \cite{Pareigis,Kirillov:2001ti} and \cite[Cor.\,3.18]{Frohlich:2003hm}).
Thus local $A$-modules correspond to $[l,r,s]$ with even $l+r+s$:
$$h_{d-2-l,r+d,s+2} - h_{l,r,s} = \frac{(d{-}2{-}l)(d{-}l)-l(l{+}2)}{4d} + \frac{(s{+}2)^2-s^2}{8} - \frac{(r{+}d)^2-r^2}{4d} 
=\frac{s-l-r}{2}\ .$$

The fermionic part $V(N{=}2,d)_1$ of $V(N{=}2,d)$ corresponds to the $A$-module 
$$A \otimes [0,0,2] ~\simeq~ [0,0,2]\ \op\ [d-2,d,0]$$ 
so that the simple objects of the NS (R) sector of $\C(N{=}2,d)$ are $A\ot[l,r,s]$ with even (odd) $s$:
$$h_{l,r,s+2} - h_{l,r,s} - h_{0,0,2} = \frac{(s+2)^2-s^2-4}{8} = \frac{s}{2}\ .$$
Denote by $\C(N{=}2,d)_{NS}$ the full subcategory of $\C(N{=}2,d)$ consisting of NS objects, i.e.\ with simple objects of the form $A\ot[l,r,s]$ with even $s$. 
By \eqref{moa} any simple object in $\C(N{=}2,d)_{NS}$ can be written as 
\beq\label{eq:[l,r]-def} 
[l,r] := A\ot[l,r,0] 
\quad \text{with}
\quad l \in \{0,1,\dots,d-2\} ~,~~ r \in \bZ_{2d} ~,~~ l+r \text{ even} \ .
\eeq

\subsection[The structure of $\C(N{=}2,d)_{NS}$ for odd $d$]{The structure of $\boldsymbol{\C(N{=}2,d)_{NS}}$ for odd $\boldsymbol{d}$}

Note that direct sums of objects $[l,r,s]$ with even $l+r+s$ form a ribbon fusion subcategory $\E_{even}$ of $\E$.
It can be characterised as the M\"uger centralizer of $[d{-}2,d,2]$ in $\E$. Recall that the M\"uger centralizer of a subcategory $\D\subset\C$ in a ribbon fusion category is $\{X\in\C\ |\ \theta_{X\ot Y} = \theta_X\ot\theta_Y,\ \forall Y\in\D\}$ \cite{mueger}. 

The induction functor $A\ot-:\E\to {_A\E}$ is a faithful tensor functor. Its restriction to $\E_{even}$ is in addition ribbon, so that
$$\E_{even}~ \xrightarrow{A \ot -} ~ {_A\E}_{even} = {_A\E}^{loc} = \C(N{=}2,d)$$
is a faithful ribbon tensor functor.
For odd $d$ 
the object $[1,d,0]$ lies in $\E_{even}$ and tensor generates a subcategory 
of $\E_{even}$ with simple objects $[l,dl,0],\ l=0,...,d-2$ and the fusion with $[1,d,0]$ given by
\beq\label{eq:E-fusion}
[1,d,0]\otimes [l,dl,0] \simeq 
\begin{cases}  [l{-}1,d(l{-}1),0]\ \op\ [l{+}1,d(l{+}1),0]  &; ~~ 1\leq l < d-2 \\  { [d{-}3,d(d{-}3),0] }
 & ; ~~ l = d-2\end{cases}
 \eeq
Since the last entry in $[l,dl,0]$ is zero,
the restriction of the induction functor $A\ot-$ to this subcategory is fully faithful. Denote by $\T$ its image in $\C(N{=}2,d)$.

The invertible object $[0,2,0]$ belongs to the M\"uger centralizer of $[1,d,0]$ in $\E_{even}$:
$$\exp 2 \pi i \big(h_{1,d+2,0} - h_{1,d,0} - h_{0,2,0}\big) = \exp 2 \pi i\big(\tfrac{(d+2)^2-d^2-4}{4d}\big) = 1 \ .$$
It tensor generates a pointed subcategory $\V$ in $\E_{even}$ equivalent to $\C(\bZ_d,q^{-2}_d)$. The restriction of the induction functor $A\ot-$ to this subcategory is fully faithful.

For $d$ odd, $[1,d] \in \C(N{=}2,d)_{NS}$ and it is straightforward to see that $\C(N{=}2,d)_{NS}$ is tensor generated by $[1,d]$ and $[0,2]$ (recall the notation \eqref{eq:[l,r]-def}). Furthermore, the intersection of the subcategories tensor generated by $[1,d]$ and by $[0,2]$ is trivial. 
Since (the associated bicharacter of) $q^{-2}_d$ is non-degenerate the subcategory $\V$ is non-degenerate as a braided category. 
Hence by M\"uger's centralizer theorem \cite[Prop.\,4.1]{mueger} $\C(N{=}2,d)_{NS}\simeq \T\boxtimes\V$ as ribbon fusion categories.

Finally, we will show that as a tensor category and for odd $d$, $\C(\bZ_d,q^{-2}_{d})$ is equivalent to the category $\V(\bZ_d)$ of $\bZ_d$-graded vector spaces with the trivial associator. The quadratic form $q^{-2}_{d} \in Q(\bZ_d,\bC^*)$ determines the braided tensor structure on $\C(\bZ_d,q^{-2}_{d})$ via the canonical isomorphism from $Q(\bZ_d,\bC^*)$ to the third abelian group cohomology $H^3_{ab}(\bZ_d,\bC^*)$ \cite{js}. The associator on $\C(\bZ_d,q^{-2}_{d})$, i.e.\ the structure as a tensor category, is determined by the image under the homomorphism $H^3_{ab}(\bZ_d,\bC^*)\to H^3(\bZ_d,\bC^*)$. For $d$ odd, this homomorphism is trivial, hence the associator on $\C(\bZ_d,q^{-2}_{d})$ is trivial.

The above discussion is summarized in the following statement.

\bpr\lb{ddn}
For an odd $d$ there is an equivalence of 
	braided
fusion categories $$\C(N{=}2,d)_{NS}\ \simeq\ \T\boxtimes\V(\bZ_d)\ .$$
\epr

\subsection{Universal properties}\label{sec:univ}

Recall the universal property of Temperley-Lieb categories described in Section \ref{sec:TLcats}. Here, we describe a universal property of $\C(N{=}2,d)_{NS}$ for odd $d$ as a tensor category.
	This description makes use of group actions on tensor categories and equivariant objects, which we review in Appendix \ref{sec:equiv+pointed}. In the following proposition, a pointed subcategory of a tensor category $\D$ with underlying group $\bZ_d$ acts by conjugation, and  $\D^{\bZ_d}$ denotes the corresponding tensor category of equivariant objects.

\bth\lb{upn}
Let $d$ be odd. A tensor functor $F:\C(N{=}2,d)_{NS}\to\D$ is determined by 
\begin{itemize}
\item a tensor functor $\V(\bZ_d)\to \D$, 
\item a self-dual object $T=F([1,d])$ in the category $\D^{\bZ_d}$ of quantum dimension $\dim(T)=2\cos\big(\frac{\pi}{d}\big)$ such that the induced functor $\TL_{2\cos(\frac{\pi}{d})}\to \D^{\bZ_d}$ is not faithful.
\end{itemize}
\eth
\bpf
By Proposition \ref{ddn}, the category $\C(N{=}2,d)_{NS}$ is tensor equivalent to the Deligne product $\T\boxtimes\V(\bZ_d).$
By Theorem \ref{ciz}, a tensor functor $F:\T\boxtimes\V(\bZ_d)\to\D$ is determined by a tensor functor $\V(\bZ_d)\to \D$ and a tensor functor $\T\to \D^{\bZ_d}$.

The dimension of $[1,d]\in\T$ (which coincides with the dimension of $[1,0,0]$ in $\E$)
is equal to $2\cos\big(\frac{\pi}{d}\big)$.
The fusion rules of $\T$ (see \eqref{eq:E-fusion}) show that it is freely generated as a tensor category by $[1,d]$, and that the Wenzl--Jones projector $p_{d-1}$ vanishes (Corollary \ref{ftl}). By semi-simplicity, it follows that $\TL_{2\cos(\frac{\pi}{d})}\to \T$ descends to a tensor equivalence $\T_{2\cos(\frac{\pi}{d})}\to \T$. Consequently,
a tensor functor $\T\to \D^{\bZ_d}$ is determined by a self-dual object $T=F([1,d])$ in the category $\D^{\bZ_d}$ with quantum dimension $\dim(T)=2\cos\big(\frac{\pi}{d}\big)$ and such that the induced functor $\TL_{2\cos(\frac{\pi}{d})}\to \D^{\bZ_d}$ is not faithful.
\epf

\section{On matrix factorizations}\label{sec:mf}

\subsection{Categories of matrix factorizations and tensor products}\label{sec:cat-mf}

Recall the facts and notation in Chapter \ref{ch:mfs} for matrix (bi)factorizations. From here on and for the remainder Section \ref{sec:mf} we fix
$$
	S = \bC[x] \quad , \qquad W=x^d \quad , \quad \text{where} \quad
	d \in \bZ \quad , \quad d \ge 2 \ .
$$
For calculations it will often be convenient to describe $\bC[x]$-$\bC[x]$-bimodules as $\bC[x,y]$-left modules $M$. Here, the left action of $p \in \bC[x]$ is by acting on $M$ with $p(x)$ and the right action by acting with $p(y)$. We will employ this tool without further mention.

The tensor unit in $\HMFbi$ is
$$
I ~:\quad\xymatrix{\bC[x,y] \ar@/^10pt/[rrr]^{d_1 = x-y} &&& \bC[x,y] \ar@/^10pt/[lll]^{d_0 = \frac{x^d-y^d}{x-y}}} \quad .
$$
The left and right unit isomorphisms $\lambda_M : I \otimes M \to M$ and $\rho_M : M \otimes I \to M$ are given by
$$\xymatrix{
I\ot M \ar[dd]_{\lambda_M} && 
{\begin{array}{c} {I_1 \otimes M_0} \\  {\oplus} \\ {I_0 \otimes M_1} \end{array} }
\ar[dd]_{(0\ L_{M_1})} \ar@/^10pt/[rrrr]^{\left(\begin{array}{rr} \scriptstyle{(x-y)\ot\id} & {\scriptstyle{1\ot d_1^M}} \\ {\scriptstyle{-1\ot d_0^M}} &  \scriptstyle{\frac{x^d-y^d}{x-y}\ot\id} \end{array}\right)} &&&& 
{\begin{array}{c} {I_0 \otimes M_0} \\  {\oplus} \\ {I_1 \otimes M_1} \end{array} }
\ar[dd]^{(L_{M_0}\ 0)} \ar@/^10pt/[llll]^{\left(\begin{array}{rr}  \scriptstyle{\frac{x^d-y^d}{x-y}\ot\id} & {\scriptstyle{-1\ot d_1^M}} \\ {\scriptstyle{1\ot d_0^M}} & \scriptstyle{(x-y)\ot\id} \end{array}\right)}
 \\ \\ 
M  & & M_1 \ar@/^10pt/[rrrr]^{d_1^M}  &&&& M_0 \ar@/^10pt/[llll]^{d_0^M} 
}$$
\beq\label{eq:unit-isos}
\xymatrix{
M\ot I \ar[dd]_{\rho_M} && 
{\begin{array}{c} {M_1 \otimes I_0} \\  {\oplus} \\ {M_0 \otimes I_1} \end{array} }
\ar[dd]_{(R_{M_1}\ 0)} \ar@/^10pt/[rrrr]^{\left(\begin{array}{rr} {\scriptstyle{d_1^M}\ot 1} & \scriptstyle{\id\ot(x-y)}  \\  \scriptstyle{-\id\ot\frac{x^d-y^d}{x-y}} & {\scriptstyle{d_0^M\ot 1}}  \end{array}\right)} &&&& 
{\begin{array}{c} {M_0 \otimes I_0} \\  {\oplus} \\ {M_1 \otimes I_1} \end{array} }
\ar[dd]^{(R_{M_0}\ 0)} \ar@/^10pt/[llll]^{\left(\begin{array}{rr}  {\scriptstyle{d_0^M\ot 1}} & \scriptstyle{-\id\ot(x-y)} \\  \scriptstyle{\id\ot\frac{x^d-y^d}{x-y}} & {\scriptstyle{d_1^M\ot 1}} \end{array}\right)}
 \\ \\
M  & & M_1 \ar@/^10pt/[rrrr]^{d_1^M}  &&&& M_0 \ar@/^10pt/[llll]^{d_0^M} 
}\eeq	
The maps $L$ and $R$ were, for a given $\bC[x]$-$\bC[x]$-bimodule $N$, already defined in Eq. \ref{leftrightactions}.
It is easy to verify that $\lambda_M$ and $\rho_M$ are in $\ZMFbi$. As we have previously mentioned, they have homotopy inverses, see \cite{carqrunkel1}.

\medskip

Finite rank factorizations in $\HMFbi$ have right duals \cite{carqrunkel3,carqmurfet}.
We will only need explicit duals of matrix factorizations $M \in \HMFbi$ for which $M_0$ and $M_1$ are of rank 1. In this case we have \cite{carqrunkel3}:
$$
M~:\quad
\xymatrix{\bC[x,y] \ar@/^10pt/[rr]^{d_1(x,y)} && \bC[x,y] \ar@/^10pt/[ll]^{d_0(x,y)}}
\qquad \leadsto \qquad
M^+~:\quad
\xymatrix{\bC[x,y] \ar@/^10pt/[rr]^{d_1^{M^+} := -d_1(y,x)} &&\bC[x,y] \ar@/^10pt/[ll]^{d_0^{M^+} := d_0(y,x)}}\quad 
$$
Note that $I^+ = I$. Since the corresponding duality maps play an important role in our construction, we take some time to recall their explicit form and some properties from \cite{carqrunkel3}. The coevaluation  $\mathrm{coev}_M : I \to M \otimes M^+$ is the simpler of the two,
$$
\xymatrix{
I \ar[ddd] & & \bC[x,z] \ar@/^10pt/[rrrr]^{x-z} \ar[ddd]_{\left(\begin{array}{c} \scriptstyle{1} \\ \scriptstyle{1} \end{array}\right)} &&&& \bC[x,z] \ar@/^10pt/[llll]^{\frac{x^d-z^d}{x-z}} \ar[ddd]^{\left(\begin{array}{r} \frac{d_1(x,y)-d_1(z,y)}{x-z} \\ \frac{d_0(x,y)-d_0(z,y)}{x-z} \end{array}\right)} \\ \\ \\
M \otimes M^+ && \bC[x,y,z]^{\op 2} \ar@/^10pt/[rrrr]^{\left(\begin{array}{rr} \scriptstyle{d_1(x,y)} & \scriptstyle{-d_1(z,y)} \\ \scriptstyle{-d_0(z,y)} &  \scriptstyle{d_0(x,y)} \end{array}\right)} &&&& \bC[x,y,z]^{\op 2} \ar@/^10pt/[llll]^{\left(\begin{array}{rr} \scriptstyle{d_0(x,y)} & \scriptstyle{d_1(z,y)} \\ \scriptstyle{d_0(z,y)} & \scriptstyle{d_1(x,y)} \end{array}\right)}
}
$$
Here the left and the right bottom instances of $\bC[x,y,z]^{\op 2}$ correspond to 
$$(M \otimes M^+)_1 = {\begin{array}{c} {M_1 \otimes M^+_0} \\  {\oplus} \\ {M_0 \otimes M^+_1} \end{array} }\qquad ,
\qquad (M \otimes M^+)_0 = {\begin{array}{c} {M_0 \otimes M^+_0} \\  {\oplus} \\ {M_1 \otimes M^+_1} \end{array} }
\quad , 
$$ resp.
It is immediate that this is indeed a morphism in $\ZMFbi$. The evaluation $\mathrm{ev}_M : M^+ \otimes M \to I$ takes the form
$$
\xymatrix{
M^+ \otimes M \ar[ddd] && \bC[x,y,z]^{\op 2} \ar[ddd]_{(B_M\ C_M)} \ar@/^10pt/[rrrr]^{\left(\begin{array}{rr} \scriptstyle{-d_1(y,x)} & \scriptstyle{d_1(y,z)} \\ \scriptstyle{-d_0(y,z)} &  \scriptstyle{d_0(y,x)} \end{array}\right)} &&&& \bC[x,y,z]^{\op 2} \ar[ddd]^{(A_M\ 0)} \ar@/^10pt/[llll]^{\left(\begin{array}{rr} \scriptstyle{d_0(y,x)} & \scriptstyle{-d_1(y,z)} \\ \scriptstyle{d_0(y,z)} &  \scriptstyle{-d_1(y,x)} \end{array}\right)}
 \\ \\ \\
I  & & \bC[x,z] \ar@/^10pt/[rrrr]^{x-z}  &&&& \bC[x,z] \ar@/^10pt/[llll]^{\frac{x^d-z^d}{x-z}} 
}
$$
Here the left and the right top instances of $\bC[x,y,z]^{\op 2}$ correspond to 
$$(M^+ \otimes M)_1 = {\begin{array}{c} {M^+_1 \otimes M_0} \\  {\oplus} \\ {M^+_0 \otimes M_1} \end{array} }\qquad,\qquad (M^+ \otimes M)_0 = {\begin{array}{c} {M^+_0 \otimes M_0} \\  {\oplus} \\ {M^+_1 \otimes M_1} \end{array} }\quad,$$ resp.
The $\bC[x,z]$-module maps $A_M, B_M, C_M$ are defined as follows. The map $C_M$ is simply minus the projection onto terms independent of $y$: $C_M(y^m) = -\delta_{m,0}$. For $A_M$ and $B_M$ we introduce the auxiliary function
$$
	\mathcal{G}_M(f) = \frac{1}{2 \pi i} \oint \frac{x-z-y}{y \, d_1(y,z)} \, f(x,y,z) dy 
	\qquad , \quad f \in \bC[x,y,z] \ .
$$
The contour integration is along a counter-clockwise circular contour enclosing all poles. It is not immediately evident but still true that $\mathcal{G}_M(f)$ is a polynomial. One way to see this is to rewrite $\mathcal{G}_M(f) = \frac{1}{2 \pi i} \oint \frac{x-z-y}{y (y^d-z^d)} \, d_0(y,z)f(x,y,z) dy$ and to expand $(y^d-z^d)^{-1} = \sum_{m=0}^\infty (z/y)^m$. In this way one can rewrite the integrand as a formal Laurent series in $y$ whose coefficients are polynomials in $x,z$. The contour integration picks out the coefficient of $y^{-1}$. 

We will need two further properties of $\mathcal{G}_M$:
\beq\label{eq:G-properties}
	\mathcal{G}_M\big(d_1(y,z) \, y^m\big) = (x-z) \delta_{m,0}
		\quad , \quad
	\mathcal{G}_M\big(d_1(y,x) \, f(x,y,z)\big) \in (x-z) \bC[x,z] \ .
\eeq
The first property is clear. For the second property, let $g(x,z) := \mathcal{G}_M\big(d_1(y,x) \, f(x,y,z)\big)$. The condition $g(z,z)=0$ is then immediate from the first property. 

We can now give the maps $A_M$ and $B_M$:
$$
	A_M(f) = - \mathcal{G}_M(f)
	\quad , \quad
	B_M(f) = \frac{\mathcal{G}_M\big(d_1(y,x)f(x,y,z)\big)}{x-z} \ .
$$
To verify that  $\mathrm{ev}_M \in \ZMFbi(M^+ \otimes M,I)$, it suffices to check $(\mathrm{ev}_M)_0 \circ d^{M^+ \otimes M}_1 = d^I_1 \circ (\mathrm{ev}_M)_1$ on $(y^m,y^n)$ for all $m,n \ge 0$. This is straightforward using \eqref{eq:G-properties}:
\begin{align*}
(\mathrm{ev}_M)_0 \circ d^{M^+ \otimes M}_1(y^m,y^n) &= A_M\big( -d_1(y,x) y^m + d_1(y,z) y^n \big)
=  \mathcal{G}_M(d_1(y,x)y^m) - (x-z) \delta_{n,0} \ ,
\nonumber\\
d^I_1 \circ (\mathrm{ev}_M)_1(y^m,y^n) &= (x-z)(B_M(y^m) + C_M(y^n)) = \mathcal{G}_M(d_1(y,x)y^m) - (x-z) \delta_{n,0} \ .
\nonumber
\end{align*}

The zig-zag identities for $\mathrm{ev}_M$ and $\mathrm{coev}_M$ are verified in \cite[Thm.\,2.5]{carqrunkel3}.

\subsection{Permutation-type matrix bifactorizations}\label{sec:permutation-mf}

In this Subsection we review some facts on permutation-type matrix bifactorizations (recall Definition \ref{permtypembf}). For a subset $S \subset \mathbb{Z}_d$ we will write $\overline{S}=\mathbb{Z}_d \setminus S$.

The bifactorizations $P_\emptyset$ and $P_{\{0,1,\dots,d-1\}}$ are isomorphic to the zero object in $\HMFbi$. The remaining $P_S$ are non-zero and mutually distinct.
To see this, in the following remark we recall a useful tool from $\cite{khovroz}$.

\bre\label{rem:H(M)}
Given a matrix bifactorization $(M,d)$, we obtain a $\bZ_2$-graded complex by considering the differential $\bar d$ on $M / \langle x,y \rangle M$. Since $x^d-y^d \in \langle x,y \rangle$, $\bar d$ is indeed a differential. Denote by $H(M)$ the homology of this complex. Then \cite[Prop.\,8]{khovroz} states that $f \in \HMFbi(M,N)$ is an isomorphism in $\HMFbi$ if and only if the induced map $H(f) : H(M) \to H(N)$ is an isomorphism of $\bC$-vector spaces.
\ere

\ble
Let $R,S \subset \bZ_d$ be nonempty proper subsets. The permutation-type matrix bifactorizations $P_R$ and $P_S$ are non-zero, and they are isomorphic in $\HMFbi$ if and only if $R=S$.
\ele

\bpf
For a non-empty proper subset $S$, the matrix factorization $P_S$ is reduced, that is, the differential $\bar d$ induced on the quotient $P_S/\langle x,y\rangle P_S$ is zero. Thus $H(P_S) \simeq \bC \oplus \bC$. It follows that $f \in \ZMFbi(P_S,P_R)$ is an isomorphism in $\HMFbi$ if and only if $f_0$ and $f_1$ contain a non-zero constant term. Writing out the condition that $f$ is a cycle shows that this is possible only for $R=S$.
\epf

We will mostly be concerned with a special subset of permutation-type bifactorizations, namely those with consecutive index sets. For $a \in \bZ_d$ and $\lambda \in \{0,1,2,\dots,d-2\}$ we write
$$
	P_{a:\lambda} := P_{\{a,a+1,\dots,a+\lambda\}} \ 
$$
We define $\Pd$ to be the full subcategory of $\HMFbi$ consisting of objects isomorphic (in $\HMFbi)$ to finite direct sums of the $P_{a:\lambda}$. A key input in our construction is the following result established in \cite[Sect.\,6.1]{brunrogg1}.

\bth\label{thm:br-decomp}
$\Pd$ is closed under taking tensor products. Explicitly, for $\lambda,\mu \in \{0,\dots,d-2\}$,
$$
	P_{m:\lambda}  \otimes P_{n:\mu} \simeq \bigoplus_{\nu=|\lambda-\mu|~\mathrm{step}\,2}^{\min(\lambda+\mu,2d-4-\lambda-\mu)} P_{m+n-\frac12(\lambda+\mu-\nu):\nu} \quad 
$$
\eth

For the dual of a permutation-type matrix bifactorizations one finds $(P_S)^+ \simeq P_{-S}$. Explicitly:
\beq\label{eq:PS+P-S-iso}
\xymatrix{
 P_{-S} \ar[dd] &&&& \bC[x,y] \ar@/^10pt/[rrr]^{\prod\limits_{j \in S}(x-\eta^{-j} y)} \ar[dd]_{(-1)^{|S|+1}\prod_{j\in S}\eta^{-j}} &&&\bC[x,y] \ar@/^10pt/[lll]^{\prod\limits_{j \in \overline S}(x-\eta^{-j} y)}  \ar[dd]^{1} \\ \\
(P_S)^+ &&&& \bC[x,y] \ar@/^10pt/[rrr]^{-\prod\limits_{j \in S}(y-\eta^{j}x)} &&& \bC[x,y] \ar@/^10pt/[lll]^{\prod\limits_{j \in \overline S}(y-\eta^{j}x)}
}
\eeq
The self-dual permutation-type matrix bifactorizations of the form $P_{a:\lambda}$ therefore have to satisfy $2a \equiv -\lambda \mod d$. Depending on the parity of $d$, one finds:
\begin{itemize}
\item $d$ even: $\lambda$ must be even and $a \equiv \frac\lambda2 \mod d$ or $a \equiv \frac{\lambda+d}2 \mod d$,
\item $d$ odd: $\lambda$ can be arbitrary and $a \equiv \frac{d-1}2 \lambda \mod d$.
\end{itemize}

\subsection[A tensor functor from $\bZ_d$ to $\Pd$]{A tensor functor from $\boldsymbol{\bZ_d}$ to $\boldsymbol{\Pd}$}\label{sec:Zd_to_Pd}

Consider the algebra automorphism $\sigma$ of $\bC[x]$ which acts on $x$ as $\sigma(x) = \eta x$. It leaves the potential $x^d$ invariant and generates the group of algebra automorphisms with this property. We get a group isomorphism
$$
	\bZ_d ~\longrightarrow~ \mathrm{Aut}(\bC[x] \text{ with $x^d$ fixed})
	\quad , \quad k \mapsto \sigma^k \ .
$$
Given a matrix bifactorization $M \in \MFbi$ and $a,b \in \bZ_d$, we denote by ${}_aM_b$ the matrix bifactorization whose underlying $\bC[x]$-bimodule is equal to $M$ as a $\bZ_2$-graded $\bC$-vector space, but has twisted left/right actions ($p \in \bC[x]$, $m \in M$):
$$
	(p,m) \mapsto \sigma^{-a}(p).m \quad , \quad
	(m,p) \mapsto m.\sigma^{b}(p) \ ,
$$
where the dots denotes the left/right action on the original bimodule $M$. Since $\bZ_d$ is abelian, we get a left action even if we were to omit the minus sign in $\sigma^{-a}$, but we include it to match the conventions of \cite[Sect.\,7.1]{carqrunkel2}.
For permutation-type matrix bifactorizations we have isomorphisms:
\beq\label{eq:PS-with-twisted-action-iso-to-PS}
  \xymatrix{
P_{S-a-b}  \ar[dd]_{s_{a,b}} && \bC[x,y] \ar@/^10pt/[rrr]^{\prod\limits_{j \in S}(x-\eta^{j-a-b} y)}  \ar[dd]_{\eta^{-|S|a} \,\cdot\, \sigma^{-a} \otimes \sigma^{b} } &&& \bC[x,y] \ar@/^10pt/[lll]^{\prod\limits_{j \in \overline S}(x-\eta^{j-a-b} y)}  \ar[dd]^{ \sigma^{-a} \otimes \sigma^{b}}
  \\ \\
 {}_a(P_S)_b && {}_a(\bC[x,y])_b \ar@/^10pt/[rrr]^{\prod\limits_{j \in S}(x-\eta^j y)} &&& {}_a(\bC[x,y])_b \ar@/^10pt/[lll]^{\prod\limits_{j \in \overline S}(x-\eta^j y)}  
}
\eeq
Here, $\sigma^{-a} \otimes \sigma^{b}$ is the automorphism of $\bC[x,y]$ which acts as $x \mapsto \eta^{-a} x$ and $y \mapsto \eta^{b}y$. 

The following lemma is straightforward.

\ble\label{lem:a()b-functor}
For all $a,b \in \bZ_d$, 
${}_a(-)_{b}$ defines an auto-equivalence of $\HMFbi$ and of $\Pd$. If $b=-a$, this auto-equivalence is 
tensor with ${}_a(M \otimes N)_{-a} = {}_aM_{-a} \otimes {}_aN_{-a}$ and $s_{a,-a} : I \to {}_aI_{-a}$.
\ele

Consider the objects ${}_aI \in \HMFbi$ for $a \in \bZ_d$. Applying the functor ${}_a(-)$ to the unit isomorphism $\lambda_{{}_bI} : I \otimes {}_bI \to {}_bI$ gives the isomorphism
\beq\label{eq:mu-morph}
	\mu_{a,b}  := {}_a(\lambda_{{}_bI}) ~:~ {}_aI \otimes {}_bI \to 
	{}_{a+b}I \ .
\eeq
By $\underline{\bZ_d}$ we mean the monoidal category whose set of objects is $\bZ_d$, whose set of morphisms consists only of the identity morphisms, and whose tensor product functor is the group operation (i.e.\ addition), see Appendix \ref{ceo}. 

\bpr\label{prop:chi-functor}
$\chi : \underline{\bZ_d} \to \Pd$, $\chi(a) = {}_aI$, together with 
$\mu_{a,b} : \chi(a) \otimes \chi(b) \to \chi(a+b)$, 
defines a tensor functor.
\epr

\bpf
First note that by \eqref{eq:PS-with-twisted-action-iso-to-PS}, ${}_aI \simeq P_{\{-a\}}$, so that indeed $\chi(a) \in \Pd$.
It is shown in \cite[Prop.\,7.1]{carqrunkel3} that the $\mu_{a,b}$ satisfy the associativity condition
$$
	\mu_{a,b+c} \circ (\id_{{}_aI} \otimes \mu_{b,c}) = \mu_{a+b,c} \circ (\mu_{a,b} \otimes \id_{{}_cI}) 
	\qquad \text{for all} \quad a,b,c \in \bZ_d \ .
$$
This amounts to the hexagon condition for the coherence isomorphisms $\mu_{a,b}$.
\epf

We can now construct two tensor functors $\underline{\bZ_d} \to \mathrm{Aut}_\otimes(\Pd)$. The first functor takes $a \in \bZ_d$ to ${}_a(-)_{-a}$; we denote this functor by $A$. This functor is strictly tensor: $A(0)=\Id$ and $A(a) \circ A(b) = A(a+b)$.

The second functor is the adjoint action of $\chi$; we denote it by $\Ad_\chi$. Given $a\in\bZ_d$, on objects the functor $\Ad_\chi(a)$ acts as $M \mapsto \chi(a) \otimes M \otimes \chi(-a)$. The morphism $f : M \to N$ gets mapped to $\id_{\chi(a)} \otimes f \otimes \id_{\chi(-a)}$. The isomorphisms $\mu_{-a,a} : \chi(-a) \otimes \chi(a) \to \chi(0) = I$ give the tensor structure on $\Ad_\chi(a)$. 
So far we saw that for all $a \in \bZ_d$, $\Ad_\chi(a) \in  \mathrm{Aut}_\otimes(\Pd)$. Next we need the coherence isomorphisms $\Ad_\chi(a) \circ \Ad_\chi(b) \to \Ad_\chi(a+b)$. These are simply given by $\mu_{a,b} \otimes (-) \otimes \mu_{-b,-a}$. 

The following lemma will simplify the construction of $\bZ_d$-equivariant structures below.

\ble\label{lem:A-iso-Ad}
$A$ and $\Ad_\chi$ are naturally isomorphic as tensor functors.
\ele

\bpf
We need to provide a natural monoidal isomorphism $\alpha : \Ad_\chi \to A$. That is, for each $a \in \bZ_d$ we need to give a natural monoidal isomorphism $\alpha(a) : \Ad_\chi(a) \to A(a)$, such that the diagram
\beq\label{eq:alpha-mon}
	\xymatrix{ 
	\Ad_\chi(a) \circ \Ad_\chi(b) \ar[d]_-{\alpha(a) \circ \alpha(b)} \ar[r]^-{\mu_*} & \Ad_\chi(a+b) \ar[d]^-{\alpha(a+b)}
	\\
	 A(a) \circ A(b)  \ar@{=}[r] & A(a+b)
	}
\eeq
commutes, where $\mu_* := \mu_{a,b} \otimes (-) \otimes \mu_{-b,-a}$. Define
$$
	\alpha(a)_M 
	:=
	\big[ {}_aI \otimes M \otimes {}_{-a}I 
		\xrightarrow{ {}_a(\lambda_M) \otimes (s^{-1}_{-a,a})_{-a} }
		{}_aM \otimes I_{-a}
		\xrightarrow{ {}_a( \rho_M )_{-a} }
		{}_aM_{-a}
	\big] \ .
$$

\noindent
{\em $\alpha(a)$ is tensor:} we need to verify commutativity of
$$
	\xymatrix{ 
	\Ad_\chi(a)(M) \otimes \Ad_\chi(a)(N) \ar[d]_-{\alpha(a)_M \otimes \alpha(a)_N} \ar[r]^-{\sim} & \Ad_\chi(a)(M \otimes N) \ar[d]^-{\alpha(a)_{M \otimes N}}
	\\
	 A(a)(M) \otimes A(a)(N)  \ar@{=}[r] & A(a)(M \otimes N)
	}
$$
where the top isomorphism is $\id \otimes \mu_{-a,a} \otimes \id$. Commutativity of this diagram is a straightforward calculation if one notes the following facts: $M \otimes {}_{-a}I = M_{-a} \otimes I$ and $M_{-a} \otimes {}_aN = M \otimes N$ (equal as matrix factorizations, not just isomorphic), and
$$
	\big[ M \otimes {}_{-a} I \xrightarrow{\id \otimes (s^{-1}_{-a,a})_{-a}} M \otimes I_{-a} \xrightarrow{ (\rho_M)_a} M_{-a} \big] ~=~ \big[ M_{-a} \otimes I \xrightarrow{ \rho_{M_{-a}} } \big] \ .
$$

\noindent
{\em $\alpha$ satisfies \eqref{eq:alpha-mon}:}
One way to see this is to act on elements. The unit isomorphisms \eqref{eq:unit-isos} are non-zero only on summands in the tensor products involving $I_0$, in which case they act as
$$
	\lambda_M ~:~ p(x,y) \otimes m \mapsto p(x,x).m
	\quad , \quad
	\rho_M ~:~ m \otimes p(x,y)\mapsto m.p(x,x) \ .
$$
One verifies that the top and bottom path in \eqref{eq:alpha-mon} amount to mapping
$$
	u(x,y) \otimes v(x,y) \otimes m \otimes p(x,y) \otimes q(x,y)
	~\in~ ({}_aI)_0 \otimes ({}_bI)_0 \otimes M \otimes ({}_{-b}I)_0 \otimes ({}_{-a}I)_0
$$
to
$$
	\big\{ \sigma^{-b}(u(x,x))  \, v(x,x) \big\} \,.\,m\,.\, \big\{\sigma^{-b}(p(x,x)) \, \sigma^{-a-b}(q(x,x)) \big\}
	~ \in ~ {}_{a+b}M_{-a-b} \ .
$$
\epf

\subsection[A functor from $\TL_\kappa$ to $\bZ_d$-equivariant objects in $\Pd$]{A functor from $\boldsymbol{\TL_\kappa}$ to $\boldsymbol{\bZ_d}$-equivariant objects in $\boldsymbol{\Pd}$}\label{sec:TL-to-PdZd}

We write $\Pd^{\bZ_d}$ for the category of $\bZ_d$-equivariant objects in $\Pd$, where the $\bZ_d$ action is given by the functor $A$ defined in the previous section. The definition and properties of categories of equivariant objects are recalled in Appendix \ref{ceo}.

By Theorem \ref{upn}, our aim now is to find a tensor functor
$$
	F : \TL_\kappa \to  \Pd^{\bZ_d} \ .
$$ 
According to Section \ref{sec:TLcats}, to construct a functor out of $\TL_\kappa$, we need to give a self dual object, duality maps, and compute the resulting constant $\kappa$. We will proceed as follows:
\begin{enumerate}
\item Give a self dual object $T \in \Pd$.
\item Give duality maps $u,n$, show they satisfy the zig-zag identities \eqref{eq:zigzag}, and compute $\kappa$.
\item Put a $\bZ_d$-equivariant structure on $T$ and show that the maps $u,n$ are $\bZ_d$-equivariant.
\end{enumerate}

\noindent
{\em Step 1:}
we listed self-dual objects of the from $P_{a:\lambda}$ at the end of Section \ref{sec:permutation-mf}. 
By Theorem \ref{thm:br-decomp}, there are only two choices which match the tensor products required by Corollary \ref{ftl}. In both cases, $d$ is odd, and either $\lambda=1$, $a = (d-1)/2$ or $\lambda=d-3$, $a = (d-3)(d-1)/2$. Both choices can be used in the construction below; we will work with the first option:
$$
	\text{$d$ odd}
	\quad , \qquad T := P_{\frac{d-1}2:1} = P_{\{\frac{d-1}2,\frac{d+1}2\}} \ .
$$	
Explicitly,
$$
T ~:\quad \xymatrix{\bC[x,y] \ar@/^10pt/[rrr]^{K(x,y)} &&& \bC[x,y] \ar@/^10pt/[lll]^{\frac{x^d-y^d}{K(x,y)}}}\quad ,
$$
where
$$
K(x,y)=\left( x-\eta^{\frac{d-1}{2}}y \right)\left( x-\eta^{\frac{d+1}{2}}y \right)=x^2+y^2+\kappa xy,\qquad \kappa = -(\eta^{\frac{d-1}{2}} + \eta^{\frac{d+1}{2}}) = 2 \cos  \tfrac{\pi}{d}  \ .
$$
Writing $\kappa$ for the coefficient of $xy$ will be justified below, where we will find it to be the parameter in $\TL_\kappa$.

\medskip
\noindent
{\em Step 2:}
denote the isomorphism given in \eqref{eq:PS+P-S-iso} by $t : T \to T^+$, $t = (\id,-\id)$. Define maps $u : T \otimes T \to I$ and $n : I \to T \otimes T$ via
\beq\label{eq:u-n-def}
	u = \big[ T \otimes T \xrightarrow{t \otimes \id} T^+ \otimes T \xrightarrow{\mathrm{ev}_T} I  \big]
	\quad , \quad
	n = \big[ I \xrightarrow{\mathrm{coev}_T} T \otimes T^+ \xrightarrow{\id \otimes t^{-1}} T \otimes T \big] \ .
\eeq
{}From this one computes $u \circ n = \kappa$. For example,
$$
	u_0 \circ n_0 = A_T(x+z+\kappa y) = \kappa \ .
$$
Together with the zig-zag identities for $\mathrm{ev}_T$ and $\mathrm{coev}_T$ established in \cite{carqrunkel3} we have proved:

\bpr
$u$ and $n$ are morphisms in $\ZMFbi$. They satisfy the zig-zag identities in $\HMFbi$, as well as $n \circ u = \kappa$.
\epr

\noindent
{\em Step 3:}
we can make the $P_S$ $\bZ_d$-equivariant via
\beq\label{eq:PS-Zd-equiv-structure}
	\tau_{S;a} : P_S \to {}_{a}(P_S)_{-a}
	\quad , \qquad
	\tau_{S;a} = \eta^{\frac{d+1}2 \,a (|S|-1)} \, s_{a,-a} \ ,
\eeq
where $s_{a,-a}$ was given in \eqref{eq:PS-with-twisted-action-iso-to-PS}. 
These maps satisfy ${}_a(\tau_{S;b})_{-a} \circ \tau_{S;a} = \tau_{S;a+b}$, as required (cf.\ Appendix \ref{ceo}).
Note that on $I=P_{\{0\}}$, the above $\bZ_d$-equivariant structure is just $s_{a,-a} : I \to {}_aI_{-a}$, in agreement with the one on the tensor unit of $\Pd^{\bZ_d}$ as prescribed by Lemma \ref{lem:a()b-functor} and Proposition \ref{prop:CG-is-tensor}.

\ble\label{lem:ev-coev-t-Zd-equiv}
The maps $\mathrm{ev}_{P_S}$ and $\mathrm{coev}_{P_S}$ composed with the isomorphism $P_{-S} \simeq (P_S)^+$ from \eqref{eq:PS+P-S-iso} are $\bZ_d$-equivariant.
\ele

\bpf
For $\mathrm{coev}$ we need to check commutativity of
$$
\xymatrix{ I  \ar[rr]^-{\mathrm{coev}_{P_S}} \ar[d]_-{s_{a,-a}} && P_S \otimes (P_S)^+ \ar[r]^-{\sim} & P_S \otimes P_{-S} \ar[d]^-{\tau_{S;a} \otimes \tau_{-S;a}}
\\
{}_aI_{-a} \ar[rr]^-{{}_a(\mathrm{coev}_{P_S})_{-a}} && {}_a(P_S \otimes (P_S)^+)_{-a} \ar[r]^-{\sim} & 
\begin{minipage}{9em} 
${}_a(P_S)_{-a} \otimes {}_a(P_{-S})_{-a}$ \newline
$=\, {}_a(P_S \otimes P_{-S})_{-a}$   
\end{minipage}
}
$$
which follows straightforwardly by composing the various maps. The corresponding diagram for $\mathrm{ev}$ is checked analogously.
\epf

\bco
$u$ and $n$ are $\bZ_d$-equivariant morphisms.
\eco

According to Section \ref{sec:univ},
at this point we proved the existence of the tensor functor $\TL_\kappa \to  \Pd^{\bZ_d}$. To describe its image and to show that it annihilates the non-trivial tensor ideal in $\TL_\kappa$, we need to introduce a graded version of the above construction.

\subsection{Graded matrix factorizations}

There are several variants of graded matrix factorizations, see e.g.\ \cite{khovroz,Hori:2004ja,Wu09,carqrunkel1}.
The following one is convenient for our purposes. 
We take the grading group to be $\bC$, which  is natural from the relation to the R-charge in conformal field theory, but other groups are equally possible. For example, to construct the tensor equivalence in Theorem \ref{thm:main}, the grading group $d^{-1} \bZ$ is sufficient.

\bde
Let $S$ be a $\bC$-graded $\kk$-algebra such that $W \in S$ has degree 2. A {\em $\bC$-graded matrix factorization} of $W$ over $S$ is a matrix factorization $(M,d)$ of $W$ over $S$ such that the $S$ action on $M$ is compatible with the $\bC$-grading and $d$ has $\bC$-degree 1. That is, if $q(s)$ (resp.\ $q(m)$) denotes the $\bC$-degree of a homogeneous element of $S$ (resp.\ $M$), then $q(s.m) = q(s)+q(m)$ and $q(d(m)) = q(m)+1$.
\ede

In analogy with Section \ref{sec:cat-mf} we define 
$\MFgr_{S,W}$, $\ZMFgr_{S,W}$ and $\HMFgr_{S,W}$
to have $\bC$-graded matrix factorizations as objects and only $\bC$-degree zero morphisms. For example, 
\begin{align*}
	\HMFgr_{S,W}(M,N) ~=~ & \big\{ f \in \ZMF_{S,W}(M,N) \,| \text{ $f$ has $\bC$-degree $0$ }\big\} 
	\nonumber \\
	& ~ /~ \big\{  \delta(g)\, \big| \text{ $g : M \to N$ is $S$-linear, $\bZ_2$-odd and of $\bC$-degree $-1$ } \big\} \ .
	\nonumber
\end{align*}
The same definitions apply to matrix bifactorizations, giving categories $\MFgr_{\mathrm{bi};S,W}$, etc. Under tensor products, the $\bC$-degree is additive. 

\medskip

We will again restrict our attention to the case $S = \bC[x]$ and $W = x^d$, so that $q(x) = \frac2d$.

\medskip

As an example, let us describe all $\bC$-gradings on permutation-type matrix bifactorizations. The $\bC$-grading on $\bC[x,y]$ is fixed by choosing the degree of 1. Let thus $\bC[x,y]\{\alpha\}$ be the graded $\bC[x]$-$\bC[x]$-bimodule with $q(1) = \alpha$. The possible $\bC$-gradings on $P_S$ are
$$
P_{S}\{\alpha\} ~:\quad\xymatrix{\bC[x,y]\{\alpha+\frac2d|S| -1\}  \ar@/^10pt/[rrr]^{d_1 = \prod\limits_{j \in S}(x-\eta^j y)} &&& \bC[x,y]\{\alpha\}\ar@/^10pt/[lll]^{d_0 = \prod\limits_{j \in \overline S}( x-\eta^j y)}} \quad .
$$
The unit isomorphism $\lambda_M$ given in \eqref{eq:unit-isos} above becomes a morphism in $\HMFbigr$ precisely if the unit object is $\bC$-graded as
$$
	I = P_{\{0\}}\{0\} \ .
$$
To see this note that $x^m y^n \in I_0 = \bC[x,y]$ will act as a degree $2(m+n)/d$ map on $M$. With this charge assignment for $I$, $\HMFbigr$ is tensor.

Next we work out the grading on $M^+$ for $M$ with $M_0$ and $M_1$ of rank 1. One first convinces oneself that for a homogeneous $p \in \bC[x,y,z]$ we have $\mathrm{deg}(A_M(p)) = \mathrm{deg}(p) - \mathrm{deg}(d_1^M(x,y))+1$, where $\mathrm{deg}$ denotes the polynomial degree. So if $M_0 = \bC[x,y]\{\alpha\}$, for $A_M$ to give a $\bC$-degree 0 map, we need $M^+_0 = \bC[x,y]\{-\alpha+\frac2d(1-\mathrm{deg}(d^M_1))\}$ (cf.\ \cite[Sect.\,2.2.4]{carqrunkel3}). This forces the $\bC$-grading to be
\begin{align*}
M~&:~
\xymatrix{\bC[x,y]\{\alpha+\frac2d\mathrm{deg}(d_1)-1\} \ar@/^10pt/[rr]^{d_1(x,y)} && \bC[x,y]\{\alpha\} 
\ar@/^10pt/[ll]^{d_0(x,y)}}
\nonumber\\
\leadsto\ M^+~&:~
\xymatrix{\bC[x,y]\{-\alpha-1+\frac2d\} \ar@/^10pt/[rr]^{d_1^{M^+} := -d_1(y,x)} &&\bC[x,y]\{-\alpha+\frac2d(1-\mathrm{deg}(d_1))\} \ar@/^10pt/[ll]^{d_0^{M^+} := d_0(y,x)}}\quad .
\end{align*}
One can check that $\mathrm{ev}$ and $\mathrm{coev}$ are indeed degree 0 maps with respect to these gradings. Note that we have $I^+=I$ also as graded matrix bifactorizations.

In the next section we will be interested in the $P_S\{\alpha\}$ with $\alpha = \frac{1-|S|}d$. We abbreviate these as $\hat P_S$. This subset of graded permutation-type matrix bifactorizations 
is closed under taking duals:
$$
	(\hat P_S)^+ \simeq \hat P_{-S}
	\qquad , \quad \text{where} \quad \hat P_S = P_S\big\{\tfrac{1-|S|}d\big\} \ .
$$
An explicit isomorphism is again given by \eqref{eq:PS+P-S-iso}, which is easily checked to have $\bC$-degree 0.

The next two lemmas show that the $\hat P_S$ generate (under direct sums) a semi-simple subcategory of $\HMFbigr$.

\ble
$\ZMFbigr(\hat P_{R}, \hat P_{S})$ is $\bC \id$ if $R=S$ and $0$ else.
\ele

\bpf
Write $\alpha = \frac{1-|R|}d$ and $\beta = \frac{1-|S|}d$, such that $\hat P_R = P_R\{\alpha\}$ and $\hat P_S = P_S\{\beta\}$. 
The morphism space $\ZMFbi(P_R,P_S)$ is given by all $(p,q)$ with $p,q \in \bC[x,y]$ such that $p \cdot d^{P_R}_1 = d^{P_S}_1 \cdot q$. For $(p,q)$ to be also in $\ZMFbigr(P_R\{\alpha\},P_S\{\beta\})$, we need  $p,q$ to be homogeneous and $\alpha = \beta + \frac2d\deg(p)$ and $\alpha + \frac2d |R|-1 = \beta + \frac2d |S|-1 + \frac2d\deg(q)$. This simplifies to $2 \deg(p) = |S|-|R|$ and $2 \deg(q) = |R|-|S|$, which is possible only for $|R|=|S|$, in which case $p,q$ are constants.
Finally, the condition $p \cdot d^{P_R}_1 = d^{P_S}_1 \cdot q$ has non-zero constant solutions only if $R=S$.
\epf

\ble \label{lem:hmfgrPSPR}
$\hat P_\emptyset$ and $\hat P_{\bZ_d}$ are zero objects in $\HMFbigr$.
For $R,S \neq \emptyset,\bZ_d$ we have $\HMFbigr(\hat P_{R}, \hat P_{S})=\ZMFbigr(\hat P_{R}, \hat P_{S})$.
\ele

\bpf
That $\hat P_\emptyset$ and $\hat P_{\bZ_d}$ are zero objects in $\HMFbigr$ follows since one component of the twisted differential is $1$, and hence there is a contracting homotopy for the identity morphism. 

Let now $R,S$ be nonempty proper subsets of $\bZ_d$.
For the second part of the statement one checks that there are no $\bZ_2$-odd morphisms of $\bC$-degree $-1$ from $\hat P_R$ to $\hat P_S$. For example, a $\bC$-degree $-1$ map $\psi_0 : (\hat P_{R})_0 \to (\hat P_{S})_1$ has to satisfy
$$
  \frac{1+|S|}{d}-1+\frac{2\,\mathrm{deg}(\psi_0(x,y))}{d} - \frac{1-|R|}{d} = -1 \ ,
$$
where $\mathrm{deg}(\psi_0)$ is the polynomial degree of $\psi_0(x,y)$. Thus $\mathrm{deg}(\psi_0) =- \frac{|S|+|R|}2$, and $\psi_0$ can be non-zero only if $|R|=|S|=0$. An analogous computation for $\psi_1$ shows $\mathrm{deg}(\psi_1) = \frac{|S|+|R|}2 - d$, and so $\psi_1$ can be non-zero only if $|R|=|S|=d$.
\epf

We now focus on the graded matrix factorizations $\hat P_{a:\lambda}$, i.e.\ the $P_S\{\alpha\}$ with consecutive set $S = \{a,a{+}1,\dots,a{+}\lambda\}$ and $\alpha = - \lambda/d$. We define
$$
	\Pdgr = \big\langle \hat P_{a:\lambda} \,\big|\, a \in \bZ_d, \lambda \in \{0,\dots,d-2\} \big
	\rangle_\oplus
	~\subset~ \HMFbigr \ ,
$$
i.e.\ the full subcategory of $\HMFbigr$ consisting of objects isomorphic, in $\HMFbigr$, to finite direct sums of the $\hat P_{a:\lambda}$.

Then we need to check whether the decomposition of tensor products in Theorem \ref{thm:br-decomp} carries over to the graded case. This could be done by adapting the method used in \cite{brunrogg1}, which works in the stable category of $\bC[x,y]/\langle x^d-y^d \rangle$ modules. We give a related but different proof by providing explicit $\bC$-charge 0 embeddings of the direct summands in the decomposition of $\hat P_{a:1} \otimes \hat P_{b:\lambda}$ and proving that they give an isomorphism via Remark \ref{rem:H(M)}. This is done in Appendix \ref{app:graded-tensor}.

\bth\label{thm:Pdgr-tensor-closed}
The category $\Pdgr$ is semi-simple with simple objects $\hat P_{a:\lambda}$, $a \in \bZ_d$ and $\lambda \in \{0,1,\dots,d-2\}$. It is closed under tensor products and the direct sum decomposition of $\hat P_{m:\lambda} \otimes \hat P_{n:\nu}$ in $\HMFbigr$ is as in Theorem \ref{thm:br-decomp}.
\eth

\subsection[A functor from $\TL_\kappa$ to $\bZ_d$-equivariant objects in $\Pdgr$]{A functor from $\boldsymbol{\TL_\kappa}$ to $\boldsymbol{\bZ_d}$-equivariant objects in $\boldsymbol{\Pdgr}$}

The morphisms $\mu_{a,b}$ in \eqref{eq:mu-morph} have $\bC$-degree 0. The functor $\chi$ in Proposition \ref{prop:chi-functor} therefore also defines a tensor functor
$$
	\chi : \underline{\bZ_d} \longrightarrow \Pdgr \ .
$$
As in Section \ref{sec:Zd_to_Pd} we obtain two tensor functors $A , \Ad_\chi : \underline{\bZ_d} \to \mathrm{Aut}_\otimes(\Pdgr)$. The natural monoidal isomorphism $A \to \Ad_\chi$ established in Lemma \ref{lem:A-iso-Ad} uses only $\bC$-degree 0 morphisms. 

\medskip

Next we follow the three steps in Section \ref{sec:TL-to-PdZd} and verify that they carry over to the $\bC$-graded setting. Consider the self-dual object $\hat T \in \Pdgr$. The duality maps $n$ and $u$ from \eqref{eq:u-n-def} are of $\bC$-degree 0 since $t$, $\mathrm{ev}_T$, $\mathrm{coev}_T$ are. The maps $\tau$ from \eqref{eq:PS-Zd-equiv-structure} are equally of $\bC$-degree 0 and hence equip $\hat T$ with a $\bZ_d$-equivariant structure. The proof of Lemma \ref{lem:ev-coev-t-Zd-equiv} still applies and shows that $u$ and $n$ are $\bZ_d$-equivariant morphisms in $\HMFbigr$.

\medskip

By Section \ref{sec:TLcats} the data $\hat T$, $\tau$, $u$ and $n$ determine a tensor functor
$$
	F : \TL_\kappa \to  (\Pdgr)^{\bZ_d} \ .
$$ 
Here we used that $\hat T \in \Pdgr$ and that by Theorem \ref{thm:Pdgr-tensor-closed}, $\Pdgr$ is a full tensor subcategory of $\HMFbigr$.

\bth\label{thm:main}
There is a tensor equivalence $G : \C(N{=}2,d)_{NS} \to \Pdgr$ such that $G([l,l+2m]) \simeq \hat P_{m:l}$.
\eth

\bpf 
Corollary \ref{ftl} and the tensor product established in Theorem \ref{thm:Pdgr-tensor-closed} show that $F$ is not faithful and induces a fully faithful embedding $\tilde F:\T_\kappa \to  (\Pdgr)^{\bZ_d}$. By Theorem \ref{upn} the embedding $\tilde F$ gives rise to the functor $G: \C(N{=}2,d)_{NS}\to \Pdgr$.
The functor $G$ is fully faithful (it sends simple objects to simple objects) and surjective on (simple) objects.
Thus $G$ is an equivalence.

Recall that the $\bZ_d$-action on $\Pdgr$ is such that $a \in \bZ_d$ gets mapped to ${}_aI \cong P_{\{-a\}}$, and that $\tilde F$ maps $T \in \T_\kappa$ to $\hat P_{\frac{d-1}2:1} \in \Pdgr$. We choose the monoidal embedding $\underline \bZ_d \to \C(N{=}2,d)$ as $a \mapsto [0,-2a]$ (to avoid this minus sign, one can define $\chi$ in Proposition \ref{prop:chi-functor} as $\chi(a) = {}_{-a}I$, resulting in lots of minus signs in other places). The induced tensor functor $G$ obeys $G([1,d]) = \hat P_{\frac{d-1}2:1}$ and $G([0,2a]) = P_{\{a\}}$.
\epf

\bre
Note that one can replace $\eta$ with any other primitive $d$'th root
of unity $\eta^l$ (here $l$ is coprime to $d$).
In particular replacing $\eta$ with $\eta^l$ in \eqref{permtypembf} gives another
matrix bifactorization,
$P_S(\eta^l)$.
It is not hard to see that $P_{\{\frac{d-1}2,\frac{d+1}2\}}(\eta^l)$ is a
self-dual object of dimension $\kappa_l=2 \cos\tfrac{\pi l}{d}$
and defines a fully faithful embedding $\T_{\kappa_l}\to \HMFbi$.
Its image is additively generated by the direct summands in tensor powers of $P_{\{\frac{d-1}2,\frac{d+1}2\}}(\eta^l)$ and can be computed explicitly from Theorem \ref{thm:Pdgr-tensor-closed} with $\eta^l$ in place of $\eta$. This is an instance of the action of a Galois group on categories of matrix factorization, see \cite[Rem.\,2.9]{crcr} for a related discussion.
\ere

\newpage
\chapter{Orbifold equivalence}
\label{ch:orbeq}

\section{Orbifold completion}

In this section we will briefly recall the highlights of \cite{carqrunkel2}, where a proof of all the results we present can be found in detail. This paper presents a description in bicategorical terms of a generalization of the orbifolding procedure\footnote{By orbifolding, we mean the following procedure, which is standard in (quantum) field theory, string theory, algebraic geometry and representation theory. Consider a theory (let us be slightly vague in this regard) together with a symmetry group. Gauging this symmetry amount to restricting to the invariant sectors while simultaneously adding new twisted sectors. In this way, the orbifold theory is constructed from the original one, and it often inherits nice properties from the symmetry group.}, especially for the case of topological field theories and with special attention to the case of 2 dimensions. They develop a general framework which takes a bicategory $\mathcal{B}$ as an input and returns its so-called orbifold completion $\mathcal{B}_{orb}$. This completion supports the construction of new equivalences of 1-morphisms of our interest. Then, they focus on the example of the bicategory of Landau-Ginzburg models and apply there this framework.

Let $\mathcal{B}$ a bicategory whose morphism categories are idempotent complete.

\begin{defn}
The \textit{equivariant completion of} $\mathcal{B}$, denoted $\mathcal{B}_{eq}$, is the bicategory with:
\begin{itemize}
\item Objects: pairs $\left( a,A \right)$ where $a \in \mathrm{Ob}\left(\mathcal{B}\right)$ and $A \in \mathcal{B} \left( a,a \right)$ is a separable Frobenius algebra;
\item 1-morphisms: given two objects $\left( a,A \right)$ and $\left( b,B \right)$, a 1-morphism $X \colon \left( a,A \right) \to \left( b,B \right) \in \mathcal{B} \left( a,b \right)$ is a $B$-$A$-bimodule;
\item 2-morphisms: 2-morphisms in $\mathcal{B}$ that are bimodule maps;
\item Composition: let $X \colon \left( a,A \right) \to \left( b,B \right)$ and $Y \colon \left( b,B \right) \to \left( c,C \right)$ be two 1-morphisms; their composition is given by the tensor product $X \otimes_B Y \colon \left( a,A \right) \to \left( c,C \right)$. Composition of 2-morphisms is that of $\mathcal{B}$. The associator in $\mathcal{B}_{eq}$ is the one induced from $\mathcal{B}$ as well.
\item Unit 1-morphism: for $\left( a,A \right) \in \mathcal{B}_{eq}$ is $A$. The left and right unit action on $X \colon \left( a,A \right) \to \left( b,B \right)$ is given by the left and right action on the corresponding bimodule.

\end{itemize}
\end{defn}

\begin{rem}
Note here that $\mathcal{B}$ fully embeds in $\mathcal{B}_{eq}$. In addition, this bicategory also satisfies that $\left( \mathcal{B}_{eq} \right)_{eq} \cong \mathcal{B}_{eq}$.
\end{rem}

Note here that in general, the tensor product over a given algebra $A$ may not exist, but in the following lemma we provide an existence criterion which suffices for our purposes.

\begin{lem}
Suppose that $A$ is separable Frobenius. Then $X \otimes_A Y$ exists for all modules $X,Y$.
\end{lem}

Before we continue, we demand in addition that $\mathcal{B}$ has adjoints satisfying $^\dagger X=X^\dagger$ for all 1-morphisms $X$, and that $\mathcal{B}$ is pivotal. With a bit of work, it is possible to prove the following theorem in $\mathcal{B}_{eq}$.
\begin{thm}
Let $X \in \mathcal{B}\left( a,b \right)$ have invertible right quantum dimension. Then $A:=X^\dagger \otimes X$ is a symmetric separable Frobenius algebra in $\mathcal{B} \left( a,a \right)$ and $X \colon \left( a,A \right) \leftrightarrows \left( b,I_b \right)\colon X^\dagger$ is an adjoint equivalence in $\mathcal{B}_{eq}$. 
\label{adjointequivalence}
\end{thm}

Note here that $A$ not only needs to be separable but also symmetric. Then, introduce the following related bicategory.
\begin{defn}
The \textit{orbifold completion $\mathcal{B}_{orb}$ of $\mathcal{B}$} is the full subbicategory of $\mathcal{B}_{eq}$ whose objects are pairs $\left( a,A \right)$ where $a \in \mathrm{Ob}\left(\mathcal{B}\right)$ and $A$ a separable symmetric Frobenius algebra. We refer to the objects of this subbicategory as \textit{generalized orbifolds}.  
\end{defn}

Notice here that $\mathcal{B} \subset \mathcal{B}_{orb} \subset \mathcal{B}_{eq}$. The orbifold completion $\mathcal{B}_{orb}$ is pivotal as well, and the same results for $\mathcal{B}_{eq}$ hold for this case: one can show that $\left( \mathcal{B}_{orb} \right)_{orb} \cong \mathcal{B}_{orb}$, and in addition, an analogous theorem of \ref{adjointequivalence} holds for this subbicategory. It is precisely this result what will be most useful for us. To finish this review, let us focus our attention on the special example of $\mathcal{LG}$. In this bicategory, one can state the following proposition as a consequence of Theorem \ref{adjointequivalence}. 

\begin{prop}
Let $X \colon \left( R,W \right) \to \left( S,V \right)$ (with $R=\mathbbm{k} \left[ x_1,\ldots,x_n \right]$, $S=\mathbbm{k} \left[ z_1,\ldots,z_m \right]$) with invertible right quantum dimension in $\mathcal{LG}$. Then $m=n$ mod 2, $^\dagger X \cong X^\dagger$ and $$X^\dagger \otimes \left( - \right) \colon \mathrm{hmf}_{S,V}^\omega \cong \mathrm{mod} \left( X^\dagger \otimes X \right) \colon X \otimes \left( - \right)$$
where by $\mathrm{mod} \left( X^\dagger \otimes X \right)$ we mean the category of $X^\dagger \otimes X$-modules.
\label{adjointequivalenceformfs}
\end{prop}

There is some kind of polynomials which we will pay careful attention to: that of simple singularities. They have an ADE classification (\cite{arnold}), and as we have mentioned in \ref{sec:LGCFTcorrespondence}, it was conjectured in the physics literature (\cite{martinec,vafawarner, howewest1,howewest2,howewest3}) that Landau-Ginzburg models with these potentials correspond to $N=2$ minimal conformal field theories. These rational CFTs are known to be (generalized) orbifolds of each other (\cite{gray,ffrs}), hence it makes sense to conjecture some equivalences of categories relating these potentials. This was a starting point which inspired the next section.

\section{Orbifold equivalent potentials}
\label{sec:introductionpart1}

This section contains the paper \cite{crcr}, where our purpose was two-fold. Firstly, we defined an equivalence relation on the set $\mathcal{P}_{\mathbbm{k}}$ of all potentials, where the number of variables in the polynomial ring is allowed to vary, but where the field $\mathbbm{k}$ is kept fixed. Secondly, in the case $\mathbbm{k}=\mathbbm{C}$ we listed all equivalence classes in the restricted set of simple singularities, leading to new equivalences of categories predicted by the LG/CFT correspondence. Let us first give a flavour of the highlights of this joint work with Nils Carqueville and Ingo Runkel.

The equivalence relation between potentials is defined by the existence of a matrix factorization with certain properties. Let us write $x$ for the sequence of variables $x_1,\ldots,x_m$.

\begin{defthm}\label{defthm:MF-orbequiv}
We say that two homogeneous potentials $V \left( x_1,\ldots,x_m \right)$, $W\left( y_1,\ldots,y_n \right) \in \mathcal{P}_{\mathbbm{k}}$ are \textsl{orbifold equivalent}, $V \sim W$, if there exists a finite-rank graded matrix factorization of $W-V$ for which the left and right quantum dimension are non-zero. This defines an equivalence relation on $\mathcal{P}_{\mathbbm{k}}$.
\end{defthm}

We shall prove the above statement, as well as Propositions~\ref{prop:sum+Kn"orrer} and~\ref{prop:MF-summand}, in Subsection~\ref{sec:proofs}. The name `orbifold equivalence' has its roots in the study of orbifolds via defects in two-dimensional quantum field theories \cite{ffrs, dkr, carqrunkel2}.

Recall here that for ungraded matrix factorizations the quantum dimensions can be polynomials, but in the graded case the quantum dimensions are just numbers: 

\begin{lem}\label{lem:dimink}
Quantum dimensions of graded matrix factorizations take value in~$\mathbbm{k}$. 
\end{lem}

\begin{proof}
By construction the adjunction maps of \cite{carqmurfet} have degree zero for any graded matrix factorization~$X$. Hence $\mathrm{qdim}_l(X)$ and $\mathrm{qdim}_r(X)$ must be in~$\mathbbm{k}$ since the variables have positive degrees and every closed endomorphism of the unit factorization is homotopy equivalent to a polynomial. Alternatively, one may simply count degrees in the explicit formulas in~\ref{qdims}. 
\end{proof}

As a consequence of this lemma, in the graded case the statements ``$X$ has non-zero quantum dimensions'' and ``$X$ has invertible quantum dimensions'' are equivalent.

Two basic properties of orbifold equivalences are:

\begin{prop}\label{prop:sum+Kn"orrer}
Let $V(x),W(y),U(z) \in \mathcal{P}_{\mathbbm{k}}$.
\begin{enumerate}
\item (\textsl{Compatibility with external sums}) If $V \sim W$, then $V+U \sim W+U$. 
\item (\textsl{Kn\"orrer periodicity}) $W \sim W + u^2 + v^2$.
\end{enumerate}
\end{prop}

The following proposition gives some simple necessary conditions for any two potentials to be orbifold equivalent.

\begin{prop}\label{prop:necessary}
Suppose $V,W \in \mathcal{P}_{\mathbbm{k}}$ are orbifold equivalent. Then
\begin{enumerate}
\item $m - n$ is even, where $V \in \mathbbm{k}[x_1,\dots,x_m]$ and $W \in \mathbbm{k}[y_1,\dots,y_n]$.
\item $c_V = c_W$.
\end{enumerate}
\end{prop}

Part (i) is trivial since the supertrace in Proposition \ref{qdims} is zero for an odd matrix. Part (ii) is proved in \cite[Sect.\,6.2]{carqrunkel2}. 

The converse of Proposition~\ref{prop:necessary} is expected to be false. For example, consider the family of potentials $x_1^{\,3} + x_2^{\,3} + x_3^{\,3} + c \, x_1 x_2 x_3$ with $c \in \mathbbm{C}$ which all have central charge~3 and whose zero locus is an elliptic curve in $\mathbbm{C}\mathbb{P}^2$. In analogy with \cite{fgrs0705.3129, bbr1205.4647} we expect the potentials for different values of~$c$ to be orbifold equivalent iff the complex structure parameters of the corresponding curves are related by some $\operatorname{GL}(2,\mathbb{Q})$ transformation, resulting in infinitely many equivalence classes for these potentials.

\medskip

If $V,W \in \mathcal{P}_{\mathbbm{k}}$ are orbifold equivalent, then the corresponding categories of matrix factorizations are closely related. Namely, let $X = (X,d_X)$ be a finite-rank graded matrix factorization of $W(y) - V(x)$ with non-zero quantum dimensions. Let further $M = (M,d_M)$ be a graded matrix factorization of $V(x)$. Then their tensor product $X \otimes M := (X \otimes_{k[x]} M , d_X \otimes 1 + 1 \otimes d_M)$ is a graded matrix factorization of $W(y)$ (which is necessarily of infinite rank, but still equivalent to a finite-rank factorization, for $M \neq 0$). 

\begin{prop}\label{prop:MF-summand}
Suppose that $V(x)$, $W(y) \in \mathcal{P}_{\mathbbm{k}}$ are orbifold equivalent and let $X \in \mathrm{hmf}^{\mathrm{gr}}_{\mathbbm{k}[x,y],W-V}$ have non-zero quantum dimensions. Then every matrix factorization in $\mathrm{hmf}^{\mathrm{gr}}_{\mathbbm{k}[y],W}$ occurs as a direct summand of $X \otimes M$ for some $M \in \mathrm{hmf}^{\mathrm{gr}}_{\mathbbm{k}[x],V}$.
This remains true if `$\mathrm{hmf}^{\mathrm{gr}}$' is replaced by `$\mathrm{hmf}$' everywhere. 
\end{prop}

\subsection{Orbifold equivalence for simple singularities}\label{subsec:orbeqSimSin}

Now we take $\mathbbm{k}=\mathbbm{C}$ and consider the subset of potentials which define simple singularities.\footnote{%
As we shall see in Section~\ref{subsec:simple-sing}, particularly in Remark~\ref{rem:galE6}, we can also work over the cyclotomic field $\mathbbm{k} = \mathbb{Q} (\zeta)$ for an appropriate root of unity~$\zeta$.} 
These fall into an ADE classification and can be taken to be the following elements of $\mathbbm{C}[x_1,x_2]$ (see e.\,g.~\cite[Prop.\,8.5]{yoshinobuch}):
\begin{align}
W^{(\mathrm{A}_{d-1})} &= x_1^{\,d} + x_2^{\,2} & c&=1-\tfrac{2}{d} &(d \geqslant 2)&
\nonumber\\
W^{(\mathrm{D}_{d+1})} &=  x_1^{\,d} + x_1 x_2^{\,2} & c&=1-\tfrac{2}{2d} &(d \geqslant 3)&
\nonumber\\
W^{(\mathrm{E}_6)} &= x_1^{\,3} + x_2^{\,4} & c&=1-\tfrac{2}{12}
\label{eq:simple-sing-ADE}
\\
W^{(\mathrm{E}_7)} &= x_1^{\,3} + x_1 x_2^{\,3} & c&=1-\tfrac{2}{18}
\nonumber\\
W^{(\mathrm{E}_8)} &= x_1^{\,3} + x_2^{\,5} & c&=1- \tfrac{2}{30}
\nonumber
\end{align}
From Proposition~\ref{prop:necessary} we know that for two potentials to be orbifold equivalent, their central charges have to agree. The preimages of the central charge function on the potentials in~\eqref{eq:simple-sing-ADE} are precisely the sets
\begin{align}
&\big\{ W^{(\mathrm{A}_{d-1})} \big\} \quad  \text{for $d$ odd},
\nonumber \\
&\big\{ W^{(\mathrm{A}_{d-1})} , W^{(\mathrm{D}_{d/2+1})}\big\} \quad \text{for $d$ even and $d \notin \big\{12,18,30\big\}$},
\label{eq:equiv-classes}\\
&\big\{ W^{(\mathrm{A}_{11})},W^{(\mathrm{D}_{7})},W^{(\mathrm{E}_6)} \big\} \, , \quad \big\{ W^{(\mathrm{A}_{17})},W^{(\mathrm{D}_{10})},W^{(\mathrm{E}_7)} \big\} \, , \quad \big\{ W^{(\mathrm{A}_{29})},W^{(\mathrm{D}_{16})},W^{(\mathrm{E}_8)} \big\} \, .
\nonumber 
\end{align}
In fact, this list exhausts the relevant orbifold equivalences: 

\begin{thm}\label{thm:ADEorbifolds}
The orbifold equivalence classes of the potentials~\eqref{eq:simple-sing-ADE} are precisely those listed in~\eqref{eq:equiv-classes}.
\end{thm}

This is our main result (together with Corollary \ref{cor:Etypemod} below and Proposition \ref{adjointequivalenceformfs}), which we prove in Section~\ref{subsec:simple-sing} by explicitly constructing graded matrix factorizations~$X$ with non-zero quantum dimensions for the equivalences
$W^{(\mathrm{A}_{2d-1})} \sim  W^{(\mathrm{D}_{d+1})}$ (already given in \cite{carqrunkel2}), as well as $W^{(\mathrm{A}_{11})} \sim W^{(\mathrm{E}_6)}$, $W^{(\mathrm{A}_{17})} \sim W^{(\mathrm{E}_7)}$ and $W^{(\mathrm{A}_{29})} \sim W^{(\mathrm{E}_8)}$.

\medskip

Proposition~\ref{prop:MF-summand} and Theorem~\ref{thm:ADEorbifolds} can be strengthened to equivalences of categories by invoking the general theory of equivariant completion -see Proposition \ref{adjointequivalenceformfs}.

In Section~\ref{subsec:simple-sing} we will explicitly compute $X^\dagger \otimes X$ for the matrix factorizations~$X$ giving rise to the orbifold equivalences of simple singularities. For those involving E-type singularities we find that $X^\dagger \otimes X$ decomposes into sums of well-known matrix factorizations: 

\begin{cor}\label{cor:Etypemod}
We have 
\begin{align}
\mathrm{hmf}^{\mathrm{gr}}_{\mathbbm{C}[x], V^{(\mathrm{E}_6)}}
& \cong 
\mathrm{mod} \! \big( P_{\{0\}} \oplus P_{\{-3,-2,\ldots,3\}} \big) \, , 
\nonumber
\\
\mathrm{hmf}^{\mathrm{gr}}_{\mathbbm{C}[x], V^{(\mathrm{E}_7)}} 
& \cong 
\mathrm{mod} \! \big( P_{\{0\}} \oplus P_{\{-4,-3,\ldots,4\}} \oplus P_{\{-8,-7,\ldots,8\}} \big) \, , 
\label{eq:hmfEmod}
\\
\mathrm{hmf}^{\mathrm{gr}}_{\mathbbm{C}[x], V^{(\mathrm{E}_8)}} 
& \cong 
\mathrm{mod} \! \big( P_{\{0\}} \oplus P_{\{-5,-4,\ldots,5\}} \oplus P_{\{-9,-8,\ldots,9\}} \oplus P_{\{-14,-13,\ldots,14\}} \big) 
\nonumber
\end{align}
where the rank-one matrix factorizations $P_S$ are those defined in Example \ref{permtypembf}. 
\end{cor}

\begin{rem}\label{rem:intro} 
\begin{enumerate}
\item 
The construction of the equivalence relation can be repeated in a number of slightly modified settings. For example, one can work with ungraded matrix factorizations, or one can work over any commutative ring~$\mathbbm{k}$ as in \cite{carqmurfet}. Or, instead of using a $\mathbb{Q}$-grading, one can consider matrix factorizations with R-charge in the sense of \cite{carqrunkel3} (which is more general). 

The setting used in this paper was chosen to be on the one hand as simple as possible and on the other hand to be strong enough for us to be able to prove the decomposition of simple singularities into equivalence classes \eqref{eq:equiv-classes} -- hence the homogeneous potentials and the $\mathbb{Q}$-grading. 
(In the ungraded setting we do not know how to exclude the existence of equivalences beyond those in~\eqref{eq:equiv-classes}.)
\item
The decompositions~\eqref{eq:equiv-classes} and~\eqref{eq:hmfEmod} are expected from two-dimensional rational conformal field theory.
\item \label{quiver}
It was shown in \cite{kst0511155} that for a simple singularity $V\in \mathbbm{C}[x]$, $\mathrm{hmf}^{\mathrm{gr}}_{\mathbbm{C}[x],V}$ is equivalent to $\mathbb{D}^{\textrm{b}}(\mathrm{Rep}\mathbbm{C} Q)$, where~$Q$ is (any choice of) the associated Dynkin quiver. Thus by Theorem~\ref{thm:ADEorbifolds} and \cite{carqrunkel2} the derived representation theory of ADE quivers enjoys orbifold equivalences analogous to~\eqref{eq:hmfEmod}.\footnote{%
Such a relation between A- and D-type quivers was already proven in \cite{ReitenRiedtmann} by different methods.} 
The monoids $X^\dagger \otimes X \cong P_{\{0\}} \oplus \ldots$ translate into functors on $\mathbb{D}^{\textrm{b}}(\mathrm{Rep}\mathbbm{C} Q)$ whose actions on simple objects are easily computable. 
\end{enumerate}
\end{rem}

\subsection{Proofs}\label{sec:proofs}

In this subsection we provide proofs of the results summarised above.

The discussion for simple singularities can also be viewed as showcasing methods that may prove useful for constructing further orbifold equivalences between potentials.

\subsubsection{Pivotal bicategories}\label{subsec:pivotalbicat}

We will use the concepts and notation from Subsection \ref{sec:catth}. The notion of orbifold equivalence actually makes sense between objects in any pivotal bicategory.

\begin{defn}\label{def:orb-equiv}
Let~$\mathcal{B}$ be a pivotal bicategory.
Two objects $a,b\in\mathcal{B}$ are \textsl{orbifold equivalent}, $a\sim b$, if there is an $X\in\mathcal{B}(a,b)$ such that
\begin{enumerate}
\item $X$ has invertible left and right quantum dimensions, and
\item $\mathcal{D}_l(X)$ and $\mathcal{D}_r(X)$ map invertible quantum dimensions to automorphisms. 
\end{enumerate}
\end{defn}

Condition (ii) is needed for orbifold equivalence in this general setting to be an equivalence relation (see Theorem~\ref{thm:orb-equiv}). 
This condition may be hard to check directly, as it requires knowing all invertible elements in $\mathrm{End}(I_{a})$ and $\mathrm{End}(I_{b})$ which arise as quantum dimension of some 1-morphism. 
A special case where condition (ii) is already implied by condition (i) is if~$\mathcal{B}$ is such that every Hom-category $\mathcal{B}(a,b)$ is $R$-linear for some fixed commutative ring~$R$, and such that the left and right quantum dimensions of every 1-morphism $X \in \mathcal{B}(a,b)$ are in $R\cdot 1_{I_a}$ and $R\cdot 1_{I_b}$, resp. In this case one may think of the left/right quantum dimension as an element of $R$, and one can check that $\mathrm{qdim}_{\mathrm{l/r}}(X \otimes Y) = \mathrm{qdim}_{\mathrm{l/r}}(X) \cdot \dim_{\mathrm{l/r}}(Y)$, where the product on the right is in $R$. Graded finite-rank matrix factorizations are an example of this $R$-linear setting (see Lemma~\ref{lem:dimink} below).

\begin{thm}\label{thm:orb-equiv}
Orbifold equivalence is an equivalence relation. 
\end{thm}

\begin{proof}
To see reflexivity $a\sim a$ one can take $X = I_a$, for which $\mathcal{D}_l(I_a) = 1_{I_a} = \mathcal{D}_r(I_a)$. 

For symmetry and transitivity it is helpful to expand condition (ii) of the equivalence relation in more detail: 
\begin{enumerate}
\item[(ii-r)] 
for every $e\in\mathcal{B}$ and every $Z \in \mathcal{B}(e,a)$ such that $\mathrm{qdim}_r(Z)$ is invertible in $\mathrm{End}(I_a)$, we have that $\mathcal{D}_r(X) (\mathrm{qdim}_r(Z))$ is invertible in $\mathrm{End}(I_b)$; 
\item[(ii-l)] 
for every $f\in\mathcal{B}$ and every $Z' \in \mathcal{B}(b,f)$ such that $\mathrm{qdim}_l(Z')$ is invertible in $\mathrm{End}(I_b)$, we have that $\mathcal{D}_l(X) (\mathrm{qdim}_l(Z'))$ is invertible in $\mathrm{End}(I_a)$.
\end{enumerate}
Of course, since (ii-r) holds for all $Z$, it also holds for all $Z^\dagger$, and hence $\mathcal{D}_r(X)$ also maps invertible left quantum dimensions to automorphisms. The same applies to (ii-l).

Symmetry $a\sim b \Leftrightarrow b\sim a$ follows by replacing $X$ by $X^\dagger$. That $X^\dagger$ satisfies (i) follows from that $\mathrm{qdim}_l (X^\dagger)=\mathrm{qdim}_r(X)$. To see (ii-r) for $X^\dagger$, use $\mathcal{D}_r(X^\dagger)=\mathcal{D}_l(X)$ and that (ii-l) holds for $X$. Condition (ii-l) follows analogously.

To check transitivity, consider $X\in\mathcal{B}(a,b)$ and $Y\in\mathcal{B}(b,c)$ satisfying conditions (i) and (ii). Then $\mathrm{qdim}_r(Y \otimes X) = \mathcal{D}_r(Y) (\mathrm{qdim}_r(X))$ which is invertible by~(i) for~$X$ and (ii-r) for $Y$. Dito for $\mathrm{qdim}_l$. Next let $Z \in \mathcal{B}(e,a)$ be as in (ii-r) above. Then $\mathcal{D}_r(Y \otimes X) (\mathrm{qdim}_r(Z)) = \mathcal{D}_r(Y) (\mathrm{qdim}_r(X \otimes Z))$. Now $\mathrm{qdim}_r(X \otimes Z)$ is invertible by (ii-r) for $X$ and therefore $\mathcal{D}_r(Y) (\mathrm{qdim}_r(X \otimes Z))$ is invertible by (ii-r) for $Y$. That $Y \otimes X$ satisfies (ii-l) is seen similarly.
\end{proof}

We conclude our brief general discussion with an implication of the existence of orbifold equivalences in the setting of retracts and idempotent splittings. 
An object~$S$ in some category is called a \textsl{retract} of an object~$U$ if there are morphisms $e : S \to U$ and $r : U \to S$ such that $r \circ e = 1_S$. In particular, $e$ is mono and so~$S$ is a subobject of~$U$.

\begin{prop}\label{prop:all-is-retract}
Let $a,b,c,d \in \mathcal{B}$ and let $X \in \mathcal{B}(a,b)$, $Y \in \mathcal{B}(c,d)$ have invertible left quantum dimensions. 
Every~$Z \in \mathcal{B}(a,c)$ is a retract of $Y^\dagger \otimes F \otimes X$ for some suitable ($Z$-dependent) $F \in \mathcal{B}(b,d)$.
\end{prop}

\begin{proof}
We set $F = Y \otimes Z \otimes X^\dagger$ and define the maps $e : Z \to Y^\dagger \otimes F \otimes X$ and $r : Y^\dagger \otimes F \otimes X \to Z$  by
\begin{align}
e &= \widetilde{\mathrm{coev}}_Y \otimes 1_Z \otimes\widetilde{\mathrm{coev}}_X \, ,
\nonumber \\
r &= \big(\mathrm{qdim}_l(Y)^{-1} \otimes 1_Z \otimes \mathrm{qdim}_l(X)^{-1}\big) \circ \big(\mathrm{ev}_Y \otimes 1_Z \otimes \mathrm{ev}_X \big) \, .
\end{align}
Clearly, $r \circ e = 1_Z$.
\end{proof}

\begin{rem}\label{rem:all-direct-summand}
If the Hom-categories of~$\mathcal{B}$ are additive and if idempotent 2-morphisms split, we can improve on Proposition~\ref{prop:all-is-retract} slightly: instead of~$Z$ just being a retract, it now even occurs as a direct summand of $Y^\dagger \otimes F \otimes X$ for a suitable~$F$. To see this, take the maps $e,r$ from the proof, note that $p = e \circ r$ is an idempotent endomorphism of $Y^\dagger \otimes Y \otimes Z \otimes X^\dagger \otimes X$, and consider the decomposition $1 = p + (1-p)$ into orthogonal idempotents.
\end{rem}

\medskip

\medskip

We may now proceed to prove the results advertised in Section~\ref{sec:introductionpart1}. 

\begin{proof}[Proof of Theorem~\ref{defthm:MF-orbequiv}]
This is a corollary of \cite[Prop.\,8.5]{carqmurfet}: reflexivity and symmetry follow analogously to Theorem~\ref{thm:orb-equiv}, and transitivity follows from the fact that the quantum dimensions~\ref{qdims} are manifestly multiplicative up to a sign. 
\end{proof}

\begin{proof}[Proof of Proposition~\ref{prop:sum+Kn"orrer}]
From~\ref{qdims} it is clear that quantum dimensions are multiplicative (up to a sign) also for the external tensor product $\otimes_{\mathbbm{k}}$, showing part~(i). 
Part~(ii) is proven in \cite[Sect.\,7.2]{carqrunkel2}. 
\end{proof}

\begin{proof}[Proof of Proposition~\ref{prop:MF-summand}]
This is a direct consequence of the proof of Proposition~\ref{prop:all-is-retract} and Remark~\ref{rem:all-direct-summand} as we are dealing with idempotent complete matrix factorization categories. 
\end{proof}

\subsubsection{Simple singularities}\label{subsec:simple-sing}

Let $W(y_1,y_2)$ be one of the potentials $W^{(\mathrm{D}_{d/2+1})}, W^{(\mathrm{E}_6)}, W^{(\mathrm{E}_7)}, W^{(\mathrm{E}_8)}$ in~\eqref{eq:simple-sing-ADE}, and let $V(x_1,x_2) = W^{(\mathrm{A}_{d-1})} = x_1^d + x_2^2$ be the corresponding A-type potential of the same central charge as~$W$. To avoid too many indices, we  rename
\beq
	x_1 \leadsto u \, , \quad
	x_2 \leadsto v \, , \quad
	y_1 \leadsto x \, , \quad
	y_2 \leadsto y \, .
\eeq

In this section we will give a finite-rank graded matrix factorization of $W(x,y)-V(u,v)$ of non-zero left and right quantum dimension in each case, thus proving the results collected in Section~\ref{subsec:orbeqSimSin}. These matrix factorizations have all been constructed along the following lines:
\begin{enumerate}
\item Pick a matrix factorization $X_0$ of $W(x,y) - v^2$ that is of low rank.
\item Thinking of~$u$ as a deformation parameter, add to each entry of the matrix $d_{X_0}$ the most general homogeneous polynomial of the form $u \cdot p(u,v,x,y)$ with the same total degree as the given entry. Let $d_{X_u}$ be the resulting matrix. 
\item Reduce the number of free parameters in the polynomials~$p$ by absorbing some of them via a similarity transformation $d_{X_u} \mapsto \Phi \circ d_{X_u} \circ \Phi^{-1}$.
\item Try to find a set of parameters (the remaining coefficients in the polynomials~$p$) such that $d_{X_u} \circ d_{X_u} = (W-V) \cdot \mathrm{id}_{X_u}$.
\end{enumerate}
Choosing a low-rank starting point in~(i) and reducing the number of parameters via (iii) only serves to simplify the problem in~(iv). We will work through these steps for $W=W^{(\mathrm{D}_{d/2+1})}$ and $W=W^{(\mathrm{E}_6)}$ in some detail, while our discussion will be briefer for $W^{(\mathrm{E}_7)}$ and $W^{(\mathrm{E}_8)}$.

\subsubsection*{$\boldsymbol{W^{(\mathrm{D}_{d/2+1})} \sim W^{(\mathrm{A}_{d-1})}}$}

We set $b=d/2$, so that $W = x^b + x y^2$ and $V = u^{2b}+v^2$, and we write 
$R = \mathbbm{C}[u,v,x,y]$. 
As the starting point in step (i) we choose $X_0 = (X , d_{X_0})$ with $\mathbb{Z}_2$-graded $R$-module $X = R^2 \oplus R^2$ and twisted differential
\beq\label{eq:D-calc-X0}
	d_{X_0} = \begin{pmatrix} 0 & d_{X_0}^1 \\ d_{X_0}^0 & 0 \end{pmatrix}
	\quad \text{with} \quad
	d_{X_0}^1 = \begin{pmatrix} x & v \\ v & x^{b-1} + y^2 \end{pmatrix}
	.
\eeq
Since $\det(d_{X_0}^1) = x^b + x y^2 - v^2$, the component $d_{X_0}^0$ is determined to be the adjunct matrix, $d_{X_0}^0  = (d_{X_0}^1)^\#$. Recall that the adjunct $M^\#$ of an invertible matrix~$M$ has entries which are polynomial in those of~$M$ and satisfies $M^{-1} = (\det M)^{-1} \cdot M^\#$.

The deformed matrix factorization $X_u$ has the same underlying $R$-module $R^2 \oplus R^2$. 
For step (ii) we need to pick the most general homogeneous deformation of $d_{X_0}^1$, which is
\beq\label{eq:D-calc-ansatz}
	d_{X_u}^1 = \begin{pmatrix} 
	x + u \, p_{11} & v  + u \, p_{12}  \\ v  + u \, p_{21}  & x^{b-1} + y^2  + u \, p_{22} 
	\end{pmatrix}
	\quad \text{where} \quad p_{ij} \in \mathbbm{C}[u,v,x,y] \, .
\eeq
The degrees of the variables are $|u|=1/b$, $|v|=1$, $|x|=2/b$, and $|y| = 1-1/b$. Hence the total degrees of the $p_{ij}$ have to be
\beq
	|p_{11}| = \tfrac1b	\, , \quad
	|p_{12}| = 1-\tfrac1b	\, , \quad
	|p_{21}| = 1-\tfrac1b	\, , \quad
	|p_{22}| = 2-\tfrac3b	 \, .
\eeq
From this it follows that $p_{11} = a_1 u$ for some $a_1 \in \mathbbm{C}$, but the remaining $p_{ij}$ will contain of the order~$b$ many free coefficients. 

Moving to step (iii), we will now use degree-preserving row and column operations on $d_{X_u}^1$ to reduce the number of free coefficients. By applying the inverse operations to $d_{X_u}^0$ one produces in this way an isomorphic matrix factorization. The most general such row and column manipulations turn out to be
\beq\label{eq:D-calc-ansatz-reduced}
	\begin{pmatrix} 1 & 0 \\ f & 1 \end{pmatrix} d_{X_u}^1 \begin{pmatrix} 1 & g \\ 0 & 1 \end{pmatrix}
	= \begin{pmatrix}
	x + a_1 u^2 & v  + u \, p_{12}  + g (x + a_1 u^2)  \\ v  + u \, p_{21}  + f (x + a_1 u^2) & *
	\end{pmatrix}
\eeq
where $|f|=|g|=1-2/b$. We see that $f,g$ can be used to remove any $x$-dependence from $p_{12}$ and $p_{21}$, and we arrive at the following reduced ansatz: $d_{X_u}^1$ of the form \eqref{eq:D-calc-ansatz} with
\begin{align}
	p_{11} &= a_1 u \, , &
	p_{12} &= a_2 y + a_3 u^{b-1} \, , \nonumber \\
	p_{21} &= a_4 y + a_5 u^{b-1} \, ,  &
	p_{22} &= q_1  v + q_2  y + q_3  \, ,
\end{align}
where $a_i \in \mathbbm{C}$ and $q_i \in \mathbbm{C}[u,x]$ with degrees $|q_1| = 1-3/b$, $|q_2| = 1-2/b$, and $|q_3|=2-3/b$. 

Step (iv) amounts to the tedious task of trying to find conditions such that $d_{X_u} \circ d_{X_u} = (W-V) \cdot 1$.
The second component of the twisted differential is uniquely determined to be $d_{X_u}^0  = q / \det(d_{X_u}^1) \cdot (d_{X_u}^1)^\#$ with $q=x^b + x y^2 - u^{2b} - v^2$, and we need to find values of the deformation parameters so that this matrix has polynomial entries. Since $q \in \mathbbm{C}[x,y,u,v]$ is irreducible, either $q$ is a factor of $\det(d_{X_u}^1)$, or $\det(d_{X_u}^1)$ has to cancel against the entries of $(d_{X_u}^1)^\#$. Degree considerations show that the latter is not possible, and in fact $\det(d_{X_u}^1)$ equals~$q$ up to a multiplicative constant. By rescaling $d_{X_u}^1$ if necessary, without restriction of generality we can impose $\det(d_{X_u}^1) = q$.
In solving this condition, one is lead to distinguish between two cases, $a_2=0$ and $a_2 \neq 0$. Setting $a_2=0$ produces a solution with zero left and right quantum dimension. On the other hand, keeping $s := a_2 \neq 0$ forces
\beq \label{eq:D-calc-solution}
	d_{X_u}^1 = \begin{pmatrix} 
	x - (su)^2 & v  + y(s u)  \\ v  - y(s u)  & \frac{x^b - (su)^{2b}}{x - (su)^2} + y^2 
	\end{pmatrix}
	\quad \text{where} \quad s^{2b} = 1 \, .
\eeq
For $s=1$ this is the solution given in \cite[Sect.\,7.3]{carqrunkel2}. The quantum dimensions are  
\beq
	\mathrm{qdim}_l(X_u) = -2 s
	\, , \quad
	\mathrm{qdim}_r(X_u) = - s^{-1} \, .
\eeq

What remains to be done is to check that there exists a $\mathbb{Q}$-grading on $X_u = (X, d_{X_u})$ such that $d_{X_u}$ has $\mathbb{Q}$-degree 1. It is in fact easy to write down all such gradings. The $\mathbb{Q}$-grading on the ring~$R$ is fixed by $|1|=0$, with the variables $u,v,x,y$ having degrees as stated below \eqref{eq:D-calc-ansatz}. Writing $R[\alpha]$, $\alpha \in \mathbb{Q}$, for~$R$ with $|1|=\alpha$, the possible gradings on~$X$ are
\begin{equation}
X^0 = R[\alpha] \oplus R[\alpha-1+\tfrac2b] \, , \quad
	X^1 = R[\alpha-1+\tfrac2b] \oplus R[\alpha]
	\, , \quad \alpha \in \mathbb{Q} \, .
\end{equation}

This completes the proof that the potentials $W^{(\mathrm{A}_{d-1})}$ and $W^{(\mathrm{D}_{d/2+1})}$ are orbifold equivalent. We can summarise the above discussion as follows.

\begin{lem}\label{lem:D-calc-summary}
Let~$Y$ be a rank-two $\mathbb{Q}$-graded matrix factorization of $x^b + x y^2 - u^{2b} - v^2$. Suppose that~(i) $Y$ has non-zero quantum dimensions and~(ii) when setting~$u$ to zero, $Y$ is equal to $X_0$ in~\eqref{eq:D-calc-X0}. Then~$Y$ is isomorphic in $\mathrm{hmf}^{\mathrm{gr}}_{\mathbbm{C}[u,v,x,y],x^b + x y^2 - u^{2b} - v^2}$ to $X_u$ in~\eqref{eq:D-calc-solution} for some choice of~$s$.
\end{lem}

Let us analyze the matrix factorization $X_u$ a bit further. 
First we recall the definition of the permutation-matrix factorizations~$P_S$ from \ref{permtypembf}. The left and right quantum dimensions of a permutation-type matrix factorization $P_J$ were computed in \cite[Sect.\,3.3]{carqrunkel3} to be
\begin{equation}
  \mathrm{qdim}_l(P_J) = \sum_{l \in J} \zeta_d^l 
  \, , \quad 
  \mathrm{qdim}_r(P_J) = \sum_{l \in J} \zeta_d^{-l} \, . 
\end{equation}

The automorphism~$\sigma$ of $\mathbbm{C}[u]$ determined by $\sigma(u) = \zeta_d u$ leaves the potential $u^d$ invariant. 
Given a matrix factorization~$X$ of $W - u^d$ for some $W \in \mathbbm{C}[x]$, twisting the $\mathbbm{C}[u,x]$-action on~$X$ by $\sigma^l$ (with $\sigma$ extended to $\mathbbm{C}[u,x]$ via $\sigma(x_i)=x_i$) results again in a matrix factorization of $W - u^d$ which we denote by $X_{\sigma^l}$. It follows from \cite[Lem.\,2.10]{carqrunkel1} that
\begin{equation}
\label{eq:P_{-k}-action}
	X_{\sigma^l} \cong X \otimes P_{\{-l\}} 
\end{equation}
in $\mathrm{hmf}^{\mathrm{gr}}_{\mathbbm{C}[u,x], W-u^d}$. 

\begin{rem}\label{rem:D-calc-s-choice}
The observation \eqref{eq:P_{-k}-action} together with the classification given in Lemma \ref{lem:D-calc-summary} explains the $2b$-th root of unity~$s$ appearing as a parameter in \eqref{eq:D-calc-solution}. To wit, given any 1-morphism $P : V \to V$ of non-zero quantum dimensions, the composition $X_u \otimes P : V \to W$ also has non-zero quantum dimensions. Consider the choice $P = P_{\{l\}} \otimes_{\mathbbm{C}} I_{v^2}$, where $I_{v^2}$ the unit 1-morphism for the potential~$v^2$ with twisted differential $(\begin{smallmatrix} 0 & v'-v \\ v'+v & 0 \end{smallmatrix})$. 
Then~$P$ is a 1-endomorphism of $u^{2b}+v^2$ and $X_u \otimes P$ is again a rank-two factorization satisfying the conditions in Lemma~\ref{lem:D-calc-summary}. Hence $X_u \otimes P$ must be isomorphic to~$X_u$ for a possibly different choice of~$s$. But different choices of~$s$ precisely amount to twisting the $u$-action by some power of the automorphism~$\sigma$.
\end{rem}

It was checked in \cite[Sect.\,7.3]{carqrunkel2} that for $d\in \{ 2,3,\ldots,10 \}$, and up to a trivial factor of the unit $I_{v^2}$ the monoid $X^\dagger_{u'} \otimes X_u$ is isomorphic to
\beq\label{eq:PPinDcase}
P_{\{0\}} \oplus P_{\{0,1,\ldots,d-1\} \backslash \{\frac{d}{2}\}} \, . 
\eeq
It is straightforward to compute an explicit basis for the endomorphisms of~$X_u$ in the homotopy category, e.\,g.~by using the Singular code of \cite{cdr}. We have done so for small values of~$d$, with the result 
\beq\label{eq:EndXDcase}
\mathrm{End}
(X_u) 
\cong 
\Big(
\bigoplus_{i=0}^{d-2} \mathbbm{C}_{i\cdot \frac{2}{d}}
\Big)
\oplus 
\mathbbm{C}_{\frac{d-2}{d}}
\eeq
where~$\mathbbm{C}_{j}$ denotes the one-dimensional subspace of maps of $\mathbb{Q}$-degree~$j$. 
Both~\eqref{eq:PPinDcase} and~\eqref{eq:EndXDcase} are expected to be the correct expressions for all~$d$.

\subsubsection*{$\boldsymbol{W^{(\mathrm{E}_6)} \sim W^{(\mathrm{A}_{11})}}$}

Our starting point in step~(i) now is the matrix factorization $X_0 = (R^2 \oplus R^2 , d_{X_0})$ of $x^3 + y^4 - v^2$ 
with $R = \mathbbm{C}[u,v,x,y]$ and twisted differential
\beq\label{eq:E6-calc-X0}
	d_{X_0}^1 = \begin{pmatrix} y^2-v & -x \\ x^2 & y^2+v \end{pmatrix}
	 , \quad
	d_{X_0}^0  = (d_{X_0}^1)^\#
	\, .
\eeq
This is one of the six indecomposable objects of $\mathrm{hmf}^{\mathrm{gr}}_{\mathbbm{C}[v,x,y],W^{(\mathrm{E}_6)} - v^2}$ listed in \cite[Sect.\,5]{kst0511155}. The variable degrees are
\beq
	|u| = \tfrac16
	\, , \quad
	|v| = \tfrac66
	\, , \quad
	|y| = \tfrac36
	\, , \quad
	|x| = \tfrac46 
	\, .
\eeq

Carrying out steps~(ii) and (iii) leads to the possibility
\beq
	d_{X_u}^1 = \begin{pmatrix} 
	y^2-v + a_1 x u^2 + a_2 u^6 & 
	-x + a_3 y u + a_4 u^4 \\ 
	x^2 + a_5 y x u +  a_6 x u^4 + a_7 v u^2 + q & 
	y^2+v + a_8 x u^2 + a_9 u^6
	\end{pmatrix} 
\eeq
for the reduced ansatz, where $a_i \in \mathbbm{C}$ and $q \in \mathbbm{C}[u,y]$ with $|q| = 8/6$. Of course a different choice of similarity transformation may give a different (but isomorphic) reduced ansatz. Here, the row and column manipulations were used to absorb the terms $y u^3$ in the diagonal entries.

By the same argument as used in the D-case, we now need to solve 
the condition $\det(d_{X_u}^1)= x^3 + y^4 - u^{12} - v^2$ under the extra constraint that $X_u$ has non-zero quantum dimensions. This leads to 
\begin{align} \label{eq:E6-calc-solution}
d_{X_u}^1 = \big(\begin{smallmatrix} a&b \\ c & d \end{smallmatrix}\big) \quad \text{with} \quad 
	a &= y^2-v + \tfrac12 x (s u)^2 + \tfrac{2t+1}8  (s u)^6   \, ,
	\nonumber \\
	b &= -x + y (s u) + \tfrac{t+1}4  (s u)^4   \, ,
	\nonumber \\
	c &= x^2 + y x (s u) +   \tfrac{t}4 x (su)^4 + \tfrac{2t+1}4 y (su)^5
	 - \tfrac{9t+5}{48} (su)^8   \, ,
	\nonumber \\
	d &= y^2+v + \tfrac12 x (s u)^2 + \tfrac{2t+1}8 (s u)^6 \, ,
\end{align}
and $d_{X_0}^0  = (d_{X_0}^1)^\#$. Here, $s$ and~$t$ can be any solution of
\beq\label{eq:stE6}
	t^2  = \tfrac13
	\, , \quad	
	s^{12} = - 576 \, (26 \, t -15) \, .
\eeq
Note that~$s$ can be modified by a 12-th root of unity -- the interpretation of this is as in Remark~\ref{rem:D-calc-s-choice}.

The quantum dimensions of $X_u$ are
\begin{equation}
	\mathrm{qdim}_l(X_u) = s
	\, , \quad
	\mathrm{qdim}_r(X_u) = 3 \, (1 - t) s^{-1}
	 \, ,
\end{equation}

and all possible $\mathbb{Q}$-gradings on $X_u$ are again easily found: the underlying $\mathbb{Z}_2$-graded $R$-module $X = X_0 \oplus X_1$ has components with $\mathbb{Q}$-grading
\begin{equation}
	X^0 = R[\alpha] \oplus R[\alpha-\tfrac13] \, , \quad
	X^1 = R[\alpha] \oplus R[\alpha-\tfrac13]
	\, , \quad \alpha \in \mathbb{Q} .
\end{equation}

As in the D-case, we can summarise the above as:

\begin{lem}\label{lem:E6-calc-summary}
Let~$Y$ be a rank-two $\mathbb{Q}$-graded matrix factorization of $x^3 + y^4 - u^{12} - v^2$. Suppose that (i)~$Y$ has non-zero quantum dimensions and (ii) when setting~$u$ to zero, $Y$ is equal to $X_0$ in~\eqref{eq:E6-calc-X0}. Then~$Y$ is isomorphic in $\mathrm{hmf}^{\mathrm{gr}}_{\mathbbm{C}[u,v,x,y],x^3 + y^4 - u^{12} - v^2}$ to $X_u$ in~\eqref{eq:E6-calc-solution} for some solution $s,t$ of~\eqref{eq:stE6}.
\end{lem}

\medskip

Computing the endomorphism of~$X_u$ in the homotopy category one finds the 16-dimensional space 
\begin{equation}\label{eq:EndhmfX-E6}
\mathrm{End}
(X_u) 
\cong 
\Big(
\bigoplus_{i=0}^{10} \mathbbm{C}_{i\cdot \frac{2}{12}}
\Big)
\oplus 
\Big(
\bigoplus_{i=3}^{7} \mathbbm{C}_{i\cdot \frac{2}{12}}
\Big) \, .
\end{equation}

As mentioned in Section~\ref{subsec:orbeqSimSin} and explained in detail in \cite{carqrunkel2}, a matrix factorization~$X$ of $W(y)-V(x)$ with invertible quantum dimensions allows us to describe \textsl{all} matrix factorizations of~$W$ in terms of modules over $A := X^\dagger \otimes X \in \mathrm{hmf}^{\mathrm{gr}}_{\mathbbm{k}[x,x'], V(x)-V(x')}$. The matrix $d_{X^\dagger \otimes X}$ also depends on the $y$-variables, and hence the matrix factorization~$A$ is of infinite rank. However, by the results of \cite{dm} it is homotopy equivalent (and thus isomorphic in $\mathrm{hmf}^{\mathrm{gr}}_{k[x,x'], V(x)-V(x')}$) to a finite-rank matrix factorization. 

The construction of this finite-rank factorization and the explicit homotopy equivalence can be implemented on a computer; this was done in \cite{khovhom}, where it was used to compute Khovanov-Rozansky link invariants. In our present situation we can use this implementation to find that $X_{u'}^\dagger \otimes X_u$ is equivalent to the matrix factorization $A'\otimes_\mathbbm{k} I_{v^2}$, where the twisted differential of~$A'$ is represented by a 4-by-4 matrix~$a'$ with nonzero entries
\begin{align}
a'_{13}& = (\tfrac{76}{3}-44 t) u^8+(\tfrac{136}{3}-80 t) u' u^7+(44-76 t) u'^2 u^6+(\tfrac{152}{3}-88 t) u'^3 u^5 \nonumber \\ 
& \quad +(\tfrac{208}{3}-120 t) u'^4 u^4+(\tfrac{152}{3}-88 t) u'^5 u^3+(20 t-12) u'^6 u^2 \nonumber \\ 
& \quad +(88 t-\tfrac{152}{3}) u'^7 u+(52 t-\tfrac{92}{3}) u'^8 \nonumber \, , \\
a'_{14} & = (224-384 t) u^5 +(832-1440 t) u' u^4+(1440-2496 t) u'^2 u^3 \nonumber \\ 
& \quad +(1440-2496 t) u'^3 u^2+(832-1440 t) u'^4 u+(224-384 t) u'^5 \nonumber \, , \\
a'_{23} & =  (\tfrac{5}{3}-\tfrac{19 t}{6}) u^{11}+(\tfrac{49}{6}-\tfrac{43 t}{3}) u' u^{10}+(\tfrac{287}{18}-28 t) u'^2 u^9+(\tfrac{361}{18}-35 t) u'^3 u^8 \nonumber \\ 
& \quad +(\tfrac{208}{9}-\tfrac{241 t}{6}) u'^4 u^7+(\tfrac{475}{18}-46 t) u'^5 u^6+(\tfrac{64}{3}-\tfrac{223 t}{6}) u'^6 u^5+(\tfrac{23}{6}-7 t) u'^7 u^4 \nonumber \\ 
& \quad +(\tfrac{65 t}{3}-\tfrac{227}{18}) u'^8 u^3+(\tfrac{74 t}{3}-\tfrac{259}{18}) u'^9 u^2+(\tfrac{65 t}{6}-\tfrac{58}{9}) u'^{10} u+(\tfrac{5 t}{3}-\tfrac{19}{18}) u'^{11} \nonumber \, , \\
a'_{24}& =  (16-28 t) u^8+(104-180 t) u' u^7+(\tfrac{908}{3}-524 t) u'^2 u^6+(524-908 t) u'^3 u^5 \nonumber \\ 
& \quad +(\tfrac{1780}{3}-1028 t) u'^4 u^4+(448-776 t) u'^5 u^3+(\tfrac{652}{3}-376 t) u'^6 u^2 \nonumber \\ 
& \quad +(60-104 t) u'^7 u+(\tfrac{20}{3}-12 t) u'^8 \nonumber \, , \\
a'_{31}& = \tfrac{3 t u^4}{4}-\tfrac{3 u' u^3}{4}+\tfrac{u'^2 u^2}{4}+\tfrac{3 u'^3 u}{4}+(-\tfrac{3 t}{4}-\tfrac{1}{4}) u'^4 \nonumber \, , \\
a'_{32}& =  6 u-6 u'\nonumber \, , \\
a'_{41}& =  (-\tfrac{13 t}{32}-\tfrac{3}{16}) u^7+(\tfrac{7 t}{32} +\tfrac{1}{8}) u' u^6+(-\tfrac{13 t}{32}-\tfrac{31}{96}) u'^2 u^5+(\tfrac{3 t}{32}+\tfrac{19}{96}) u'^3 u^4 \nonumber \\ 
& \quad +(\tfrac{t}{8}+\tfrac{1}{48}) u'^4 u^3+(\tfrac{t}{16}-\tfrac{1}{12}) u'^5 u^2+(\tfrac{t}{8}+\tfrac{11}{96}) u'^6 u+(\tfrac{3 t}{16}+\tfrac{13}{96}) u'^7 \nonumber \, , \\
a'_{42}& = (\tfrac{3 t}{4}-\tfrac{1}{4}) u^4+2 u' u^3-3 u'^2 u^2+\tfrac{5 u'^3 u}{2}+(-\tfrac{3 t}{4}-\tfrac{5}{4}) u'^4 
\, .
\end{align}
This is a graded matrix factorization with grading:
\begin{equation}
	X^0 = R[\alpha] \oplus R[\alpha-\tfrac12] \, , \quad
	X^1 = R[\alpha+\tfrac13] \oplus R[\alpha-\tfrac16]
	\, , \quad \alpha \in \mathbb{Q} .
\end{equation}

In anticipation of the comparison to conformal field theory in Subsection \ref{sec:CFTcomparison} we single out one of the roots of \eqref{eq:stE6},
\beq
	t_{\text{cft}} = -1/\sqrt{3} \, .
\eeq
Choosing $t=t_\text{cft}$ and performing a series of row and column manipulations produces the isomorphism 
\beq\label{eq:AMF-E6}
\begin{pmatrix}
\phi & 6 & 0 & 0 \\
-\tfrac{1}{32} (12t+7) & 0 & 0 & 0 \\
0 & 0 & 1 & 0 \\
0 & 0 & \psi & 1
\end{pmatrix}
: 
A' \longrightarrow 
P_{\{0\}} \oplus P_{\{-3,-2,\ldots,3\}} 
\eeq
where $\phi=\frac{1}{4} ( u+u' ) ( 3 t u^2 -3uu'+u'^2+3t u'^2 )$ and $\psi=\frac{1}{24} ( u^3-3t u^3-7 u^2 u' -3 t ^2 u'+5 u u'^2-3 t u u'^2-5 u'^3-3 t u'^3 )$. 
This proves the first third of Corollary~\ref{cor:Etypemod}. 

\begin{rem}\label{rem:galE6}
One may wonder what 
the monoid 
$X_{u'}^\dagger \otimes X_u$ 
reduces to for the other solution $t = 1/\sqrt{3}$ of~\eqref{eq:stE6}, i.\,e.~what is the equivalent to the right-hand side of~\eqref{eq:AMF-E6}. As we will see momentarily, the other solution can be related to $t_\text{cft}$ via the action of an appropriate Galois group. 

Define $\zeta_d = \mathrm{e}^{2 \pi \text{i}/d}$ and consider the cyclotomic field $\mathbbm{k} = \mathbb{Q}(\zeta_d)$, i.\,e.~the field obtained from~$\mathbb{Q}$ by adjoining the $d$-th primitive root of unity~$\zeta_d$. The Galois group is isomorphic to the group of units in~$\mathbb{Z}_d$, $\mathrm{Gal}(\mathbbm{k}/\mathbb{Q}) \cong \mathbb{Z}_d^\times$. Given $\nu \in \mathbb{Z}_d^\times$ the action of the corresponding Galois group element $\sigma_\nu$ is $\sigma_\nu(\zeta_d^a) = \zeta_d^{\nu a}$.

Let now $V \in \mathbb{Q}[x]$ be a potential with rational coefficients and let~$M$ be a finite-rank matrix factorization of~$V$ over~$\mathbbm{k}[x]$. We may take $M = (\mathbbm{k}[x]^{2r}, d_M)$, where~$d_M$ is a matrix with entries in~$\mathbbm{k}[x]$. Let $\sigma \in \mathrm{Gal}(\mathbbm{k}/\mathbb{Q})$ be an element of the Galois group and denote by $\sigma(d_M)$ the matrix obtained by applying~$\sigma$ to each entry. Since $\sigma(V) = V$, $\sigma(d_M)$ is still a factorization of~$V$, and we set $\sigma(M) = (\mathbbm{k}[x]^{2r}, \sigma(d_M))$. Analogously, if $f : M \to N$ is a morphism with entries in $\mathbbm{k}[x]$, then $\sigma(f)$ is a morphism from $\sigma(M)$ to $\sigma(N)$.

Let us apply this to the isomorphism in \eqref{eq:AMF-E6}. Write $A'(t)$ for~$A'$ to highlight the $t$-dependence. We choose $\mathbbm{k} = \mathbb{Q}(\zeta_{12})$ so that all $P_S$ are matrix factorizations over $\mathbbm{k}[u,u']$. Since
\beq
	t_\text{cft} = -\tfrac13 \big( \zeta_{12} + \zeta_{12}^{-1} \big) \, , 
\eeq
also $A'(t_\text{cft})$ has entries with coefficients in~$\mathbbm{k}$. The same holds for the isomorphism~\eqref{eq:AMF-E6}, and so we get an isomorphism from $\sigma(A'(t_\text{cft}))$ to $\sigma(P_{\{0\}} \oplus P_{\{-3,-2,\ldots,3\}})$. 
Since the entries of $A'(t)$ are polynomials in $u,u',t$ with rational coefficients, we have $\sigma(A'(t_\text{cft})) = A'(\sigma(t_\text{cft})))$, and $\sigma(P_S) = P_{\sigma_*(S)}$, where $\sigma_*$ is the permutation of $\mathbb{Z}_{12}$ induced by the action of~$\sigma$ on the $12$-th roots of unity.

It turns out that the orbit of $t_\text{cft}$ under $\mathrm{Gal}(\mathbbm{k}/\mathbb{Q})$ covers all, namely both, roots of~\eqref{eq:stE6}: we have $\sigma_5(t_\text{cft}) = -\tfrac13 ( \zeta^5_{12} + \zeta_{12}^{-5} ) = 1/\sqrt{3}$ and $\sigma_{5*}(\{-3,-2,\dots,3\}) = \{ -5,-3,-2,0,2,3,5 \}$ and so we obtain the isomorphism
\beq
	X^\dagger_{u'} \otimes X_u 
	\cong
	\big( P_{\{0\}} \oplus P_{\{-5,-3,-2,0,2,3,5\}} \big) 
	\otimes_k I_{v^2}
	\quad \text{for} \quad 
	t  = \sigma_5(t_\text{cft}) \, .
\eeq
Note that by construction $A'(t_\text{cft})$ and $A'(\sigma_5(t_\text{cft}))$ are Morita equivalent, i.\,e.~the category $\operatorname{mod}(X^\dagger_{u'} \otimes X_u)$ does not depend on the choice of solution~$t$. The situation is analogous for the $\text{E}_7$- and $\text{E}_8$-singularities, cf.~Remarks~\ref{rem:galE7} and~\ref{rem:galE8}. 
\end{rem}

\subsubsection*{$\boldsymbol{W^{(\mathrm{E}_7)} \sim W^{(\mathrm{A}_{17})}}$}

In step~(i) we pick the matrix factorization $X_0 = (R^2 \oplus R^2 , d_{X_0})$ of $x^3 + x y^3  - v^2$ with twisted differential
\beq\label{eq:E7-calc-X0}
	d_{X_0}^1 = \begin{pmatrix} v & -x \\ x^2+y^3 & -v \end{pmatrix}
	\, , \quad
	d_{X_0}^0  = (d_{X_0}^1)^\#
	\, ,
\eeq
one of the seven indecomposable objects of $\mathrm{hmf}^{\mathrm{gr}}_{\mathbbm{C}[v,x,y],W^{(\mathrm{E}_7)} - v^2}$ listed in \cite{kst0511155}. 
In step (iii) we again choose a similarity transformation that removes the term $x u^3$ in the diagonal entries. The result of step (iv) reads
\begin{align} \label{eq:E7-calc-solution}
d_{X_u}^1 = \big(\begin{smallmatrix} a&b \\ c & d \end{smallmatrix}\big) \quad \text{with} \quad 
	a &= v -\tfrac{t^2- 10 t + 19}2 \, (s u)^9  + (t{-}2) \, y (s u)^5  + y^2 (s u) \, ,
	\nonumber \\
	b &= -x + (2 t{-}5) \, (s u)^6 + y (s u)^2  \, ,
	\nonumber \\
	c &= x^2+y^3 + (2 t{-}5)^2 \, (s u)^{12} + (2 t{-}5) \, x (s u)^6  
	\nonumber \\
	& \quad\quad + 2 (2 t{-}5) \, y  (s u)^8 + x y (s u)^2 + y^2 (s u)^4  \, ,
	\nonumber \\
	d &=  -v - \tfrac{t^2- 10 t + 19}2 \, (s u)^9  + (t{-}2) \, y (s u)^5  + y^2 (s u)
\end{align}
and $d_{X_0}^0  = (d_{X_0}^1)^\#$. This time~$s$ and~$t$ can be any solution of
\beq\label{eq:stE7}
	t^3 - 21 \, t + 37 = 0
	\, , \quad	
	s^{18} = 26220 \, t^2 + 67488 \, t - 376912  
	\, .
\eeq

We find that 
\beq
	\mathrm{qdim}_l(X_u) = -2s
	\, , \quad
	\mathrm{qdim}_r(X_u) = (-30 + 5 \, t + 2 \, t^2) s^{-1}
	 \, ,
\eeq
and the the possible $\mathbb{Q}$-gradings on~$X_u$ are given by
\begin{equation}
	X^0 = R[\alpha] \oplus R[\alpha-\tfrac13] \, , \quad
	X^1 = R[\alpha] \oplus R[\alpha-\tfrac13]
	\, , \quad \alpha \in \mathbb{Q} .
\end{equation}

\begin{lem}\label{lem:E7-calc-summary}
Let~$Y$ be a rank-three $\mathbb{Q}$-graded matrix factorization of $x^3 + xy^3 - u^{18} - v^2$. Suppose that (i)~$Y$ has non-zero quantum dimensions and (ii) when setting~$u$ to zero, $Y$ is equal to $X_0$ in~\eqref{eq:E7-calc-X0}. Then~$Y$ is isomorphic in $\mathrm{hmf}^{\mathrm{gr}}_{\mathbbm{C}[u,v,x,y],x^3 + xy^3 - u^{18} - v^2}$ to $X_u$ in~\eqref{eq:E7-calc-solution} for some solution $s,t$ of~\eqref{eq:stE7}.
\end{lem}

\medskip

As before we compute the endomorphisms of~$X_u$ in $\mathrm{hmf}^{\mathrm{gr}}_{\mathbbm{C}[u,v,x,y],x^3 + xy^3 - u^{18} - v^2}$ to be the 27-dimensional space 
\beq\label{eq:EndhmfX-E7}
\mathrm{End}
(X_u) 
\cong 
\Big(
\bigoplus_{i=0}^{16} \mathbbm{C}_{i\cdot \frac{2}{18}}
\Big)
\oplus 
\Big(
\bigoplus_{i=4}^{12} \mathbbm{C}_{i\cdot \frac{2}{18}}
\Big)
\oplus
\mathbbm{C}_{\frac{8}{18}}
\, , 
\eeq
and using the code of \cite{khovhom} together with row and column manipulations as in the previous example, for the solution
\beq 
t_\text{cft} = 3 \big( \zeta_{18} + \zeta^{-1}_{18}\big) - 2 \big(\zeta^2_{18} + \zeta^{-2}_{18}\big) ,
\eeq 
with $\zeta_d$ as in the $W_{E_6}$ case, to~\eqref{eq:stE7} we find that $X_{u'}^\dagger \otimes X_u$ is isomorphic, up to the factor $I_{v^2}$, to the rank-three graded matrix factorization
\beq\label{eq:AMF-E7}
P_{\{0\}} \oplus P_{\{-4,-3,\ldots,4\}} \oplus P_{\{-8,-7,\ldots,8\}} \, ,
\eeq
proving the second third of Corollary~\ref{cor:Etypemod}. 

\begin{rem}\label{rem:galE7}
As in Remark \ref{rem:galE6} one can check that the other solutions~$t$ to~\eqref{eq:stE7} form a single orbit under the Galois group $\mathrm{Gal}(\mathbb{Q}(\zeta_{18})/\mathbb{Q})$: the other two solutions are
\begin{align}
\sigma_5(t_\text{cft}) &= 3 \big( \zeta_{18}^5 + \zeta^{-5}_{18}\big) - 2 \big(\zeta^{10}_{18} + \zeta^{-10}_{18}\big) \, , \nonumber 
\\
\sigma_7(t_\text{cft}) &= 3 \big( \zeta_{18}^7 + \zeta^{-7}_{18}\big) - 2 \big(\zeta^{14}_{18} + \zeta^{-14}_{18}\big) \, .
\end{align}
Computing the actions of $\sigma_{5*}$ and $\sigma_{7*}$ on $\{-4,-3,\ldots,4\}$ and $\{-8,-7,\ldots,8\}$, for example $\sigma_{5*}(\{-4,-3,\ldots,4\}) = \{0,\pm 2,\pm 3,\pm 5,\pm 8\}$, one thus finds that 
\begin{align}
X^\dagger_{u'} \otimes X_u & \cong 
\big( P_{\{0\}} \oplus P_{\lbrace 0,\pm 2,\pm 3,\pm 5,\pm 8 \rbrace}\oplus P_{\{-8,-7,\ldots,8\}} \big) \otimes_k I_{v^2}
&  \text{for }   t = \sigma_5(t_\text{cft}) \, , \nonumber
\\
X^\dagger_{u'} \otimes X_u & \cong 
\big( P_{\{0\}} \oplus P_{\lbrace 0,\pm 3,\pm 4,\pm 7,\pm 8 \rbrace}\oplus P_{\{-8,-7,\ldots,8\}} \big) \otimes_k I_{v^2}
&  \text{for }  t = \sigma_7(t_\text{cft}) \, .
\end{align}
\end{rem}

\subsubsection*{$\boldsymbol{W^{(\mathrm{E}_8)} \sim W^{(\mathrm{A}_{29})}}$}

The $\mathrm{E}_8$-case is considerably more complicated than the cases already treated. This starts already in step (i) as the smallest factorizations of $W^{(\mathrm{E}_8)} - v^2$ are of rank four. Let us choose
\beq\label{eq:E8-calc-X0}
	d_{X_0}^1 = \begin{pmatrix}
		 -v & 0 & x & y \\
		0 & -v & y^4 & -x^2 \\
		x^2 & y & -v & 0 \\
		y^4 & -x & 0 & -v 
		\end{pmatrix}
		,\quad
	d_{X_0}^0 = \begin{pmatrix}
		v & 0 & x & y \\
		0 & v & y^4 & -x^2 \\
		x^2 & y & v & 0 \\
		y^4 & -x & 0 & v
		\end{pmatrix}
		 ,
\eeq
see \cite[Sect.\,5]{kst0511155}. This is a matrix factorization of $x^3 + y^5 - v^2$.

In step (ii), the most generic homogeneous deformation of $d_{X_0}^1$ (deforming also the zero entries, of course) has 82 free parameters. Via the similarity transformation in step (iii) one can reduce this to 60 parameters. We refrain from giving this general deformation explicitly. 

For step (iv) one has to use a different method than in the other cases, because now $\det(d_{X_0}^1) = (x^3 + y^5 - v^2)^2$ and imposing $\det(d_{X_u}^1) = (x^3 + y^5 - u^{30}-v^2)^2$ turns out to be impractical as it results in too many non-linear conditions. Instead, we make the ansatz $d_{X_u}^0 = (x^3 + y^5 - u^{30}-v^2)^{-1} \cdot (d_{X_u}^1)^\#$ and require that $d_{X_u}^0$ be a matrix with polynomial entries. As before, this leads to a (very long) matrix factorization of $x^3 + y^5 - u^{30}-v^2$ in terms of two parameters $s,t$, where $t$ satisfies an eighth order equation and $s^{30}$ is equal to some polynomial in $t$. The eighth order equation, however, is a product of two fourth order ones, and we select one of these and use it to simplify the matrix factorization. One is left with the matrix
$m := d_{X_u}^1$, where with $\varsigma=su$ the matrix entries $m_{ij}$ are as follows: 
\begin{align*}
m_{11} &= - v - \tfrac{(1 + t) (3 + t) (5 + 7t)}{64}  \varsigma^{15} - \tfrac{1 + t}4 \varsigma^5 x - \tfrac{19 + 47 t + 25 t^2 + 5 t^3}{192} \varsigma^9 y - \tfrac12 \varsigma^3 y^2 \, , \\
m_{12} &= \varsigma\, , \\
m_{13} &= x + \tfrac{(-1 + t) (23 + 36 t + 5 t^2)}{96} \varsigma^{10}\, , \\
m_{14} &=  y\, , \\
m_{21} &= \tfrac{-138089 - 562209 t - 600371 t^2 - 116355 t^3}{11520}) \varsigma^{29} 
+ \tfrac{-73 - 280 t - 285 t^2 - 50 t^3}{160} \varsigma^{19} x
\\ & \quad
+ \tfrac{-29 - 25 t + 25 t^2 + 5 t^3}{96} \varsigma^9 x^2 
+ \tfrac{-2107 - 8545 t - 9085 t^2 - 1735 t^3}{960} \varsigma^{23} y  
\\ & \quad
+ \tfrac{-33 - 57 t - 11 t^2 + 5 t^3}{64} \varsigma^{13} x y
+ \tfrac{(5 + 7 t) (13 + 36 t + 7 t^2)}{384} \varsigma^{17} y^2  
- \tfrac{3 + 4 t}{4} \varsigma^7 x y^2 
\\ & \quad 
+ \tfrac{-35 - 49 t + 7 t^2 + 5 t^3}{96} \varsigma^{11} y^3
- \varsigma x y^3 
- \tfrac12 (1 + t) \varsigma^5 y^4 \, , 
\\
m_{22} &= -v
+ \tfrac{(1 + t) (3 + t) (5 + 7 t)}{64} \varsigma^{15}
+ \tfrac{1 + t}{4} \varsigma^5 x
+ \tfrac{19 + 47 t + 25 t^2 + 5 t^3}{192} \varsigma^9 y
+ \tfrac12 \varsigma^3 y^2\, , \\
m_{23} &= y^4 
+ \tfrac{3587 + 14687 t + 15785 t^2 + 3125 t^3}{1920} \varsigma^{24}
+ \tfrac{(1 - t) (23 + 36 t + 5 t^2}{96} \varsigma^9 v
\\ & \quad
+ \tfrac{43 + 102 t + 67 t^2 + 12 t^3}{96} \varsigma^{14} x
- \tfrac{(1 + t) (81 + 126 t + 17 t^2)}{384} \varsigma^{18} y
+ \tfrac{2 + 3 t}{4} \varsigma^8 x y
\\ & \quad
+ \tfrac{(2 + t) (7 + 6 t - 5 t^2}{96} \varsigma^{12} y^2
+ \varsigma^2 x y^2 
+ \tfrac{1 + 2 t}{4} \varsigma^6 y^3\, , 
\\
m_{24} &= - x^2
+ \tfrac{(-1 + t) (23 + 36 t + 5 t^2)}{96} \varsigma^{10} x
+ \tfrac{2 + 21 t + 32 t^2 + 9 t^3}{48} \varsigma^{14} y\, , 
\\
m_{31} &= x^2
+ \tfrac{(1 - t) (23 + 36 t + 5 t^2)}{96} \varsigma^{10} x
- \tfrac{2 + 21 t + 32 t^2 + 9 t^3}{48} \varsigma^{14} y \, , 
\\
m_{32} &= y\, , \\
m_{33} &= - v
+ \tfrac{-37 - 39 t + 29 t^2 + 15 t^3}{192} \varsigma^{15} 
+ \tfrac{1 + t}4 \varsigma^5 x
+ \tfrac{-65 - 73 t + 37 t^2 + 5 t^3}{192} \varsigma^9 y
- \tfrac12 \varsigma^3 y^2\, , 
\\
m_{34} &= 
\tfrac{(1 - t) (23 + 36 t + 5 t^2)}{96} \varsigma^{11}
+ \varsigma x 
+ \tfrac{1 + t}2 \varsigma^5 y \, , 
\\
m_{41} &= y^4
+ \tfrac{3587 + 14687 t + 15785 t^2 + 3125 t^3}{1920} \varsigma^{24}
+ \tfrac{(-1 + t) (23 + 36 t + 5 t^2)}{96} \varsigma^9 v
\\ & \quad
+ \tfrac{43 + 102 t + 67 t^2 + 12 t^3}{96} \varsigma^{14} x
- \tfrac{(1 + t) (81 + 126 t + 17 t^2}{384} \varsigma^{18} y
+ \tfrac{2 + 3 t}4 \varsigma^8 x y
\\ & \quad
+ \tfrac{(2 + t) (7 + 6 t - 5 t^2)}{96} \varsigma^{12} y^2
+ \varsigma^2 x y^2 
+ \tfrac{1 + 2 t}4 \varsigma^6 y^3\, , \\
m_{42} &= -x 
+ \tfrac{(1 - t) (23 + 36 t + 5 t^2)}{96} \varsigma^{10} \, , 
\\
m_{43} &= 
-\tfrac{569 + 2615 t + 2855 t^2 + 425 t^3}{1920} \varsigma^{19}
+ \tfrac{17 + t - 37 t^2 - 5 t^3}{96} \varsigma^9 x
+ \tfrac{-17 - 17 t + 13 t^2 + 5 t^3}{64} \varsigma^{13} y 
\\ & \quad
- \tfrac{1 + 2 t}4 \varsigma^7 y^2
- \varsigma y^3 \, , 
\\
m_{44} &= -v
+ \tfrac{37 + 39 t - 29 t^2 - 15 t^3}{192} \varsigma^{15}
- \tfrac{1 + t}4 \varsigma^5 x  
+ \tfrac{65 + 73 t - 37 t^2 - 5 t^3}{192} \varsigma^9 y
+ \tfrac12  \varsigma^3 y^2 \, .
\end{align*}
The parameters $s,t$ can be any solution of
the equations
\begin{align}
& s^{30} = \tfrac{1}{4} (45308593275 \, t^3 - 32199587625 \, t^2 - 973905678975 \, t - 395277903075) 
\, , \nonumber 
\\
& 5 \, t^4 - 110 \, t^2 - 120 \, t - 31 = 0 \, .
\label{eq:stE8}
\end{align}

As already noted above, the matrix $d_{X_u}^0$ is given by $(x^3 + y^5 - u^{30}-v^2)^{-1} \cdot (d_{X_u}^1)^\#$, which has polynomial entries provided $s,t$ solve \eqref{eq:stE8}. 
It is now straightforward to determine the $\mathbb{Q}$-gradings on~$X_u$ to be given by
\begin{equation}
\begin{split}
	X^0 &= R[\alpha] \oplus R[\alpha-\frac{14}{15}] \oplus R[\alpha-\frac{1}{3}] \oplus R[\alpha-\tfrac35] \, , \\
	X^1 &= R[\alpha] \oplus R[\alpha-\frac{14}{15}] \oplus R[\alpha-\frac{1}{3}] \oplus R[\alpha-\tfrac35]
	\end{split}
\end{equation}
with $\alpha \in \mathbb{Q}$, 
and compute the quantum dimensions to be
\beq
	\mathrm{qdim}_l(X_u) = 2s
	\, , \quad
	\mathrm{qdim}_r(X_u) = \tfrac{5}{16} (-27 - 86 \, t - 3 \, t^2 + 4 \, t^3)  \, s^{-1} \ ,
\eeq
which are indeed non-zero for all choices of $s,t$. 
This concludes the proof of Theorem~\ref{thm:ADEorbifolds}. 

Because we have made a number of restricting assumptions before arriving at $(X,d_{X_u})$, we cannot claim a statement analogous to Lemmas \ref{lem:D-calc-summary}, \ref{lem:E6-calc-summary} and \ref{lem:E7-calc-summary}. 

\medskip

Finally we compute the endomorphisms of~$X_u$ in $\mathrm{hmf}^{\mathrm{gr}}_{\mathbbm{C}[u,v,x,y],x^5 + y^3 - u^{30} - v^2}$ to be the 60-dimensional space 
\beq\label{eq:EndhmfX-E8}
\mathrm{End}
(X_u) 
\cong 
\Big(
\bigoplus_{i=0}^{28} \mathbbm{C}_{i\cdot \frac{2}{30}}
\Big)
\oplus 
\Big(
\bigoplus_{i=5}^{23} \mathbbm{C}_{i\cdot \frac{2}{30}}
\Big)
\oplus 
\Big(
\bigoplus_{i=9}^{19} \mathbbm{C}_{i\cdot \frac{2}{30}}
\Big)
\oplus
\mathbbm{C}_{\frac{28}{30}}
\, .
\eeq
For the solution
\beq
	t_\text{cft} = -\tfrac15 \Big( 7 + 4 \big(\zeta_{30} + \zeta_{30}^{-1} \big) + 8 \big(\zeta_{30}^2 + \zeta_{30}^{-2} \big) - 16 \big(\zeta_{30}^3 + \zeta_{30}^{-3} \big) \Big)
\eeq
of~\eqref{eq:stE8}, $X_{u'}^\dagger \otimes X_u$ is isomorphic, up to the factor $I_{v^2}$, to the rank-four graded matrix factorization
\beq\label{eq:AMF-E8}
P_{\{0\}} \oplus P_{\{-5,-4,\ldots,5\}} \oplus P_{\{-9,-8,\ldots,9\}} \oplus P_{\{-14,-13,\ldots,14\}} \, .
\eeq
This concludes the proof of Corollary~\ref{cor:Etypemod}. 

\begin{rem}\label{rem:galE8}
The other solutions of~\eqref{eq:stE8} are found via the Galois group $\mathrm{Gal}(\mathbb{Q}(\zeta_{30})/\mathbb{Q})$ to be 
\begin{align}
\sigma_7(t_\text{cft}) &= 
	-\tfrac15 \Big( 7 + 4 \big(\zeta_{30}^{7} + \zeta_{30}^{-7} \big) + 8 \big(\zeta_{30}^{14} + \zeta_{30}^{-14} \big) - 16 \big(\zeta_{30}^{21} + \zeta_{30}^{-21} \big) \Big) \, , \nonumber \\
\sigma_{11}(t_\text{cft}) &= 
	-\tfrac15 \Big( 7 + 4 \big(\zeta_{30}^{11} + \zeta_{30}^{-11} \big) + 8 \big(\zeta_{30}^{22} + \zeta_{30}^{-22} \big) - 16 \big(\zeta_{30}^3 + \zeta_{30}^{-3} \big) \Big) \, , \nonumber \\
\sigma_{13}(t_\text{cft}) &= 
	-\tfrac15 \Big( 7 + 4 \big(\zeta_{30}^{13} + \zeta_{30}^{-13} \big) + 8 \big(\zeta_{30}^{26} + \zeta_{30}^{-26} \big) - 16 \big(\zeta_{30}^9 + \zeta_{30}^{-9} \big) \Big) \, . 
\end{align}
The corresponding decompositions of $X_{u'}^\dagger \otimes X_u$ are, with $Z = P_{\{0\}} \oplus P_{\{-14,-13,\ldots,14\}}$ and up to the factor $I_{v^2}$,
\begin{align}
& Z \oplus 
P_{\{0,\pm 2,\pm 5,\pm 7,\pm 9,\pm 14\}} \oplus P_{\{0,\pm 2,\pm 3,\pm 4,\pm 5,\pm 7,\pm 9,\pm 11,\pm 12,\pm 14\}} 
&&  \text{for }  t = \sigma_7(t_\text{cft}) \, , \nonumber \\
& Z \oplus 
P_{\{0,\pm 3,\pm 5,\pm 8,\pm 11,\pm 14\}} \oplus P_{\{0,\pm 2,\pm 3,\pm 5,\pm 6,\pm 8,\pm 9,\pm 11,\pm 13,\pm 14\}} 
&&  \text{for }  t = \sigma_{11}(t_\text{cft})\, , \nonumber \\
& Z \oplus 
P_{\{0,\pm 3,\pm 4,\pm 5,\pm 8,\pm 9,\pm 13\}} \oplus P_{\{0,\pm 1,\pm 4,\pm 5,\pm 8,\pm 9,\pm 12,\pm 13,\pm 14 \}} 
&&  \text{for }  t = \sigma_{13}(t_\text{cft}) \, .
\end{align}
\end{rem}

\subsubsection*{Another method to construct orbifold equivalences}

There is another way, slightly different from the one described at the beginning of Section~\ref{subsec:simple-sing}, of obtaining the orbifold equivalences $W^{(\mathrm{D}_{d/2+1})} \sim W^{(\mathrm{A}_{d-1})}$, $W^{(\mathrm{E}_6)} \sim W^{(\mathrm{A}_{11})}$, and $W^{(\mathrm{E}_7)} \sim W^{(\mathrm{A}_{17})}$. 
Roughly, this method starts with a matrix factorization~$M$ of a potential~$W$ and computes the deformations \cite{Laudal1983, Siqveland2001}~$M_u$ of~$M$. Then instead of setting the deformation parameters to be solutions of the obstruction equations (which would give factorizations of~$W$), one tries to re-interpret some of the parameters~$u$ as variables of another potential~$V$ while choosing the remaining parameters such that~$M_u$ becomes a factorization of $W-V$. 

As this method may prove useful to construct further orbifold equivalences, we illustrate it in more detail in the example of $W^{(\mathrm{E}_6)} \sim W^{(\mathrm{A}_{11})}$: 
\begin{enumerate}
\item Start with the matrix factorization~\eqref{eq:E6-calc-X0} of $W=W^{(\mathrm{E}_6)}-v^2$ and compute its deformations, using e.\,g.~the implementation of \cite{cdr}.\footnote{%
Of course the indecomposable object~\eqref{eq:E6-calc-X0} has no nontrivial deformations, but that does not matter here.} 
\item One obtains a 4-by-4 matrix~$D$ with polynomial entries in the variables $v,x,y$ and two deformation parameters $u_1, u_2$. By construction~$D$ squares to $W \cdot 1 + R$, where the matrix~$R$ vanishes if $u_1, u_2$ satisfy the obstruction equations. 
\item Interpret the parameter~$u_1$ as the (rescaled) variable~$u$ and choose $u_2 \in \mathbbm{C}[u]$ such that~$D$ becomes a factorization of $W^{(\mathrm{E}_6)}-u^{12}-v^2$. This turns out to be isomorphic to~$X_u$ of~\eqref{eq:E6-calc-solution} with $t=-1/\sqrt{3}$. 
\end{enumerate}

In the cases $W+v^2 \in \{ W^{(\mathrm{D}_{d/2+1})}, W^{(\mathrm{E}_6)}, W^{(\mathrm{E}_7)} \}$ this method is especially straightforward as the matrix~$D$ already squares to $(W+f)\cdot 1$, where~$f$ is a polynomial with leading term $u_1^{\,d/2}, u_1^{12}, u_1^{18}$, resp. 
On the other hand, for $W = W^{(\mathrm{E}_8)}-v^2$ the square of~$D$ is a more generic matrix, rendering the problem $D^2 = (W - u^{30}) \cdot 1$ much more computationally involved.

\subsection{Comparison to conformal field theory}
\label{sec:CFTcomparison}

In this subsection we describe how the orbifold equivalences we just described are precisely the expected ones from the Landau-Ginzburg/conformal field theory correspondence, and compare the conjectured algebra objects and defect spectra with those obtained.

Recall here that conformal field theories (as described in Section \ref{functorialapproach}) form a bicategory. Because of Proposition \ref{Frobbicatadj}, the general concept of orbifold equivalence is applicable, and then, there cannot be a topological defect joining two CFTs of different Virasoro central charge. As a consequence, equality of central charges provides a \textsl{necessary} condition for an orbifold equivalence to exist, cf.~Proposition~\ref{prop:necessary}. Contrary to the matrix factorization framework, for CFTs a useful \textsl{sufficient} condition is known for the existence of an orbifold equivalence \cite{ffrs}: 
\begin{quote}
Let $\mathcal{V}$ be a rational vertex operator algebra. 
Suppose two CFTs have unique bulk vacua and their algebra of bulk fields contains $\mathcal{V} \otimes_\mathbbm{C} \overline{\mathcal{V}}$ as a subalgebra. Then these two CFTs are orbifold equivalent.
\end{quote}
Here the bar over $\mathcal{V}$ indicates that the second factor is embedded in the anti-holomorphic fields. 

Those LG models whose potentials define simple singularities are believed to renormalize to CFTs that contain the $\mathcal N=2$ minimal super Virasoro vertex operator algebra with the same central charge. The latter are rational, so by the above criterion all these CFTs of the same central charge are orbifold equivalent.

One may expect topological defects between two infrared CFTs to have analogues already in the corresponding LG model. Furthermore, one may expect that composition of topological defects commutes with renormalisation group flow. If so, the orbifold equivalences of infrared CFTs should exist also for LG models. This is the reason why Theorem~\ref{thm:ADEorbifolds} is expected from the above CFT criterion.

After these qualitative considerations, we now turn to quantitative, more technical comparisons. The correspondence between topological defect lines of LG models of A-type singularities and those of the associated diagonal $\mathcal N=2$ CFTs motivates the following conjecture, whose ingredients we explain directly after stating it.

\begin{conj}\label{conj:tensor-equiv}
For any integer $d \geqslant 3$, consider the monoidal subcategory of the category $\mathrm{hmf}^{gr}(\mathbbm{C}[x,x'],x^d - x'^d)$ whose morphisms only have $\mathbb{Q}$-degree zero, and which is generated by $\{ P_{\{a,a+1,\dots,a+b\}} \,|\, a,b \in \mathbbm{Z}_d\}$ with respect to tensor products and direct sums. This subcategory is monoidally equivalent to $\mathcal{C}_d$.
\end{conj}

Where we have followed the notation introduced along this thesis.

We denote the simple objects of~$\mathcal{C}_d$ by $U_{l,m}$, where $l \in \{0,1,\dots,d-2\}$ and $m \in \{0,1,\dots,2d-1\}$. As previously mentioned in Chapter \ref{LGCFTTL}, via the coset construction, this describes the NS sector of the representations of the $\mathcal N=2$ minimal super Virasoro vertex operator algebra at central charge $c = 3 (1-2/d)$, see again \cite[App.\,A.2]{carqrunkel1} for details and references.\footnote{%
In \cite{carqrunkel1} the $(-)_{\textrm{NS}}$ is missing in (A.38) and (A.45). This is a typo or an error, depending on one's disposition towards the authors.
} 

\medskip

The status of Conjecture~\ref{conj:tensor-equiv} is currently as follows.
\begin{itemize}
\item[-] The functor which conjecturally provides the tensor equivalence acts on simple objects as $P_{\{a,a+1,\dots,a+b\}} \mapsto U_{b,b+2a}$. It is verified in \cite{brunrogg1} that this is compatible with the tensor product on the level of isomorphism classes.
\item[-] Some, but far from all, associativity isomorphisms were proved to be compatible in \cite{carqrunkel1}.
\item[-] For odd $d$, this was proved in Chapter \ref{LGCFTTL}.
\end{itemize}

\subsubsection*{Comparison of algebra objects}

According to \cite{tft1, Fjelstad:2006aw}, the full conformal field theories (with unique bulk vacuum) that can be constructed starting from a rational vertex operator algebra $\mathcal{V}$ are parametrised by Morita classes of special symmetric Frobenius algebras in the modular tensor category of representations of $\mathcal{V}$. The algebras relevant for the CFT describing the infrared fixed point of an LG model with ADE-type potential are predicted from \cite{gannon, gray} to be non-trivial only in the $\mathfrak{su}(2)$ factor of $\mathcal{D}$. More specifically, they are representatives of the ADE classification of Morita classes of such algebras in $\mathcal{C}_{d-2}^{\mathfrak{su}(2)}$ given in \cite{ostrik}. As objects in~$\mathcal{D}$ these algebras are
\begin{align} \label{eq:CFT-ADE-list}
F^{(\textrm{A}_{d-1})} &= U_{0,0} && \text{for $d \geqslant 2$} \, ,
\nonumber \\
F^{(\textrm{D}_{d/2+1})} &= U_{0,0} \oplus U_{d-2,0}  && \text{for $d\in 2\mathbb{Z}_+$} \, ,
\nonumber \\
F^{(\textrm{E}_{6})} &= U_{0,0} \oplus U_{6,0}  && \text{for $d=12$} \, ,
\nonumber \\
F^{(\textrm{E}_{7})} &= U_{0,0} \oplus U_{8,0} \oplus U_{16,0} && \text{for $d=18$} \, ,
\nonumber \\
F^{(\textrm{E}_{8})} &= U_{0,0} \oplus U_{10,0} \oplus U_{18,0} \oplus U_{28,0} && \text{for $d=30$} \, .
\end{align}

Topological defects between a diagonal, i.\,e.~A-type, CFT and another CFT from the above list with the same value of $d$ given by some algebra $F$ are described by $F$-modules in $\mathcal{D}$ \cite{Frohlich:2006ch}. This is consistent with the point of view of orbifold equivalences. There, the algebra describing the theory on one side of a topological defect $X$ in terms of the other is $X^\dagger \otimes X$, see \cite{carqrunkel1} for details. In the present setting one has $X = {}_FF$, $X^\dagger = F_F$, and $F_F \otimes_F {}_FF \cong F$ as algebras. 

On the matrix factorization side, the objects underlying the algebras describing the D- and E-type singularities as orbifolds of A-type singularities are those in \eqref{eq:PPinDcase}, \eqref{eq:AMF-E6}, \eqref{eq:AMF-E7} and~\eqref{eq:AMF-E8} above. Under the tensor equivalence of Conjecture \ref{conj:tensor-equiv}, they are indeed mapped to the corresponding objects in the list~\eqref{eq:CFT-ADE-list}. (It would of course be enough to land in the same Morita class, but for our choices of matrix factorizations we get the actual representatives chosen in \eqref{eq:CFT-ADE-list}.)

\subsubsection*{Comparison of defect spectra}

Given the above observations on the objects underlying the algebras establishing the orbifold equivalences, it is natural to expect that the matrix factorizations of the potential differences described in Section \ref{subsec:simple-sing} get mapped to the topological defects described by the modules ${}_FF$ for $F$ the corresponding algebra in \eqref{eq:CFT-ADE-list}. This expectation can be tested by comparing the spectra of chiral primaries. 

The chiral primaries in $\mathcal{D}$ are the ground states in the representations labelled $U_{l,l}$ with $l \in \{0,1,\dots,d-2\}$; their charge is $l/d$. The space of chiral primaries of holomorphic and antiholomorphic labels $l$ and $m$, resp., on the defect ${}_FF$ is isomorphic to the vector space
\beq
\Hom_F(F \otimes U_{l,l} \otimes U_{m,-m}, F) \cong \Hom(U_{l,l} \otimes U_{m,-m}, F) \, , 
\eeq
where `$\Hom_F$' comprises only $F$-module maps in $\mathcal{D}$, and `$\Hom$' all maps in $\mathcal {D}$. We refer to \cite{Frohlich:2006ch} for an explanation of this formula. 
The total charge (as seen on the LG side) of these chiral primaries is $(l+m)/d$.

\medskip

Let us consider the example $F = F^{(\textrm{E}_{6})}$ in some detail. The first thing to note is that $\Hom(U_{l,l} \otimes U_{m,-m}, U_{n,0})$ has dimension zero unless $l=m$. (Actually the fusion rules in $\mathcal{C}_{2d}^{\mathfrak{u}(1)}$ tell us that $l-m \cong 0 \mod 2d$, but for the given range on~$l$ and~$m$ this just amounts to $l=m$.) For $l=m$, the space is one-dimensional if~$n$ is even and $n/2 \leqslant l \leqslant d-2-n/2$, as follows from the fusion rules of $\mathcal{C}_{d-2}^{\mathfrak{su}(2)}$.
From the summand $U_{0,0}$ in $F^{(\textrm{E}_{6})}$ we therefore get one state at each charge $l/6$, $l \in \{0,1,\dots,10\}$, and from the summand $U_{6,0}$ another state of charge $l/6$ for each $l \in \{3,4,\dots,7\}$. 
Hence the total dimension of the space of chiral primaries is 16, and we have perfect agreement with~\eqref{eq:EndhmfX-E6}. 

It is straightforward to carry out the analogous computations for $F = F^{(\textrm{D}_{d/2+1})}$, $F = F^{(\textrm{E}_{7})}$ and $F = F^{(\textrm{E}_{8})}$, again consistent with the charges and multiplicities of the matrix factorization results listed in~\eqref{eq:PPinDcase}, \eqref{eq:EndhmfX-E7} and~\eqref{eq:EndhmfX-E8}, resp.

\section{Further examples of orbifold equivalent potentials}
\label{sec:BH}

At this point, one may wonder how to find further examples of orbifold equivalences, and if possible in a structured way: this is the aim of current work in progress \cite{roscamacho2}. For that, our main strategy/idea is via the so-called \textit{Berglund-H{\"u}bsch} transposition, which was first defined in \cite{berglundhuebsch}. Let $W=W \left( x_1,\ldots,x_n \right)$ be a potential, and fix $\mathbbm{k}=\mathbbm{C}$.

\begin{defn}
Consider a homogeneous polynomial $W \left( x_1,\ldots,x_n \right)$ whose number of variables is equal to the number of monomials in $W$, namely, $$W \left( x_1,\ldots,x_n \right)= \sum_{i=1}^n a_i \prod_{j=1}^n x_j^{E_{ij}}$$ for some coefficients $a_i \in \mathbbm{C}^*$ and $E_{ij} \in \mathbb{Z}_{\geq 0}$. The \textit{Berglund-H{\"u}bsch transpose} of $W$, $W^T$ is defined to be the polynomial: $$W^T \left( x_1,\ldots,x_n \right)=\sum_{i=1}^n a_i \prod_{j=1}^n x_j^{E_{ji}}$$ We denote as $E=\left( E_{ij} \right)$ the matrix of exponents, and as $E^T=\left( E_{ji} \right)$ its Berglund-H{\"u}bsch transposed matrix of exponents.
\end{defn}

As it was proved in \cite{berglundhuebsch} and explored in other later works (e.g. \cite{ebelingtakahashi}), this operation provides pairs of mirror symmetric manifolds. Mirror symmetry is expected to be a special case of orbifold equivalence, hence this strategy should be a candidate to generate new orbifold equivalences.

Let us start with the three ADE potentials from the previous section. $W_{A_{d-1}}$ as well as $W_{E_6}$ and $W_{E_8}$, have diagonal matrices of exponents and hence $E=E^T$ for these three cases. Thus, $W_{A_{d-1}}=W^T_{A_{d-1}}$, $W_{E_6}=W^T_{E_6}$ and $W_{E_8}=W^T_{E_8}$. That means: these three potentials are orbifold equivalent via the identity matrix factorization.

The remaining two polynomials, $W_{E_7}$ and $W_{D_{d+1}}$ are slightly different. For $W_{E_7}$,
\begin{equation}
E=\left( \begin{matrix} 3 & 0 \\ 1 & 3 \end{matrix} \right) \quad \rightsquigarrow \quad E^T=\left( \begin{matrix} 3 & 1 \\ 0 & 3 \end{matrix} \right)
\nonumber
\end{equation}
but the resulting Berglund-H{\"u}bsch transposed polynomial is the same as the original one, after an exchange of variables $x \leftrightsquigarrow y$. The only polynomial which will somehow change after Berglund-H\"{u}bsch transposing is $W_{D_{d+1}}$, giving $W^T_{D_{d+1}}=x^d y+y^2$.

One can check that $W_{D_{d-1}}^T$ is equivalent to:
\begin{itemize}
\item $W_{A_{d-1}}$: for the potential $W=x^d y+y^2-u^{2d}-v^2$, consider the matrix factorization given by the differential
\begin{equation}
d_1=\left( \begin{matrix}
x-a u & -\left(y-i u^d-v \right) \\
y-i u^d+v & y \frac{x^d-\left( a u \right)^d}{x-a u}
\end{matrix} \right)
\nonumber
\end{equation} where $a$ satisfies $a^d=-2i$, with $d_0=\mathrm{det} \left( d_1 \right) d_1^{-1}$. Computing the quantum dimensions of this matrix factorization, we find out that $$\mathrm{qdim}_r=-a^{-1}$$  $$\mathrm{qdim}_l=-a$$
Hence $W_{D_{d+1}^T}$ and $W_{A_{d-1}}$ are orbifold equivalent.
\item $W_{D_{d+1}}$: for the potential $W=u^d+v^2u-x^d y-y^2$, consider the matrix factorization given by the differential
\begin{equation}
d_1=\left( \begin{matrix}
\frac{u^d-\left( x^2 a^{-2} \right)^d}{u-a^{-2} x^2} +v^2 & -\left(y+\frac{x^d}{2}-\frac{x v}{a} \right) \\
-y-\frac{x^d}{2}+\frac{x v}{a} & u-a^{-2} x^2
\end{matrix} \right)
\nonumber
\end{equation} where $a$ satisfies $a^d=-2i$, with $d_0=\mathrm{det} \left( d_1 \right) d_1^{-1}$. Computing the quantum dimensions of this matrix factorization, we find out that $$\mathrm{qdim}_r=-a$$  $$\mathrm{qdim}_l=-\frac{2}{a}$$
\end{itemize}
But notice here that the  $W_{D_{d+1}^T}$ is $W_{D_{d+1}}$ after a clever change of variables ($y=\frac{xv}{a}-\frac{x^d}{2}$, $x^2=u a^2$), and the $W_{A_{d-1}} \sim_{orb} W_{D_{d+1}}$ equivalence was already known. Hence, strictly speaking, this is not a new orbifold equivalence.

As the simple singularities are only some special polynomials belonging to larger classes of polynomials, one may wonder what other kinds of polynomials can one expect, in the hope that maybe one could trace back other examples of orbifold equivalences via Berglund-H{\"u}bsch transposition.

Let us first review some already existing results \footnote{Polynomials over an algebraically closed field (say e.g. $\mathbbm{C}$) with an isolated singularity in 0, i.e. a single critical point at 0, which have finite-dimensional Jacobi ring, are called \textit{nondegenerate}. Potentials are indeed nondegenerate polynomials.}. According to Arnold \cite{arnold},

\begin{prop}
Any homogeneous potential of two variables of corank two\footnote{By cokernel of a function we mean the dimension of the kernel of the differential of the map.} contains with nonzero coefficients one of the following sets of monomials:
\begin{center}
\begin{tabular}{c|c|}
 & Monomials \\ \hline
Class I & $\lbrace x^a,y^b \rbrace$ \\
Class II & $\lbrace x^a,x y^b \rbrace$ \\
Class III & $\lbrace x^a y,x y^b \rbrace$ \\ \hline
\end{tabular}
\end{center}
(where $a,b \in \mathbb{N}$, $a,b \geq 2$).
\end{prop}

Analogously, a similar phenomenon happens with potentials of three variables.

\begin{prop}
Every homogeneous potential on three variables of corank 3 contains with nonzero coefficients (for a suitable numbering of the variables) one or other of the following sets of monomials: 
\begin{center}
\begin{tabular}{c|c|}
 & Monomials \\ \hline
Class I & $\lbrace x^a,y^b,z^c \rbrace$ \\
Class II & $\lbrace x^a,y^b,z^c y \rbrace$ \\
Class III & $\lbrace x^a, y^b x,x z^c \rbrace$ \\ 
Class IV & $\lbrace x^a,y^b z,z^c y \rbrace$ \\
Class V & $\lbrace x^a,y^b z, z^c x \rbrace$ \\
Class VI & $\lbrace x^a y,y^b x,z^c x \rbrace$ \\
Class VII & $\lbrace x^a y, y^b z,z^c x \rbrace$ \\
\hline
\end{tabular}
\end{center}
where $a,b,c \in \mathbb{N}$, $a,b,c \geq 2$, except for classes III and VI there are some conditions on the powers $a,b,c$:
\begin{enumerate}
\item For Class III, the least common multiple $\left[ b,c \right]$ of $b$ and $c$ is divisible by $a-1$.
\item For Class VI, $\left( b-1 \right) c$ is divisible by the product of $a-1$ and the greatest common divisor $\left( b,c \right)$ of $b$ and $c$.
\end{enumerate}
\label{3variablesArnold}
\end{prop}

A special kind of polynomials of certain interest in several papers by Takahashi and Ebeling (e.g. \cite{ebelingtakahashi}) are the so-called \textit{invertible polynomials}, which one can define with a given potential and its Berglund-H\"{u}bsch transpose.

\begin{defn}
A homogeneous polynomial $W \left( x_1,\ldots,x_n \right)$ is \textit{invertible} if the following conditions are satisfied:
\begin{enumerate}
\item The number of variables has to be equal to the number of monomials in $W$, namely, $$W \left( x_1,\ldots,x_n \right)= \sum_{i=1}^n a_i \prod_{j=1}^n x_j^{E_{ij}}$$ for some coefficients $a_i \in \mathbbm{C}^*$ and $E_{ij} \in \mathbb{Z}_+$.
\item The matrix $E:=\left( E_{ij} \right)$ is invertible over $\mathbb{Q}$.
\item The Jacobian rings of $W$ and $W^T$ have both to be finite dimensional algebras over $\mathbbm{C}$ and the dimension of the Jacobi rings of $W$ and $W^T$ have to be greater or equal than 1.
\end{enumerate}
\end{defn}

In \cite{kreuzerskarke}, one can find the following result, claiming that any invertible polynomial $W$ is a (Thom-Sebastiani) \footnote{\cite{wang} If the sum of two homogeneous potentials $W' \left( x_1,\ldots,x_l \right)+W''\left( x_{l+1},\ldots,x_n \right)$ of degree $d$ for a choice of homogeneous coordinates $\left( x_i \right)_{i=1}^n$ and some $2 \leq l \leq n-1$ represents a homogeneous potential $W$  of degree $d$, we say this sum is a \textit{Thom-Sebastiani sum} of potentials.} sum of invertible polynomials of the following types:
\begin{enumerate}
\item $x_1^{a_1}$ (``\textit{Fermat type}");
\item $x_1^{a_1}x_2+x_2^{a_2}x_3+\ldots+x_{m-1}^{a_{m-1}}x_m+x_m^{a_m}$ (``\textit{chain type}"; $m \geq 2$);
\item $x_1^{a_1}x_2+x_2^{a_2}x_3+\ldots+x_{m-1}^{a_{m-1}}x_m+x_m^{a_m}x_1$ (``\textit{loop type}"; $m \geq 2$)".
\end{enumerate}

The $E$ and $E^T$ matrices for the Fermat type are in both cases a diagonal matrix. For the chain case, they look slightly different:
\begin{equation}
\begin{split}
E &=\left( 
\begin{matrix}
a_1 & 1 & 0 & 0 & \ldots & 0 & 0 \\
0 & a_2 & 1 & 0 & \ldots & 0 & 0 \\
0 & 0 & a_3 & 1 & \ldots & 0 & 0 \\
\vdots & \vdots & \vdots & \vdots & \ddots & \vdots & \vdots \\
0 & 0 & 0 & 0 & \ldots & a_{m-1} & 1 \\
0 & 0 & 0 & 0 & \ldots & 0 & a_{m} 
\end{matrix}
\right) \\
E^T &=\left( 
\begin{matrix}
a_1 & 0 & 0 & 0 & \ldots & 0 & 0 \\
1 & a_2 & 0 & 0 & \ldots & 0 & 0 \\
0 & 1 & a_3 & 0 & \ldots & 0 & 0 \\
\vdots & \vdots & \vdots & \vdots & \ddots & \vdots & \vdots \\
0 & 0 & 0 & 0 & \ldots & a_{m-1} & 0 \\
0 & 0 & 0 & 0 & \ldots & 1 & a_{m} 
\end{matrix}
\right)
\nonumber
\end{split}
\end{equation}
and same for the loop case,
\begin{equation}
\begin{split}
E &=\left( 
\begin{matrix}
a_1 & 1 & 0 & 0 & \ldots & 0 & 0 \\
0 & a_2 & 1 & 0 & \ldots & 0 & 0 \\
0 & 0 & a_3 & 1 & \ldots & 0 & 0 \\
\vdots & \vdots & \vdots & \vdots & \ddots & \vdots & \vdots \\
0 & 0 & 0 & 0 & \ldots & a_{m-1} & 1 \\
1 & 0 & 0 & 0 & \ldots & 0 & a_{m} 
\end{matrix}
\right) \\
E^T &=\left( 
\begin{matrix}
a_1 & 0 & 0 & 0 & \ldots & 0 & 1 \\
1 & a_2 & 0 & 0 & \ldots & 0 & 0 \\
0 & 1 & a_3 & 0 & \ldots & 0 & 0 \\
\vdots & \vdots & \vdots & \vdots & \ddots & \vdots & \vdots \\
0 & 0 & 0 & 0 & \ldots & a_{m-1} & 0 \\
0 & 0 & 0 & 0 & \ldots & 1 & a_{m} 
\end{matrix}
\right)
\nonumber
\end{split}
\end{equation}

These matrices allow us to compute the degrees of the variables for each considered case. These values are given in detail by Table \ref{degreesvariables}\footnote{The sum over $l'$ between $k+1$ and $k-1$ in the numerator of the degree $\vert x_k^{E^T} \vert$ for the loop case is meant to be run between $k+1$ to $m$ and from $1$ to $k-1$. We shortcut notation in the written down way hoping it is not too confusing.}. 
\begin{table}
\begin{center}
\begin{tabular}{c|c|c|}
& $\vert x_k^E \vert$ & $\vert x_k^{E^T} \vert$\\ \hline

Fermat & $\frac{2}{a_k}$ & $\frac{2}{a_k}$ \\ 

\hline

Chain & $2 \sum\limits_{i=k}^m \left( -1 \right)^{i-k} \prod\limits_{l=k}^i a_l^{-1}$ & $2 \sum\limits_{i=1}^k \left( -1 \right)^{k-i} \prod\limits_{l=i}^k a_l^{-1}$ \\ 

\hline

Loop & $2 \frac{1+\mathbf{\sum\limits_{l'=k-1}^{k+1}} \prod\limits_{i'=k-1}^{l'} \left( -a_{i'} \right)}{1+\left(-1 \right)^{m+1} \prod\limits_{i'=1}^m a_{i'}}$ & $2 \frac{1+\sum\limits_{l'=k+1}^{k-1} \prod\limits_{i'=k+1}^{l'} \left( -a_{i'} \right)}{1+\left(-1 \right)^{m+1} \prod\limits_{i'=1}^m a_{i'}}$  \\ \hline 
\end{tabular}
\caption{Degrees of variables of invertible polynomials. $2 \leq k \leq m$; $l,i,l',i'$ mod $m$ with representatives $\lbrace 1,\ldots,m \rbrace$. Boldface indices means we sum backwards.}
\label{degreesvariables}
\end{center}
\end{table}
These formulas are easily proved by induction. For instance, take the case of chain polynomials. For a start, compute the degrees of $x_m$, $x_{m-1}$ and $x_{m-2}$ from $E$:
\begin{equation}
\begin{split}
\vert x_m^E \vert &=\frac{2}{a_m} \\
\vert x_{m-1}^E \vert &=\frac{2}{a_{m-1}} \left( 1-\frac{1}{a_m} \right) \\
\vert x_{m-2}^E \vert &=\frac{2}{a_{m-2}} \left( 1-\frac{1}{a_{m-1}}+\frac{1}{a_{m-1} a_m} \right)
\end{split}
\nonumber
\end{equation}
Then, claim that $\vert x_k^E\vert=2 \sum\limits_{i=k}^m \left( -1 \right)^{i-k} \prod\limits_{l=k}^i a_l^{-1}$, and prove by induction. The induction base is clear for $k=m,m-1,m-2$, and proceed with the induction step:
\begin{equation}
\begin{split}
\vert x_{k+1} \vert &=\frac{1}{a_{k+1}} \left( 2-\vert x_k \vert \right)=\frac{2}{a_{k+1}} \left( 1-\sum\limits_{i=k}^m \left( -1 \right)^{i-k} \prod\limits_{l=k}^i a_l^{-1} \right) \\ &=\frac{2}{a_{k+1}}+2 \sum\limits_{i=k}^m \left( -1 \right)^{i-k-1} \prod\limits_{l=k}^i a_{k+1}^{-1} a_l^{-1} =2 \sum\limits_{i=k+1}^m \left( -1 \right)^{i-k-1} \prod\limits_{l=k+1}^i a_l^{-1}
\end{split}
\nonumber
\end{equation}
as expected.

Thanks to this result, we can state that in fact,

\begin{lem}
Invertible polynomials have the same central charge than their Berglund-H\"{u}bsch transpose.
\label{invBHsamecharge}
\end{lem}
\begin{proof}
The Fermat case is trivial. For the other two cases, perform the sum of all $\vert x_k^E \vert$ (resp. $\vert x_k^{E^T} \vert$). Let us explain in a bit of detail the case of chain as an enlightening case --the loop one is analogous. We need to sum all $\vert x_k^E \vert$ for values of $k$ between 1 and $m$, and the same with $\vert x_k^{E^T} \vert$. For $\vert x_k^E \vert$, we sum between $k$ and $m$ products of $a_l^{-1}$ between $k$ and $i$, where for $\vert x_k^{E^T} \vert$ we sum between $1$ and $k$ products of $a_l^{-1}$ between $i$ and $k$. One can realize here that we are only transposing the way we run over the indices of the sums and products -- but still, covering the same values. Hence we obtain the same total sum and thus the same central charge.
\end{proof}

Looking again at the Arnold classification, we can recognize that:
\begin{itemize}
\item For two variables, Class I is clearly Fermat, Class II is chain and Class III is obviously loop.
\item For three variables, let us call those classes which are not of any invertible kind (nor combinations of these) simply \textit{not invertible}. Then: 
\begin{center}
\begin{tabular}{c|c}
Class & Invertibility \\\hline
I & Fermat \\
II & Fermat+chain \\
III & not invertible \\
IV & Fermat+loop \\
V & chain \\
VI & not invertible \\
VII& loop \\ \hline
\end{tabular}
\end{center}
\end{itemize}

An obvious corollary then for Lemma \ref{invBHsamecharge} is that potentials of Classes I, II, IV, V and VII have the same charge than their corresponding Berglund-H{\"u}bsch transpose. Surprisingly, a direct computation pointed out that:
\begin{lem}
For three variables, polynomials of Classes III and VI have the same central charge as their corresponding Berglund-H{\"u}bsch transposes.
\end{lem}

But a crucial problem arising might be that of the shape of the potentials: at least for the case of three variables, the Berglund-H\"ubsch transposed polynomials of polynomials of Classes III and VI do not have finite dimensional Jacobian ring and hence should not be of our interest. In general, we do not know if this is the case for all the classes of polynomials which are not of invertible type.

Because of this, one may adventure herself to state that, conjecturally, only invertible polynomials are going to be interesting regarding finding orbifold equivalences via Berglund-H\"ubsch transposition, as these provide non-degenerate polynomials when Berglund-H\"ubsch transposed. It is work in progress to prove this, and then find new orbifold equivalences between potentials and their Berglund-H{\"u}bsch transpose. This preliminary work should serve as a good source for chasing new equivalences, and we hope it provides more examples or orbifold equivalences between potentials and their Berglund-H{\"u}bsch transpose soon.

\newpage
\addcontentsline{toc}{chapter}{Outlook}
\chapter*{Outlook}

There is still a great amount of work to do in order to reach a better understanding of the Landau-Ginzburg/CFT correspondence, and several immediate follow-up projects from the results presented in this thesis, which seem to be only the top of the iceberg.

To begin with, extending an equivalence between categories like the one proved in Chapter \ref{LGCFTTL} to the case where $d$ is even is an inmediate follow-up project. In addition, recall that for this result we only considered the bosonic part of the $N=2$ superconformal algebra -- but we do not have an intuition of what should the corresponding equivalence be for the fermionic part. In this paper we focused exclusively on the $N=2$ minimal models, but one may wish to explore other models like e.g. the Kazama-Suzuki models, where in \cite{Behr:2014bta} matrix factorizations were already detected.

Concerning Chapter \ref{ch:orbeq}, we hope to find more orbifold equivalences, hopefully via the Berglund-H\"ubsch transposition (as explained in Section \ref{sec:BH}). In addition, it would be interesting to analyse further orbifold equivalences in the framework of quivers (as explained in Remark 5.2.8.\ref{quiver}). Further higher categorical aspects of matrix factorizations should also be an object of future study.

In addition, on a more representation-theoretical side, it is still missing some work relating maximal Cohen-Macaulay modules and representations of vertex operator algebras. 
On first glance, one might not realize how deeply connected these structures seem to be. Light shed by the Landau-Ginzburg/CFT correspondence on their relationship may be useful for the representation theory community.

E.g. understanding the orbifold equivalence in terms of Cohen-Macaulay modules may also give some new insights into this area of mathematics.

It is the intention of the author to continue in this line of research --involving mathematics and physics in such a beautiful way-- for a long time, hopefully giving more and many results soon. In the end, this is only the beginning.

\quad

\quad

\begin{appendices}

\chapter{Proof of Theorem \ref{thm:Pdgr-tensor-closed}}\label{app:graded-tensor}

Semi-simplicity of $\Pdgr$ follows from Lemma \ref{lem:hmfgrPSPR}, as does the list of simple objects.

\medskip

Let $\lambda,\mu \in \{0,1,\dots,d{-}2\}$ and $a,b \in \bZ_d$.
To show the decomposition rule 
\beq\label{eq:P-hat-decomp}
	\hat P_{a:\lambda}  \otimes \hat P_{b:\mu} \simeq \bigoplus_{\nu=|\lambda-\mu|~\mathrm{step}\,2}^{\min(\lambda+\nu,2d-4-\lambda-\nu)} \hat P_{a+b-\frac12(\lambda+\mu-\nu):\nu} \quad ,
\eeq
we verify the cases $\lambda=0$ and $\lambda=1$ explicitly. The general case follows by a standard argument using induction on $\lambda$.

\medskip
\noindent
{\em Case $\lambda=0$:} The isomorphism $\hat P_{a:0}  \otimes \hat P_{b:\mu} \simeq \hat P_{a+b:\mu}$ is immediate from the isomorphism $\hat P_{a:0} \simeq {}_{-a}I$ given in \eqref{eq:PS-with-twisted-action-iso-to-PS}, the isomorphism ${}_{-a}I \otimes \hat M \to {}_{-a}M$ provided by ${}_{-a}(\lambda_M)$ for any matrix factorization $M$, and ${}_{-a}(P_{S}) \simeq P_{S+a}$, again from \eqref{eq:PS-with-twisted-action-iso-to-PS}.

\medskip
\noindent
{\em Case $\lambda=1$:}
For $\mu=0$ the isomorphism $\hat P_{a:1} \otimes \hat P_{b:0} \simeq P_{a+b:1}$ constructed as in case $\lambda=0$, using $P_{b:0} \simeq I_{-b}$. To show the decomposition \eqref{eq:P-hat-decomp} for $\mu \in \{1,2,\dots,d-2 \}$
we start by giving maps
$$
	g^- \,:\, \hat P_{a+b+1:\mu-1} \longrightarrow \hat P_{a:1} \otimes \hat P_{b:\mu}
	\quad , \quad
	g^+ \,:\, \hat P_{a+b:\mu+1} \longrightarrow \hat P_{a:1} \otimes \hat P_{b:\mu}
$$
in $\ZMFbigr$. Write $A = \hat P_{a:1}$, $B = P_{b:\mu}$, 
$Q_- = \hat P_{a+b+1:\mu-1}$ and $Q_+ = P_{a+b:\mu+1}$. We have to find $g^\varepsilon_{ij}$ that fit into the diagram
\beq\label{eq:P1Plam-embed-summands}
\xymatrix{
Q_\ve \ar[ddd] & & \bC[x,z]\left\{\frac{\mu+2+\ve}{d}-1\right\} \ar@/^10pt/[rrrr]^{q_\ve(x,z)} \ar[ddd]_{\left(\begin{array}{c} \scriptstyle{g_{10}^\ve} \\ \scriptstyle{g_{01}^\ve} \end{array}\right)} &&&& \bC[x,z]\left\{-\frac{\mu+2}{d}\right\} \ar@/^10pt/[llll]^{\frac{x^d-z^d}{q_\ve(x,z)}} \ar[ddd]^{\left(\begin{array}{r} \scriptstyle{g_{00}^\ve} \\ \scriptstyle{g_{11}^\ve} \end{array}\right)} \\ \\ \\
A\ot B && {\begin{array}{c}{\bC[x,y,z]\left\{\frac{3-\mu}{d}-1\right\}} \\ {\op} \\ {\bC[x,y,z]\left\{\frac{1+\mu}{d}-1\right\}} \end{array} }\ar@/^10pt/[rrrr]^{\left(\begin{array}{rr} \scriptstyle{p_1(x,y)} & \scriptstyle{p_\mu(y,z)} \\ \frac{y^d-z^d}{p_\mu(y,z)} &  \frac{x^d-y^d}{p_1(x,y)} \end{array}\right)} &&&& 
{\begin{array}{c}{\bC[x,y,z]\left\{-\frac{\mu+1}{d}\right\}} \\ {\op} \\ {\bC[x,y,z]\left\{\frac{5+\mu}{d}-2\right\}} \end{array} }
\ar@/^10pt/[llll]^{\left(\begin{array}{rr} \frac{x^d-y^d}{p_1(x,y)} & \scriptstyle{-p_\mu(z,y)} \\ \frac{y^d-z^d}{p_\mu(y,z)} & \scriptstyle{p_1(x,y)} \end{array}\right)}
}
\eeq
Here,
\begin{align*}
	p_1(x,y) &= (x-\eta^{a}y)(x-\eta^{a+1}y) ~,&
	q_-(x,z) &= \prod_{j=a+b+1}^{a+b+\mu} (x-\eta^j z) ~,
	\nonumber \\
	p_\mu(y,z) &= \prod_{j=b}^{b+\mu} (y-\eta^j z) ~, &
	q_+(x,z) &= \prod_{j=a+b}^{a+b+\mu+1} (x-\eta^j z) \ .
\end{align*}
Comparing $\bC$-degrees determines the polynomial degrees of the individual maps to be
\begin{align*}
\mathrm{deg}(g^\varepsilon_{10}) &= \mu-\tfrac12(1-\varepsilon)  ~,& 
\mathrm{deg}(g^\varepsilon_{00}) &= \tfrac12(1-\varepsilon) ~, 
\nonumber\\
\mathrm{deg}(g^\varepsilon_{01}) &= \tfrac12(1+\varepsilon) ~,&
\mathrm{deg}(g^\varepsilon_{11}) &= d-2 - \mu  -\tfrac12(1+\varepsilon) \ . 
\end{align*}
Commutativity of \eqref{eq:P1Plam-embed-summands} is equivalent to
\begin{align*}
&\text{(i)}& q_\varepsilon(x,z) \, g^\varepsilon_{00}(x,y,z) &= p_1(x,y) \, g^\varepsilon_{10}(x,y,z) + p_\mu(y,z) \, g^\varepsilon_{01}(x,y,z)
\nonumber \\
&\text{(ii)}& q_\varepsilon(x,z) \, g^\varepsilon_{11}(x,y,z) &= -\frac{y^d-z^d}{p_\mu(y,z)} \, g^\varepsilon_{10}(x,y,z) + \frac{x^d-y^d}{p_1(x,y)} \, g^\varepsilon_{01}(x,y,z)
\end{align*}
These conditions imply the remaining two conditions.
Let us show how one arrives at $g^-$ in some detail and then just state the result for $g^+$. 

\medskip

We have $\mathrm{deg}(g^-_{01})=0$ and we make the ansatz $g^-_{01}=1$ (choosing $g^-_{01}=0$ forces $g^-=0$, so this is really a normalisation condition).
The polynomial $g^-_{00}$ is of degree 1, so $g^-_{00}(x,y,z) = \alpha x + \beta y + \gamma z$ for some $\alpha,\beta,\gamma \in \bC$. Condition (i) determines $g^-_{10}$ uniquely to be
$$
	g^-_{10}(x,y,z) = \frac{q_-(x,z)(\alpha x + \beta y + \gamma z) - p_\mu(y,z)}{p_1(x,y)} \ .
$$
We need to impose the condition that $g^-_{10}$ is a polynomial. 
This amounts to verifying that the numerator has zeros for $y=\eta^{-a}x$ and $y=\eta^{-a-1}x$.
Using 
\begin{align*}
p_\mu(\mu^{-a}x,z) &=  q_-(x,z) \, \eta^{-a(\mu+1)} (x-\eta^{a+b}z) \ ,
\nonumber \\
p_\mu(\mu^{-a-1}x,z) &= q_-(x,z) \, \eta^{-(a+1)(\mu+1)}  (x-\eta^{a+b+\mu+1}z)
\end{align*}
gives the unique solution
$$
	g^-_{00}(x,y,z) = \eta^{-a\mu} \Bigg( - \eta^{-a-1} \frac{1-\eta^{-\mu}}{1-\eta^{-1}} \, x
	+ \frac{1 -\eta^{-\mu-1}}{1-\eta^{-1}} \, y
	- \eta^{b} \,z \Bigg) \ .
$$
Finally, a short calculation shows that condition (ii) is equivalent to
$$
	g_{11}^- = \frac{1}{p_1(x,y)}\Bigg( \frac{x^d-z^d}{q_-(x,z)} - \frac{y^d-z^d}{p_\mu(y,z)} \, g^-_{00}(x,y,z) \Bigg) \ .
$$
The term in brackets is clearly a polynomial. To show that it is divisible by $p_1(x,y)$, one simply verifies that the term is brackets is zero for $y = \eta^{-a}x$ and $y = \eta^{-a-1}x$.

For $g^+$ the calculation works along the same lines with the result
\begin{align*}
	g^+_{00} &= 1  
	\nonumber \\
	g^+_{01} &= 
		\eta^{a (\mu+1)  } \Bigg( \frac{1-\eta^{\mu+2}}{1-\eta} \, x 
		- \eta^{a+1} \frac{1 - \eta^{\mu+1}}{1-\eta} \, y - \eta^{a+b+\mu+1} \, z \Bigg)
	\nonumber \\
	g^+_{10} &= \frac{q_+(x,z) - p_\mu(y,z) \, g_{01}^+(x,y,z)}{p_1(x,y)}
	\nonumber \\
	g^+_{11} &= \frac{1}{p_1(x,y)}\Bigg( \frac{x^d-z^d}{q_+(x,z)} \, g^+_{01}(x,y,z) - \frac{y^d-z^d}{p_\mu(y,z)} \Bigg)
\end{align*}
As above, one verifies that the $g^+_{10}$ and $g^+_{11}$ are indeed polynomials in $x,y,z$.

\medskip

We will now establish that $(g^-,g^+) : \hat P_{a+b+1:\mu-1} \oplus \hat P_{a+b:\mu+1} \longrightarrow \hat P_{a:1} \otimes \hat P_{b:\mu}$ is an isomorphism in $\HMFbi$ (and thereby also in $\HMFbigr$ as $g^\pm$ have $\bC$-degree 0). We do this by employing Remark \ref{rem:H(M)} (see \cite[Cor.\,4.9]{Wu09} for the corresponding graded statement), that is, by showing that $(H(g^-),H(g^+)) : H(\hat P_{a+b+1:\mu-1}) \oplus H(\hat P_{a+b:\mu+1}) \longrightarrow H(\hat P_{a:1} \otimes \hat P_{b:\mu})$ is an isomorphism.

For $\hat P_S$ we have $H(\hat P_\emptyset) = H(\hat P_{\bZ_d}) = 0$ and
$H(\hat P_S) = \bC \oplus \bC$ if $S \neq \emptyset, \bZ_d$. The first case occurs only for $\mu=d-2$, where $H(\hat P_{a+b:\mu+1})=0$.

For $H(\hat P_{a:1} \otimes \hat P_{b:\mu})$ we need to compute the homology of the complex
$$
\xymatrix{
{\begin{array}{c}\bC[y] \\ {\op} \\ \bC[y] \end{array} }
 \ar@/^10pt/[rrrr]^{\left(\begin{array}{rr} \scriptstyle{\eta^{2a+1}y^2} & \scriptstyle{y^{\mu+1}} \\ \scriptstyle{-y^{d-\mu-1}} &  \scriptstyle{-\eta^{-2a-1}y^{d-2}} \end{array}\right)} &&&& 
{\begin{array}{c}\bC[y] \\ {\op} \\ \bC[y] \end{array} }
\ar@/^10pt/[llll]^{\left(\begin{array}{lr} \scriptstyle{-\eta^{-2a-1}y^{d-2}} & \scriptstyle{-y^{\mu+1}} \\ \scriptstyle{y^{d-\mu-1}} &  \scriptstyle{\eta^{2a+1}y^{2}} \end{array}\right)}
}
$$
Define the vectors $v_0 := (\eta^{2a+1} , - y^{d-\mu-3})$ (this has second entry equal to $y^{-1}$ for $\mu=d-2$, but the results below are polynomial nonetheless) and $v_1 = (-y^{\mu-1} , \eta^{2a+1})$. 
One finds
\begin{align*}
	\mathrm{ker}(\bar d_0) &= \bC[y] v_0 \cdot \begin{cases} 1 &; \mu < d-2 \\ y &; \mu = d-2 \end{cases}
	&
	\mathrm{ker}(\bar d_1) &=  \bC[y] v_1
\nonumber	\\
	\mathrm{im}(\bar d_1) &= y^2 \bC[y] v_0
&
	\mathrm{im}(\bar d_0) &= y \bC[y] v_1\cdot \begin{cases} 1 &; \mu < d-2 \\ y &; \mu = d-2 \end{cases}
\end{align*}
Writing $[\cdots]$ for the homology classes, the homology groups 
$H_i := H_i(\hat P_{a:1} \otimes \hat P_{b:\mu})$
are given by
$$
	H_0 = 
	\begin{cases} \{ [v_0] , [yv_0] \} &; \mu < d-2 \\ \{ [yv_0] \} &; \mu = d-2 \end{cases}
	\quad , \qquad
	H_1 = 
	\begin{cases} \{ [v_1] , [yv_1] \} &; \mu < d-2 \\ \{ [v_1] \} &; \mu = d-2 \end{cases}
	\qquad .
$$
The map $(g^-,g^+)$ acts on homology by, for $\mu < d-2$,
\begin{align*}
	H_0(g^-,g^+) &=~
	\Big(~~ \eta^{-2a-1} \beta^- \, [y v_0] ~~,~~ \eta^{-2a-1} [v_0]~~ \Big) \ ,
	\nonumber \\
	H_1(g^-,g^+) &=~ 	\Big( ~~\eta^{-2a-1} [v_1] ~~,~~ \eta^{-2a-1} \beta^+ \, [y v_1]~~ \Big) \ .
\end{align*}
Here $\beta^-$ is the coefficient of $y$ in $g^-_{00}$ and $\beta^+$ is the coefficient of $y$ in $g^+_{01}$.
For $\mu = d-2$, the second entry in the above maps is absent, as $H(\hat P_{a+b:\mu+1})=0$ in this case. Altogether we see that $H(g^-,g^+)$ is indeed an isomorphism.

\medskip

This proves the decomposition 
$\hat P_{a:1} \otimes \hat P_{b:\mu} ~\simeq~
\hat P_{a+b+1:\mu-1} \oplus \hat P_{a+b:\mu+1}$ 
in $\HMFbigr$ and
completes the proof of Theorem \ref{thm:Pdgr-tensor-closed}.

\chapter{Equivariant objects and
pointed categories}\label{sec:equiv+pointed}

Here we collect some (well known) categorical trivialities which allow us to avoid difficult calculations with matrix bifactorizations. 
Throughout Appendix \ref{sec:equiv+pointed} all tensor (and in particular all fusion) categories will be assumed to be strict. 
In labels for some arrows in our diagrams we suppress tensor product symbols for compactness. 

\section{Categories of equivariant objects}\lb{ceo}

Let $G$ be a group. 
An {\em action} of $G$ on a tensor category $\C$ is a monoidal functor $F:\underline{G}\to\Aut_\ot(\C)$ from the discrete monoidal category $\underline{G}$ to the groupoid of tensor autoequivalences of $\C$.
More explicitly, a $G$-action on $\C$ consists of a collection $\{F_g\}_{g\in G}$ of tensor autoequivalences $F_g:\C\to\C$ labelled by elements of $G$ together with natural isomorphisms $\phi_{f,g}:F_f\circ F_g\to F_{fg}$ of tensor functors such that $\phi_{f,e} = 1,\phi_{e,g} =1$ and such that the diagram 
$$\xymatrix{F_f\circ F_g\circ F_h\ar[rr]^{\phi_{f,g}\circ 1} \ar[d]_{1\circ\phi_{g,h}} && F_{fg}\circ F_g \ar[d]^{\phi_{fg,h}} \\ 
F_f\circ F_{gh} \ar[rr]^{\phi_{f,gh}} && F_{fgh}}$$
commutes for any $f,g,h\in G$.

Let $\C$ be a tensor category together with a $G$-action.
An object $X\in\C$ is {\em $G$-equivariant} if it comes equipped with a collection of isomorphisms $x_g:X\to F_g(X)$ such that 
the diagram 
$$\xymatrix{X\ar[rr]^{x_{fg}} \ar[d]_{x_{f}} && F_{fg}(X) \\ 
F_f(X) \ar[rr]_{F_f(x_g)} && F_f(F_g(X)) \ar[u]_{\phi_{f,g}} }$$
commutes for any $f,g\in G$.

A morphism $a:X\to Y$ between $G$-equivariant objects $(X,x_g), (Y,y_g)$ is {\em $G$-equivariant} if 
the diagram 
$$\xymatrix{X\ar[rr]^{x_g} \ar[d]_{a} && F_{g}(X) \ar[d]^{F_g(a)}\\ 
Y \ar[rr]^{y_g} && F_g(Y) }$$
commutes for any $g\in G$.
Denote by $\C^G$ the category of $G$-equivariant objects in $\C$.

\bpr\label{prop:CG-is-tensor}
Let $\C$ be a strict tensor category with a $G$-action. 
Then the category $\C^G$ is strict tensor with 

tensor product $(X,x_g)\ot(Y,y_g) = (X\ot Y,(x|y)_g)$, where $(x|y)_g$ is defined by
$$\xymatrix{X\ot Y \ar[rr]^(.4){x_gy_g} && F_g(X)\ot F_g(Y) \ar[rr]^(.55){(F_g)_{X,Y}} && F_g(X\ot Y)}$$
and with

unit object $(I,\iota)$, where $\iota_g:I\to F_g(I)$ is the unit isomorphism of the tensor functor $F_g$.
\epr
\bpf
All we need to check is that the $G$-equivariant structures of the triple tensor products $(X,x_g)\ot((Y,y_g)\ot(Z,z_g))$ and $((X,x_g)\ot(Y,y_g))\ot(Z,z_g)$ coincide. These $G$-equivariant structures $x|(y|z), (x|y)|z$ are the top and the bottom outer paths of the diagram
$$\xymatrix{X\ot Y\ot Z\ar[rrd]^{x_gy_gz_g} \\ && F_g(X)\ot F_g(Y)\ot F_g(Z) \ar[rr]^{(F_g)_{X,Y}\id} \ar[d]_{\id(F_g)_{Y,Z}} && F_g(X\ot Y)\ot F_g(Y) \ar[d]^{(F_g)_{XY,Z}} \\ &&
F_g(X)\ot F_g(Y\ot Z) \ar[rr]^{(F_g)_{X,YZ}} && F_g(X\ot Y\ot Z)}$$
whose commutativity is the coherence of the tensor structure of $F_g$.
\epf

Clearly the forgetful functor
$$\C^G\ \to\ \C,\qquad (X,x)\mapsto X$$ is tensor.

\bre\lb{rem2}
It is possible to define more general $G$-actions on (tensor) categories involving associators for $G$ (3-cocycles for $G$).
All constructions generalise straightforwardly.
\ere

\section{Inner actions and monoidal centralisers of pointed subcategories}\lb{iac}

An object $P$ of a tensor category $\C$ is {\em invertible} if the dual object $P^*$ exists and the evaluation $\mathrm{ev}_P:P^*\ot P\to I$ and coevaluation $\mathrm{coev}_P:I\to P\ot P^*$ maps are isomorphisms. Clearly an invertible object is simple since $\C(P,P)\simeq\C(I,I)=\kk$.

The set $Pic(\C)$ of isomorphism classes of invertible objects is a group with respect to the tensor product (the {\em Picard group} of $\C$).
Choosing a representative $s(p)$ in each isomorphism class $p\in Pic(\C)$ and isomorphisms $\sigma(p,q):s(p)\ot s(q)\to s(pq)$ for each pair $p,q\in Pic(\C)$ allows us to define a function $\alpha:Pic(\C)^{\times 3}\to \kk^*$ (here $\kk^*$ is the multiplicative group of non-zero elements of $\kk$). Indeed for $p,q,r\in Pic(\C)$ the composition 
$$
s(pqr) \xrightarrow{\sigma(pq,r)^{-1}}  s(pq)\ot s(r) \xrightarrow{\sigma(p,q)^{-1}\,\id}  s(p)\ot s(q)\ot s(r) \xrightarrow{\id\,\sigma(q,r)}  s(p)\ot s(qr) \xrightarrow{\sigma(p,qr)}  s(pqr)$$
is an automorphism of $s(pqr)$ and thus has a form $\alpha(p,q,r)\id_{s(pqr)}$ for some $\alpha(p,q,r)\in \kk^*$.
It is easy to see that $\alpha$ is a 3-cocycle and that the class $[\alpha]\in H^3(Pic(\C),\kk^*)$ does not depend on the choice of $s$ and $\sigma$. 

A tensor category $\C$ is {\em pointed} if all its simple objects are invertible. 
A fusion pointed category $\C$ can be identified with the category $\V(G,\alpha)$ of $G$-graded vector spaces, where $G=Pic(\C)$ and with the associativity constraint twisted by $\alpha\in H^3(Pic(\C),\kk^*)$.

Let $P$ be an invertible object of a tensor category $\C$.
The functor
$$P\ot-\ot P^*:\C\to \C,\qquad X\mapsto P\ot X\ot P^*$$ comes equipped with a monoidal structure 
$$\xymatrix{P\ot X\ot P^*\ot P\ot Y\ot P^* \ar[rr]^(.6){\id\,\mathrm{ev}_P\,\id} && P\ot X\ot Y\ot P^*}$$
making it a tensor autoequivalence, the {\em inner autoequivalence} corresponding to $P$.
The assignment $P\mapsto P\ot-\ot P^*$ defines a homomorphism of groups
$Pic(\C)\ \to\ Aut_\ot(\C)\ .$

The {\em monoidal centraliser} $\Z_\D(F)$ of a tensor functor $\C\to\D$ is the category of pairs $(Z,z)$, where $Z\in\D$ and $z_X:Z\ot F(X)\to F(X)\ot Z$ are a collection of isomorphisms, natural in $X\in\C$, such that $Z_I=\id$ and such that the diagram
$$\xymatrix{Z\ot F(X\ot Y) \ar[d]_{\id\,F_{X,Y}} \ar[rr]^{z_{XY}} && F(X\ot Y)\ot Z \ar[d]^{F_{X,Y}\,\id} \\ Z\ot F(X)\ot F(Y)\ar[dr]_{z_X\,\id} && F(X)\ot F(Y)\ot Z\\ & F(X)\ot Z\ot F(Y) \ar[ru]_{\id\,z_Y} }$$
commutes for any $X,Y\in\C$.
A morphism $(Z,z)\to(Z',z')$ in 

	$\Z_\D(F)$ 
is a morphism $f:Z\to Z'$ in $\D$ such that the diagram
$$\xymatrix{Z\ot F(X)\ar[rr]^{z_X} \ar[d]_{f\,\id} && F(X)\ot Z \ar[d]^{\id\,f} \\ Z'\ot F(X)\ar[rr]^{z'_X} && F(X)\ot Z'}$$
commutes for any $X\in\C$.

\bpr
Let $F:\C\to\D$ be a tensor functor between strict tensor categories. 
Then the monoidal centraliser $\Z_\D(F)$ is strict tensor with the tensor product $(Z,z)\ot(Z',z') = (Z\ot Z',z|z')$ where $(z|z')_X$ is defined by
$$\xymatrix{Z\ot Z'\ot F(X)\ar[rr]^{(z|z')_{X}} \ar[dr]_{\id\,z'_X} &&
	F(X) \ot Z\ot Z' 
 \\ & Z\ot F(X)\ot Z' \ar[ru]_{z_X\,\id} }$$
and with the unit object $(I,1)$.
\epr
\bpf
Note that the monoidal centraliser $\Z_\D(Id_\D)$ of the identity functor $Id_\D:\D\to\D$ is the monoidal centre $\Z(\D)$. 
The proof of the proposition is identical to the proof of monoidality of the monoidal centre (see \cite{js0}).
\epf

Clearly the forgetful functor
$$\Z_\D(F)\ \to\ \D,\qquad (Z,z)\mapsto Z$$ is tensor.

Let $G$ be a group and $\V(G)$ be the pointed tensor category whose group of isomorphism classes of objects is $G$ and which has trivial associator.
A tensor functor $F:\V(G)\to\C$ gives rise to the action of $G$ on $\C$ by inner autoequivalences $F_g(X) = F(g)\ot X\ot F(g)^*,\ g\in G$.

\bth\lb{cps}
Let $G$ be a group and let $\C$ be a tensor category with a tensor functor $F:\V(G)\to\C$. 
Then the monoidal centraliser $\Z_\C(F)$ is tensor equivalent to the category of $G$-equivariant objects $\C^G$, where the $G$-action is defined by the functor $\V(G)\to\C$ as above.
\eth
\bpf
Define a functor $\Z_\C(F)\to \C^G$ by assigning to $(Z,z)\in\Z_\C(F)$ a $G$-equivariant object $(Z,\tilde z_g)_{g\in G}$ with $\tilde z_g:Z\to F_g(X) = F(g)\ot X\ot F(g)^*$ given by
$$\xymatrix{Z \ar[rr]^-{\id\,\mathrm{coev}_{F(g)}} && Z\ot F(g)\ot F(g)^*\ar[rr]^{z_g\,\id} && F(g)\ot Z\ot F(g)^*}\quad .$$
It is straightforward to see that this is a tensor equivalence.
\epf

\section{Tensor functors from products with pointed categories}

Recall from
	\cite{deligne,BaKiBook}
that the {\em Deligne product} $\C\boxtimes\D$ of $\kk$-linear semi-simple categories $\C$ and $\D$ is a semi-simple category with simple objects $X\boxtimes Y$ for $X$ and $Y$ being simple objects of $\C$ and $\D$ correspondingly.
One can extend the definition of $X\boxtimes Y$ to arbitrary $X\in\C$ and $Y\in\D$. The hom spaces between these objects are
$$(\C\boxtimes\D)(X\boxtimes Y,X'\boxtimes Y') = \C(X,X')\ot_\kk \D(Y,Y')\ ,$$
where on the right is the tensor product of vector spaces over $\kk$.

The Deligne product of fusion categories is fusion with the unit object $I\boxtimes I$ and the tensor product defined by
$$(X\boxtimes Y)\ot(X'\boxtimes Y') = (X\ot X')\boxtimes (Y\ot Y')\ .$$
The Deligne product of fusion categories has another universal property, which we describe next.

We say that a pair of tensor functors $F_i:\C_i\to\D$ has {\em commuting images} if they come equipped with a collection of isomorphisms
$c_{X_1,X_2}:F_1(X_1)\ot F_2(X_2)\to F_2(X_2)\ot F_1(X_1)$ natural in $X_i\in\C_i$ and such that the following diagrams commute for all $X_i,Y_i\in \C_i$:
$${\xymatrixcolsep{1pc}\xymatrix{F_1(X_1)\ot F_2(I)\ar[rr]^{c_{X_1,I}} \ar[d] && F_2(I)\ot F_1(X_1) \ar[d] \\ F_1(X_1)\ot I\ar@{=}[r] & F_1(X_1) \ar@{=}[r] & I\ot F_1(X_1)}}\qquad 
{\xymatrixcolsep{1pc}\xymatrix{F_1(I)\ot F_2(X_2)\ar[rr]^{c_{I,X_2}} \ar[d] && F_2(X_2)\ot F_1(I) \ar[d] \\ I\ot F_2(X_2)\ar@{=}[r] & F_2(X_2) \ar@{=}[r] & F_2(X_2)\ot I}}$$
$$\xymatrix{F_1(X_1\ot Y_1)\ot F_2(X_2) \ar[rr]^{c_{X_1Y_1,X_2}} \ar[d]_{(F_1)_{X_1,Y_1}1} && F_2(X_2)\ot F_1(X_1\ot Y_1) \ar[d]^{1(F_1)_{X_1,Y_1}} \\ F_1(X_1)\ot F_1(Y_1)\ot F_2(X_2)  \ar[rd]_{1c_{Y_1,X_2}} && F_2(X_2)\ot F_1(X_1)\ot F_1(Y_1) \\ & F_1(X_1)\ot F_2(X_2)\ot F_1(Y_1)  \ar[ru]_{c_{X_1,X_2}1}}$$
$$\xymatrix{F_1(X_1)\ot F_2(X_2\ot Y_2) \ar[rr]^{c_{X_1,X_2Y_2}} \ar[d]_{1(F_2)_{X_2,Y_2}} && F_2(X_2\ot Y_2)\ot F_1(X_1) \ar[d]^{(F_2)_{X_2,Y_2}1} \\ F_1(X_1)\ot F_2(X_2)\ot F_2(Y_2)  \ar[rd]_{c_{X_1,X_2}1} && F_2(X_2)\ot F_2(Y_2)\ot F_1(X_1) \\ & F_2(X_2)\ot F_1(X_1)\ot F_2(Y_2)  \ar[ru]_{1c_{X_1,Y_2}}}$$

\bpr\lb{upd}
The Deligne product $\C_1\boxtimes\C_2$ of fusion categories $\C_1$ and $\C_2$ is the initial object among pairs of tensor functors $F_i:\C_i\to\D$ with commuting images, that is for a pair of tensor functors $F_i:\C_i\to\D$ with commuting images there is a 
	unique
tensor functor $F:\C_1\boxtimes\C_2\to\D$ making the diagram 
$$\xymatrix{\C_1 \ar[rd] \ar[rdd]_{F_1} && \C_2 \ar[ld] \ar[ldd]^{F_2} \\ & \C_1\boxtimes\C_2 \ar@{.>}[d]^F \\ & \D}$$
commutative.
\epr
\bpf
Note that the assignments $X_1\mapsto X_1\boxtimes I,\ X_2\mapsto I\boxtimes X_2$ define a pair of tensor functors $\C_i\to C_1\boxtimes\C_2$ with commuting images.

Conversely let $F_i:\C_i\to\D$ be a pair of tensor functors with commuting images.
Define $F:\C_1\boxtimes\C_2\to\D$ by $F(X_1\boxtimes X_2) = F_1(X_1)\ot F_2(X_2)$. 
	Since $\C_1$ and $\C_2$ are fusion, this determines $F$ uniquely as a $\kk$-linear functor.
	The monoidal structure for $F$ is uniquely determined to be 
$$\xymatrix{
F(X_1\boxtimes X_2)\ot F(Y_1\boxtimes Y_2) \ar@{=}[d] \ar[rrr]^{F_{X_1\boxtimes X_2,Y_1\boxtimes Y_2}} &&& F\big((X_1\boxtimes X_2)\ot(Y_1\boxtimes Y_2)\big) \ar@{=}[d] \\
F(X_1)\ot F(X_2)\ot F(Y_1)\ot F(Y_2) \ar[d]_{1c_{X_2,Y_1}1} &&& F\big((X_1\ot Y_1)\boxtimes(X_2\ot Y_2)\big) \ar@{=}[d] \\
F(X_1)\ot F(Y_1)\ot F(X_2)\ot F(Y_2) \ar[rrr]^{(F_1)_{X_1,Y_1}(F_2)_{X_2,Y_2}} &&& F_1(X_1\ot Y_1)\ot F_2(X_2\ot Y_2) }$$
It is straightforward to check that this definition satisfies the coherence axioms of a monoidal structure. 
\epf

\bre\lb{ciz}
Note that the data of a pair of tensor functors $F_i:\C_i\to\D$ with commuting images amounts to a tensor functor $\C_1\to \Z_\D(F_2)$ whose composition with the forgetful functor $\Z_\D(F_2)\to\D$ equals $F_1$.
\ere

\bth\lb{fpp}
Let $\C$ be a fusion
category and let $G$ be a finite
group.
Then the data of a tensor functor $\C\boxtimes\V(G)\to\D$ amounts to a tensor functor $\V(G)\to\D$ and a tensor functor $\C\to \D^G$, where the $G$-action is defined by the functor $\V(G)\to\D$ as in Appendix \ref{iac}.
\eth
\bpf
By Proposition \ref{upd} a tensor functor $\C\boxtimes\V(G)\to\D$ corresponds to a pair of tensor functors $F:\V(G)\to\D,\ F':\C\to\D$ with commuting images.
By Remark \ref{ciz} this is equivalent to a tensor functor $\C\to\Z_\D(F)$ to the centraliser of $F$.
Finally by Theorem \ref{cps} the centraliser $\Z_\D(F)$ is canonically equivalent to the category of equivariant objects $\D^G$.
\epf

\bre
It is possible to extend Theorem \ref{fpp} to the case of pointed categories $\V(G,\alpha)$ with non-trivial associators $\alpha\in Z^3(G,\kk^*)$.
As for Remark \ref{rem2} all constructions generalise straightforwardly.
\ere
\end{appendices}

\chapter*{Acknowledgements}

First and foremost, I would like to warmly thank my advisor, Ingo Runkel, for having given me the honour of being his student, his time, lots of patience and all the generously shared wisdom. Also, for his complete support not only on the daily research but also on applying for positions and the very careful proof-reading of this thesis. He will be sincerely missed when I leave Hamburg.

I gratefully acknowledge the Department of Algebra and Number Theory and the Research Training Group 1670 ``Mathematics inspired by string theory and quantum field theory" of the University of Hamburg for not only financial and mathematical but also the best possible of human support, specially from Jennifer Maier, Alexander Barvels, Mart\'{i}n Mombelli, Susama Agarwala, Rosona Eldred, Giovanni Bazzoni, Peter-Simon Dieterich, Lana Casselmann and Nezhla Aghaee. Thanks so much to Jeffrey Morton for uncountable great discussions on work and scientific life. Floyd, Pawel Sosna and Alessandro Valentino proof-read parts of this thesis, which would look way worse if it wasn't because of their corrections.

A mi familia por darme fuerzas para luchar por lo imposible. Esta tesis est{\'a} dedicada a mis abuelos y a Pepe, a los cuales espero haberles hecho sentir orgullosos de m{\'i}.

Without my friends I definitely could not have made it till here: Manu (por nuestras cervezacas, nuestras aventuras por Europa y todas nuestras risas juntos), Caro (entre nuestros suicidios musicales y sesiones de despotrique, la desintegraci\'{o}n por Riemann ha sido de lo m\'{a}s terap\'{e}utica y divertida), Lola y su mam\'{a} Alexandra, Ra{\'u}l i la seva Marta y Jose (gracias por ver m{\'a}s all{\'a} de la oscuridad de la noche). Susanne, sin tu ayuda yo nunca hubiera podido llegar tan lejos, y no te lo voy a poder agradecer nunca lo suficiente. Sara und Arne, meine Lieben, ihr seid geil! Floyd and Denise, life is unfair making me meet such wonderful people like you and then taking them an ocean away from me. There is many more wonderful people --that could fill several pages-- which is a shame I'm neglecting in this list, but that I love and thank you the same (and I hope guys you know that).

Vielen lieben Dank auch an die Ninjutsu Akademie Hamburg, nicht nur f\"ur das tolle Training, sondern insbesondere auch f{\"u}r die vielen gemeinsamen Stunden voller Lachen.

Finally, I also owe a big thank you to Armin van Buuren, Herman-Jan Wildervank, Willem van Hanegem and Wardt van der Harst, Dimitri and Michael Thivaios, Ferry Corsten, Orjan Nilsen, Brian Wayne Transeau, Sander Ketelaars, Alexander Popov and Dmitry Almazov for providing an amazing soundtrack to this work and inspiring many ideas. Without your music, guys, this journey would have never been so epic.

\newpage

\newcommand\arxiv[2]      {\href{http://arXiv.org/abs/#1}{#2}}
\newcommand\doi[2]        {\href{http://dx.doi.org/#1}{#2}}
\newcommand\httpurl[2]    {\href{http://#1}{#2}}

\newpage

\thispagestyle{empty}
\chapter*{}

\section*{Summary}

In this thesis we focus on the Landau-Ginzburg/conformal field theory correspondence, a correspondence which in particular predicts some relation between defects of Landau-Ginzburg models --described by matrix factorizations-- and defects in conformal field theories -- described by representations of vertex operator algebras. Matrix factorizations are the most important concept of this thesis and our main tool. With them, we describe several results conjectured in the physics literature. 

On the first chapter we introduce some categorical background necessary for the later chapters. On the second, we introduce matrix factorizations and on the third describe the Landau-Ginzburg/conformal field theory correspondence (which up to date has no clear mathematical conjecture).

Chapter 4 deals with the first result of this thesis. We describe a tensor equivalence between a category generated via direct sums of permutation-type matrix factorizations with consecutive sets with morphisms of $\mathbbm{C}$-degree zero and the subcategory of representations of the vertex operator algebra associated to the coset $\frac{\hat{\mathfrak{su}}(2)_{d-2} \oplus \hat{\mathfrak{u}}(1)_{4}}{\hat{\mathfrak{u}}(1)_{2d}}$ generated by direct sums of simples (which are of the shape $\left[ l,m,s \right]$ satisfying that $l+m \in 2 \mathbb{Z}$ and $s=0$. We prove this tensor equivalence with the help of Temperley-Lieb categories.

Chapter 5 deals with the second result of this thesis. We prove several equivalences of categories involving Landau-Ginzburg models described by simple singularities (which have an ADE classification). We support this result on the theory of orbifold completion developed by Carqueville and Runkel, where an equivalence relation between potentials was described. We also include some remarks about a possible strategy to find more of these equivalence relations.

\newpage

\quad

\quad
\quad

\quad

\thispagestyle{empty}
\section*{Zusammenfassung}

In dieser Dissertation betrachten wir die Landau-Ginzburg/konformale Feldtheorie Korrespondenz, die eine Beziehung zwischen Defekten in Landau-Ginzburg Modellen (beschrieben durch Matrixfaktorisierungen) und Defekten in konformalen Feldtheorien
(beschrieben durch Darstellungen von vertex operator Algebren) vorhersagt. Matrixfaktorisierungen sind die Hauptkonzepte in dieser Arbeit und unser allgemeines Werkzeug. Mit diesem beschreiben wir mehrere in der Physikliteratur vorhergesagte Ergebnisse.

Im ersten Kapitelbeschreiben wir einigen Grundlageder Kategorientheorie, die man f\"ur die sp\"atere Kapitel braucht. Im zweiten Kapitel f\"uhren wir Matrixfaktorisierungen ein und im dritten die Landau-Ginzburg/konformale Feldtheorie Korrespondenz (die bis heute nicht durch eine klare mathematische Vermutung beschrieben wird).

Kapitel 4 besch\"aftigt sich mit den ersten Ergebnissen dieser Arbeit. Wir beweisen eine Tensor\"aquivalenz von einer Kategorie, dargestellt durch direkte Summen von  Matrixfaktorisierungen vom Permutationstyp mit konsekutiver Menge und Morphismen mit $\mathbbm{C}$-Grad Null, und der Unterkategorie von Darstellungen von der vertex Operator Algebra, welche zu den Nebenklassen $\frac{\hat{\mathfrak{su}}(2)_{d-2} \oplus
\hat{\mathfrak{u}}(1)_{4}}{\hat{\mathfrak{u}}(1)_{2d}}$ assoziiert ist, dargestellt durch direkte Summen von einfachen Objekten (die $\left[l,m,s \right]$ aussehen), sodass $l+m \in 2 \mathbb{Z}$ und $s=0$. Wir behelfen uns mit Temperley-Lieb Kategorien, um diesen Satz zu beweisen.

Kapitel 5 besch\"aftigt sich mit den zweiten Ergebnissen dieser Arbeit. Wir beweisen mehrere Kategorien\"aquivalenzen mit Landau-Ginzburg Modellen, beschrieben durch einfache Singularit\"aten (diese haben eine ADE Klassifizierung). Wir unterst\"utzen diese Ergebnisse mit der Theorie der Vervollst\"andigung von Orbifaltigkeiten, entwickelt von Carqueville und Runkel, wobei eine\"Aquivalenzrelation zwischen Potentialen beschrieben ist. Wir enden mit einigen Bemerkungen zu einer m\"oglichen Strategie, um mehrere dieser \"Aquivalenzrelationen zu finden.

\end{document}